\documentclass[a4paper,11pt]{amsart}
\usepackage{
   amssymb
  ,amsmath
  ,stmaryrd
  ,mathrsfs
  ,enumitem
  ,hyperref
  ,leftidx
  ,mathtools
  ,tikz
  ,url
  ,newclude
  ,mathbbol
  ,mathabx
  }

\usepackage{xparse}

\def\ulangle
{{% usage \Blocks[cols]{rows}{list of shaded cells}
   \begin{tikzpicture}[scale=.1]
     \draw (0,0) -- (0,1) -- (1,1) ;
   \end{tikzpicture}}
}

\def\drangle
{{% usage \Blocks[cols]{rows}{list of shaded cells}
   \begin{tikzpicture}[scale=.1]
     \draw (0,0) -- (1,0) -- (1,1) ;
   \end{tikzpicture}}
}

\def\bigdrangle
{{% usage \Blocks[cols]{rows}{list of shaded cells}
   \begin{tikzpicture}[scale=.25]
     \draw (0,0) -- (1,0) -- (1,1) ;
   \end{tikzpicture}}
}

\def\bigulangle
{{% usage \Blocks[cols]{rows}{list of shaded cells}
   \begin{tikzpicture}[scale=.25]
     \draw (0,0) -- (0,1) -- (1,1) ;
   \end{tikzpicture}}
}
\def\bigsquare
{{% usage \Blocks[cols]{rows}{list of shaded cells}
   \begin{tikzpicture}[scale=.25]
     \draw (0,0) -- (0,1) -- (1,1) -- (1,0) -- (0,0);
   \end{tikzpicture}}
}

  \usepackage[all,cmtip]{xy}

\DeclareMathOperator{\colim}{colim}
\DeclareMathOperator{\pcolim}{pcolim}
\DeclareMathOperator{\hocolim}{hocolim}
\DeclareMathOperator{\holim}{holim}

\def\M{\mathcal{M}}
\def\Ext{\mathrm{Ext}}
\def\f{\phi}
\def\F{\mathcal{F}}

\def\FDu{\mathbb {FD}}

\def\FDm{{\mathbb F^{+}\mathbb D}}
\def\Mod{\mathrm{Mod}}
\def\End{\mathrm{End}}

\def\D{\mathbb D}
\def\E{\mathbb E}
\def\sfD{\mathbf D}
\def\T{\mathbf T}
\def\C{\mathcal C}
\def\FDp{\mathbb {F}^-{\mathbb D}}
\def\FD{\mathbb {F}^{b}{\mathbb D}}
\def\H{\mathcal H}
\def\Dia{\mathrm{Dia}}
\def\B{\mathcal B}
\def\Ho{\mathrm{Ho}}

\def\CAT{\mathrm{CAT}}

\def\sU{\mathbb U}
\def\sV{\mathbb V}
\def\Z{\mathbb Z}
\def\t{\mathbb t}

\def\A{\mathcal A}
\def\dia{\mathrm{dia}}
\def\bbone{\mathbb{1}}
\def\bbtwo{\mathbb{2}}

\def\ordn{\mathbb{n}}
\def\ordk{\mathbb{k}}
\def\ordnp{\mathbb{n+1}}
\def\Ker{\mathrm{Ker}}
\def\N{\mathbb N}
\def\id{\mathrm{id}}
\def\Ch{\mathrm{Ch}}
\def\Tot{\mathrm{Tot}}
\def\fincat{\mathrm{fd.Cat}}
\def\Cat{\mathrm{Cat}}
\def\Ob{\mathrm{Ob}}
\def\Set{\mathrm{Set}}
\def\PDer{\mathrm{PDer}}
\def\Der{\mathrm{Der}}

\def\real{\mathrm{real}}
\def\tow{\mathrm{tow}}
\def\Tow{\mathrm{Tow}}

\def\tel{\mathrm{tel}}
\def\Tel{\mathrm{Tel}}

\def\gr{\mathrm{gr}}

\def\Add{\mathrm{Add}}
\def\Ab{\mathrm{Ab}}
\def\G{\mathcal{G}}
\def\W{\mathcal W}

\def\op{{\mathrm{op}}}
\def\pt{\mathrm{pt}}
\def\S{\mathbf S}
\def\Tr{\mathrm{tr}}

\def\cone{\mathrm{cone}}
\def\fib{\mathrm{fib}}
\def\Prod{\mathrm{Prod}}

\numberwithin{equation}{section}
\newtheorem{thm}{Theorem}[section]
\newtheorem{cor}[thm]{Corollary}
\newtheorem{prop}[thm]{Proposition}
\newtheorem{lem}[thm]{Lemma}

\theoremstyle{plain}
\newtheorem{defn}[thm]{Definition}

\theoremstyle{plain}
\newtheorem{rmk}[thm]{Remark}
\newtheorem{eg}[thm]{Example}

\title{Morita theory for stable derivators}
\author{Simone Virili}
\thanks{The author was supported by the Ministerio de Econom\'\i a y Competitividad of Spain
via a grant `Juan de la Cierva-formaci\'on'. He was also supported by the research projects from the Ministerio de Econom\'\i a y
Competitividad of Spain (MTM2016-77445-P) and the Fundaci\'on `S\'eneca' of Murcia (19880/GERM/15),
both with a part of FEDER funds.}
\date{\today}
\begin{document}

\begin{abstract}
We give a general construction of realization functors for $t$-structures on the base of a strong stable derivator. In particular, given such a derivator $\D$, a $t$-structure $\t=(\mathcal D^{\leq0},\mathcal D^{\geq0})$ on the triangulated category $\D(\bbone)$, and letting $\A=\mathcal D^{\leq0}\cap \mathcal D^{\geq0}$ be its heart,  we construct, under mild assumptions, a morphism of prederivators
\[
\real_\t\colon \sfD_\A\to \D
\]
where $\sfD_\A$ is the natural prederivator enhancing the derived category of $\A$. Furthermore, we give criteria for this morphism to be fully faithful and essentially surjective. If the $t$-structure $\t$ is induced by a suitably ``bounded" co/tilting object, $\real_\t$ is an equivalence. Our construction unifies and extends most of the derived co/tilting equivalences appeared in the literature in the last years.  
\end{abstract}

\maketitle

\setcounter{tocdepth}{1}
\tableofcontents

\section*{Introduction}

{\bf Tilting theory} arose in Representation Theory as an extension of the classical Morita theory (see \cite{BB80,happel1982tilted,Bo81}). In this context, the equivalences of full module categories were substituted by suitable counter-equivalences of torsion pairs, induced by a finite dimensional tilting module. These results were successively extended to categories of modules over arbitrary rings, but still considering tilting modules with strong finiteness conditions (see \cite{CF90,miyashita-tilt}). 

\smallskip
Few years after their introduction, it was noticed that (classical) tilting modules could be used to construct equivalences  between derived categories (see \cite{CPS86,H87,happel1988triangulated}). It was finally shown by Rickard~\cite{rickard1991derived} and Keller~\cite{kel94-dg} that compact tilting complexes guarantee the existence of (bounded and unbounded) derived equivalences and vice-versa, hence establishing a derived Morita theory for rings. 

\smallskip
Motivated by problems in Approximation Theory of modules, large (i.e., non-compact) counterparts of co/tilting modules and complexes over rings were introduced (see \cite{hugel2001infinitely,CGM,CT95,stovicek2014derived,hugel2015silting,wei2013semi}). However, contrary to the compact case, these non-compact objects cannot be immediately substituted to their classical counterpart to construct derived equivalences, since their endomorphism rings are not derived equivalent to the original ring. A solution to this difficulty was suggested in \cite{stovicek2014derived}: instead of considering endomorphism rings, one should consider the hearts of the
naturally associated {\bf $t$-structures}. 

\smallskip
$t$-Structures in triangulated categories were introduced by Beilinson, Bernstein and Deligne~\cite{BBD} in their study of perverse
sheaves on an algebraic or analytic variety. A $t$-structure in a triangulated category $\mathcal{D}$  is a pair of full
subcategories satisfying a suitable set of axioms (see the precise definition
in the next section) which guarantee that their intersection is an
Abelian category $\mathcal{A}$, called the heart. $t$-Structures have found applications in many branches of Mathematics, such that Algebraic Geometry, Algebraic Topology and Representation Theory of Algebras.

\smallskip
Already in the first paper about $t$-structures, Beilinson, Bernstein and Deligne, were concerned in finding ways to compare the original category $\mathcal D$ with the bounded derived category $\sfD^b(\A)$ of the heart $\A$ of a given $t$-structure $\t$. When $\mathcal D$ is the derived category of a ring $R$, they were able to construct, via the filtered derived category of $R$, a {\bf realization functor}
\[
\real_\t^b\colon \sfD^b(\A)\to \sfD^b(\Mod(R))
\]
and to give necessary and sufficient conditions for this functor to be an equivalence. This technique was then generalized by Beilinson \cite{Beilinson} with the introduction of $f$-categories, an abstraction of the concept of filtered derived category.\\
For alternative approaches to bounded realization functors see \cite{bernhard1991derived} (and its translation \cite{Porta} in the setting of stable derivators), and \cite{keller1990chain, keller1987sous}. 

\smallskip
Beilinson's abstract construction of realization functors via {\bf $f$-categories} is at the base of the recent work of Psaroudakis and Vitoria \cite{Jorge_e_Chrisostomos}, where they use realization functors to establish a derived Morita theory for Abelian categories. The picture they obtain is extremely clear but their results do not extend to unbounded derived categories, due to the lack of general techniques to construct realization functors between unbounded derived categories. The initial motivation for this paper is to extend Beilinson's techniques to unbounded derived categories and, in this way, complete the results in \cite{Jorge_e_Chrisostomos}. 

\bigskip\noindent
{\bf General setting.} 
Beilinson's $f$-categories can be viewed, informally, as ``$\Z$-filtered" triangulated categories. What we need for our construction of unbounded realization functors is an ``$\N\times \N^{\op}\times \Z$-filtered" triangulated category. It is certainly possible to find a suitable set of axioms for such gadgets but, in this paper, we have made the choice to adopt instead the setting of derivators. Recall that a strong stable {\bf derivator} is just a $2$-functor
\[
\D\colon \Cat ^{\op}\to \CAT
\]
which satisfies certain axioms, where $\Cat $ is the 2-category of small categories and $\CAT$ is the $2$-``category" of all categories. The axioms in particular imply that the natural range of this 2-functor is the $2$-``category'' of triangulated categories. One usually views $\D$ as an enhancement of the triangulated category $\D(\bbone)$ (where $\bbone$ is the one-point category), the base of $\D$, which is in some sense the minimalistic enhancement which allows for a well-behaved calculus of homotopy (co)limits or, more generally, homotopy Kan extensions (see~\cite{Moritz}). A prototypical example is the assignment $\sfD_{\G}\colon I \mapsto \sfD(\G^I)$ for a Grothendieck category $\G$.

\smallskip
Localization theory for derivators and, in the stable context, the relation between $t$-structures and derivators has recently received a lot of attention by several researchers (see \cite{GLV,tderiv,SSV,coley,laking2018purity}). Our first result in this paper is to reformulate one of the main results of~\cite{SSV} in the following form, which is analogous to a result proved by Lurie \cite[Prop.\,1.2.1.16]{Lurie_higher_algebra} in the setting of $(\infty,1)$-categories:

\medskip\noindent
{\bf Theorem A} (Theorem  \ref{t_structures_are_co_localizations}){\bf .}
{\em
Let $\D\colon \Cat^{op}\to \CAT$ be a strong and stable derivator. There are bijections among the following classes:
\begin{enumerate}
\item $t$-structures in $\D(\bbone)$; 
\item extension-closed reflective sub-derivators of $\D$;
\item extension-closed coreflective sub-derivators of $\D$.
\end{enumerate}}

\medskip
Consider now a Grothendieck category $\G$ and the natural derivator $\sfD_\G$ enhancing the derived category of $\G$. It is an exercise on the definitions to prove that the (bounded) filtered derived category of $\G$ is a suitable full subcategory of $\sfD_\G(\Z)=\sfD(\G^\Z)$. For a general strong stable derivator $\D$, the analogous subcategory of $\D(\Z)$ is an $f$-category. In view of this analogy, the $\N\times \N^{\op}\times \Z$-filtered category we need in order to construct the unbounded realization functors will be naturally defined as a full subcategory of $\D(\N\times \N^{\op}\times \Z)$, so the language of derivators is extremely appropriate for the kind of constructions we are introducing.

\bigskip\noindent
{\bf Main results.} 
Let us now briefly list the main results of this paper. The main  statement about unbounded realization functors is the following:

\medskip\noindent
{\bf Theorem B} (Theorems  \ref{heart_is_ch} and \ref{unbounded_ff_real}){\bf .}
{\em Let $\D\colon \Cat^{op}\to \CAT$ be a strong and stable derivator, let $\t=(\mathcal D^{\leq0},\mathcal D^{\geq0})$ be a $t$-structure on its base $\D(\bbone)$, and let $\A:=\mathcal D^{\leq0}\cap \mathcal D^{\geq0}$ be the heart. Suppose that $\A$ has enough injectives and it is (Ab.$4^*$)-$k$ for some $k\in\N$, then there is an exact morphism of prederivators 
\[
\real_\t\colon \sfD_\A\to \D
\]
that restricts to bounded levels. If the following conditions hold:
\begin{enumerate}
\item $\t$ is (co)effa\c cable;
\item $\t$ is non-degenerate (right and left);
\item $\t$ is $k$-cosmashing for some $k\in\N$;
\item $\t$ is ($0$-)smashing;
%\item the heart $\A$ of $\t$ has enough injectives;
\end{enumerate} 
then, $\real_\t\colon \sfD_\A\to \D$ commutes with products and coproducts and it is fully faithful.}

\medskip
Of course a dual statement holds for cosmashing $k$-smashing $t$-structures whose heart has enough projectives. The conditions appearing in the above general theorem are naturally satisfied by certain co/tilting $t$-structures. Indeed, recall from~\cite{NSZ,Jorge_e_Chrisostomos} that an object $X$ in a triangulated category is said to be {\bf tilting} provided $(T^{\perp_{>0}},T^{\perp_{<0}})$ is a $t$-structure and $\Add(T)\subseteq T^{\perp_{\neq0}}$; {\bf cotilting} objects are defined dually.
A tilting object is {\bf classical} if it is compact. 
Furthermore, two $t$-structures $\t_i=(\mathcal D_i^{\leq0},\mathcal D_i^{\geq0})$ ($i=1,2$) are said to be at {\bf finite distance} if there exists $n\in\N$ such that $\mathcal D_1^{\leq-n}\subseteq \mathcal D_2^{\leq0}\subseteq \mathcal D_1^{\leq n}$.

\medskip\noindent
{\bf Theorem C} (Theorem \ref{general_tilt_thm}){\bf .}
{\em Let $\D\colon \Cat^{\op}\to \CAT$ be a strong, stable derivator, and let  $T$ be a tilting object for which the associated $t$-structure $\t_T=(\D_T^{\leq0},\D_T^{\geq0})$ is at finite distance from a classical tilting $t$-structure. Letting $\A_{T}:=\D_{T}^{\leq0}(\bbone)\cap\D_{T}^{\geq0}(\bbone)$, there is an equivalence of prederivators
\[
\real_{\t_T}\colon \sfD_{\A_T}\to \D
\]
that restricts to bounded levels. Hence, $\sfD_{\A_T}$ is a strong, stable derivator.}

\medskip
As a corollary to the above theorem one can deduce one of the main results of~\cite{NSZ}, that is, bounded tilting sets in compactly generated algebraic triangulated categories induce exact equivalences (see Corollary \ref{coro_NSZ_tilt_alg}). Also in the dual setting of cotilting $t$-structures, we can obtain the following general statement:

\medskip\noindent
{\bf Theorem D} (Theorem \ref{general_co_tilt_thm}){\bf .}
{\em Let $\D\colon \Cat^{\op}\to \CAT$ be a strong, stable derivator, and let $C$ be a cotilting object for which the associated $t$-structure $\t_{ C}=(\D_{ C}^{\leq0},\D_{ C}^{\geq0})$ is at finite distance from a classical tilting $t$-structure. Letting $\A_{ C}:=\D_{ C}^{\leq0}(\bbone)\cap\D_{ C}^{\geq0}(\bbone)$, there is an equivalence of prederivators
\[
\real_{\t_{ C}}\colon \sfD_{\A_{ C}}\to \D
\]
that restricts to bounded levels. Hence, $\sfD_{\A_{\mathbf C}}$ is a strong, stable derivator.}

\medskip
The setting of the above theorem is inspired to the following situation: $\D=\sfD_{\Mod(R)}$ is the derivator enhancing the derived category of a ring $R$ and $C$ is a big cotilting $R$-module (see Definition \ref{big_cotilting}). Under these assumptions, $\t_{C}$ is at finite distance from the canonical $t$-structure on $\sfD_{\Mod(R)}$, and this $t$-structure is induced by the classical tilting object $R$. In this way, it is easy to recover as a corollary of the above theorem, one of the main results of~\cite{stovicek2014derived} (see Corollary \ref{coro_stovicek2014derived}).

\medskip
As another consequence of our general theory of realization functors, one can obtain a ``Derived Morita Theory" for Abelian categories, completing~\cite[Thm.\,A]{Jorge_e_Chrisostomos}:

\medskip\noindent
{\bf Theorem E} (Theorem \ref{morita_derived_derivators}){\bf.}
{\em Let $\A$ be a Grothendieck category (resp., an (Ab.$4^*$)-$h$ Grothendieck category for some $h\in\N$), denote by $\t_\A=(\sfD^{\leq0}(\A),\sfD^{\geq0}(\A))$ the canonical $t$-structure on $\sfD(\A)$, and let $\B$ be an Abelian category. The following are equivalent:
\begin{enumerate}
\item $\B$ has a projective generator (resp., an injective cogenerator) and there is an exact equivalence $\sfD(\B)\to \sfD(\A)$ that restricts to bounded derived categories;
\item $\B$ has a projective generator  (resp., an injective cogenerator) and there is an exact equivalence of prederivators $\sfD_\B\to \sfD_\A$ that restricts to an equivalence $\sfD^b_\B\to \sfD^b_\A$;
\item there is a tilting (resp., cotilting) object $X$  in $\sfD(\A)$,  whose heart $X^{\perp_{\neq 0}}$ (resp., $ {}^{\perp_{\neq0}}X$) is equivalent to $\B$ and such that the associated $t$-structure $\t_X=(\D_X^{\leq0},\D_X^{\geq0})$ has finite distance from $\t_\A$.
\end{enumerate}}

\medskip\noindent
{\bf Acknowledgement.} It is a pleasure for me to thank Dolors Herbera, Fosco Loregian, Moritz Groth, Chrysostomos Psaroudakis, Manuel Saor\'\i n, Nico Stein, Jan \v{S}\v{t}ov\'{\i}\v{c}ek, and Jorge Vitoria for several discussions and suggestions.

\newpage
\section{Generalities and preliminaries}

\subsection{Preliminaries and notation}\label{sec_prelim}

Given a category $\C$ and two objects $x,\, y\in\Ob(\C)$, we denote by $\C(x,y):=\hom_\C(x,y)$ the set of morphism from $x$ to $y$ in $\C$.

\medskip\noindent
\textbf{Ordinals.}
Any ordinal $\lambda$ can be viewed as a category in the following way: the objects of $\lambda$ are the ordinals $\alpha<\lambda$ and, given $\alpha$, $\beta<\lambda$, the $\hom$-set $\lambda(\alpha,\beta)$ is a point if $\alpha\leq \beta$, while it is empty otherwise. Following this convention,
\begin{itemize}
\item $\bbone=\{0\}$ is the category with one object and no non-identity morphisms;
\item $\bbtwo=\{0\to 1\}$ is the category with exactly two different objects and one non-identity morphism between them;
\item in general, $\ordn=\{0\to 1\to\cdots\to (n-1)\}$, for any $n\in\N_{>0}$.
\end{itemize}

\medskip\noindent
\textbf{Functor categories, limits and colimits.}
A category $I$ is said to be ({\bf skeletally}) {\bf small} when (the isomorphism classes of) its objects form a set. If $\C$ and $I$
are an arbitrary and a small category, respectively, a functor $I\to \C$ is said to be a {\bf diagram} on $\C$ of shape $I$. The category of diagrams on $\C$ of shape $I$, and natural transformations among them, will be denoted by $\C^I$.
A diagram $X$ of shape $I$, will be also denoted as $(X_i)_{i\in I}$, where $X_i
:= X(i)$ for each $i \in \Ob (I)$. When any diagram of shape $I$ has a limit (resp.\,colimit), we say
that $\C$ has all $I$-limits (resp., colimits). In this case, $\lim_I \colon \C^I\to \C$ (resp., $\colim_I \colon \C^I\to \C$) will denote
the ($I$-)limit (resp., ($I$-)colimit) functor and it is right (resp., left) adjoint to the constant diagram functor
$\Delta_I\colon \C\to \C^I$. 
The category $\C$ is said to be
{\bf complete} (resp., {\bf cocomplete, bicomplete}) when $I$-limits (resp., $I$-colimits, both) exist in $\C$, for any small category $I$.
When $I$ is a directed set, viewed as
a small category in the usual way, the corresponding colimit
functor is the ($I$-){\bf directed colimit functor} 
$\varinjlim_I\colon \C^I\to \C$. The $I$-diagrams on $\C$ are usually called {\bf directed
systems}  of shape $I$ in $\C$.

\medskip\noindent
\textbf{Triangulated categories.}
We refer  to~\cite{Neeman} for the precise definition of triangulated category. In particular, given a triangulated category $\mathcal D$, we will always denote by $\Sigma\colon\mathcal D\to \mathcal D$ the {\bf suspension functor}, and we will denote ({\bf distinguished}) {\bf triangles} in $\mathcal D$ either by $X \to Y\to  Z\overset{+}\to$ or  by $X\to Y \to Z\to \Sigma X$.
A set $\S \subseteq \Ob(\mathcal D)$ is called a {\bf set of generators} of $\mathcal D$ if an object $X$ of $\mathcal D$ is zero whenever $\mathcal D(\Sigma^kS, X) = 0$,
for all $S \in\S$ and $k \in\Z$. In case $\mathcal D$ has coproducts, we shall say that an object $X$ is a {\bf compact object} when the functor $\mathcal D(X,-) \colon \mathcal D\to \Ab$ preserves coproducts. We will say that $\mathcal D$ is {\bf compactly generated} when it has a set of compact generators. Given a set $\mathcal X$ of objects in $\mathcal D$ and a subset $I\subseteq \Z$, we let
\begin{align*}
\mathcal{X}^{\perp_{I}}&:=\{Y\in\mathcal D:\mathcal D(X,\Sigma^iY)=0\text{, for all }X\in\mathcal{X}\text{ and }i\in I\}\\
{}^{\perp_{I}}\mathcal{X}&:=\{Z\in\mathcal D:\mathcal D(Z,\Sigma^iX)=0\text{, for all }X\in\mathcal{X}\text{ and }i\in I\}.
\end{align*}
If $I=\{i\}$ for some $i\in \Z$, then we let $\mathcal{X}^{\perp_{i}}:=\mathcal{X}^{\perp_{I}}$ and ${}^{\perp_{i}}\mathcal{X}:={}^{\perp_{I}}\mathcal{X}$. If $i=0$, we even let $\mathcal{X}^{\perp_{}}:=\mathcal{X}^{\perp_{0}}$ and ${}^{\perp_{}}\mathcal{X}:={}^{\perp_{0}}\mathcal{X}$.

\medskip\noindent
\textbf{D\'evissage} 
Let $\mathcal D$ be a triangulated category and $\S\subseteq \mathcal D$ a subclass. we denote by $\mathrm{thick}(\S)$ the smallest subcategory containing  $\S$ which is triangulated and closed under direct summands, and by $\mathrm{Loc}(\S)$ (resp., $\mathrm{coLoc}(\S)$) the smallest subcategory containing $\S$ which is triangulated and closed under small coproducts (resp., products). We recall from~\cite{nicolas2008torsion}, that $\mathcal D$ is said to satisfy the {\bf principle of d\'evissage} (resp., {\bf infinite d\'evissage}, {\bf dual infinite d\'evissage}) with respect to $\S$ provided $\mathrm{thick}(\S)=\mathcal D$ (resp., $\mathrm{Loc}(\S)=\mathcal D$, $\mathrm{coLoc}(\S)=\mathcal D$).

\medskip\noindent
\textbf{Cohomological functors and $t$-structures.}
Given a triangulated category $\mathcal D$ and an Abelian category $\C$, an additive functor $H^0 \colon \mathcal D\to \C$ is said to be a {\bf cohomological functor} if, for any given triangle $X\to Y\to Z\to \Sigma X$, the sequence $H^0(X)\to H^0(Y)\to H^0(Z)$ is exact in $\C$. In particular, one obtains a long exact sequence as follows:
\[
\cdots \to H^{n-1}(Z)\to H^{n}(X)\to H^{n}(Y)\to H^{n}(Z)\to H^{n+1}(X)\to \cdots
\]
where $H^n := H^0 \circ \Sigma^{n}$, for any $n \in \Z$.
A {\bf $t$-structure} in $\mathcal D$ is a pair $\t=(\mathcal D^{\leq0}, \mathcal D^{\geq0})$ of full subcategories, closed
under taking direct summands in $\mathcal D$, which satisfy the following properties, where $\mathcal D^{\leq n}:=\Sigma^{-n}\mathcal D^{\leq0}$, and $\mathcal D^{\geq n}:=\Sigma^{-n}\mathcal D^{\geq0}$:
\begin{enumerate}
\item[($t$-S.1)] $\mathcal D(X, Y) = 0$, for all $X \in \mathcal D^{\leq0}$ and $Y \in \mathcal D^{\geq1}$;
\item[($t$-S.2)] $ \mathcal D^{\leq-1} \subseteq \mathcal D^{\leq0}$ (or, equivalently, $ \mathcal D^{\geq1} \subseteq \mathcal D^{\geq0}$);
\item[($t$-S.3)] for each $X \in \Ob(\mathcal D)$, there is a triangle 
\[
X^{\leq0} \to X \to X^{\geq 1}\overset{+}\to
\] 
in $\mathcal D$, where $X^{\leq0} \in \mathcal D^{\leq0}$ and $X^{\geq 1} \in \mathcal D^{\geq 1}$.
\end{enumerate}
We will call $\mathcal D^{\leq0}$ and $\mathcal D^{\geq0}$ the {\bf aisle} and the {\bf co-aisle} of the $t$-structure, respectively. The objects $X^{\leq0}$ and $X^{\geq1}$ appearing in the triangle of the above axiom ($t$-S.3) are uniquely determined by $X$, up to a unique isomorphism, and define
functors $(-)^{\leq0}\colon \mathcal D\to \mathcal D^{\leq0}$ and $(-)^{\geq1}\colon \mathcal D\to \mathcal D^{\geq1}$ which are right and left adjoints to the respective inclusion functors. We call them the {\bf left} and {\bf right truncation functors} with respect to the given $t$-structure $\t$. Furthermore, the above triangle will be referred to as the {\bf truncation triangle} of $X$ with respect to $\t$. 
The full subcategory $\mathcal H := \mathcal D^{\leq0} \cap \mathcal D^{\geq0}$ is called the {\bf heart} of the $t$-structure.

The following easy lemma will be needed later on.

\begin{lem}\label{general_iso_truncations}
Let $\mathcal D$ be a triangulated category and let $\t=(\mathcal D^{\leq0},\mathcal D^{\geq0})$ be a $t$-structure in $\mathcal D$. Given a triangle
\[
A\to B\to C\to \Sigma A
\]
such that $A\in\mathcal D^{\leq-1}$, we have that $B^{\geq0}\cong C^{\geq0}$.
\end{lem}
\begin{proof}
Consider the following commutative diagram with exact rows and columns, that comes from the ``tetrahedral axiom":
\[
\xymatrix{
&&C^{\leq-1}\ar[d]\ar@{=}[r]&C^{\leq-1}\ar[d]\\
A\ar[r]\ar[d]&B\ar[r]\ar@{=}[d]&C\ar[r]\ar[d]&\Sigma A\ar[d]\\
A'\ar[r]&B\ar[r]&C^{\geq0}\ar[r]\ar[d]&\Sigma A'\ar[d]\\
&&\Sigma C^{\leq-1}\ar@{=}[r]&\Sigma C^{\leq-1}&
}
\]
Since $A$ and $C^{\leq-1}\in\mathcal D^{\leq-1}$, and using that $\mathcal D^{\leq-1}$ is closed under taking extensions, one shows that $A'\in \mathcal D^{\leq-1}$. But then, the triangle $A'\to B\to C^{\geq0}\to \Sigma A'$ is an approximation triangle for $B$, showing that $C^{\leq-1}\cong A'$ and the desired isomorphism $B^{\geq0}\cong C^{\geq0}$.
\end{proof}

\medskip\noindent
\textbf{Model categories and model approximations.} 
For the definition of  {\bf model category}  we refer to~\cite{dwyer1995homotopy}. In particular, 
a model category for us is just a finitely bicomplete category $\M$ endowed with a model structure $(\W,\C,\F)$. When needed we will explicitly mention that $\M$ is a {\bf $\Dia$-bicomplete} model category to mean that $\M$ has all limits and colimits of diagrams of shape $I$, for any $I\in \Dia$ (where $\Dia$ is a suitable class of small categories).

A rich source of model categories for us is the construction of towers and telescopes of model categories. We now recall these constructions  from~\cite{chacholski2017relative}:

\begin{defn}\label{tower_of_models}
Let $\M_\bullet=\{\M_n:n\in\N\}$ be a sequence of categories connected with adjunctions 
$$l_{n+1}:  \xymatrix{\M_{n+1}\ar@<2pt>[r]\ar@<-2pt>@{<-}[r]&\M_{n}} :r_n\,.$$ 
The {\bf category of towers on $\M_\bullet$}, $\Tow(\M_\bullet)$ is defined as follows:
\begin{itemize}
\item[\rm (Tow.$1$)]  an {\em object} is a pair $(a_\bullet, \alpha_\bullet)$, where $a_\bullet=\{a_n\in \M_n:n\in\N\}$ is a sequence of objects one for each $\M_n$, and $\alpha_\bullet=\{\alpha_{n+1}\colon a_{n+1}\to r_n(a_n):n\in\N\}$ is a sequence of morphisms;
\item[\rm (Tow.$2$)] a {\em morphism} $f_\bullet\colon(a_\bullet,\alpha_\bullet)\to (b_\bullet,\beta_\bullet)$ is a sequence of morphisms $f_\bullet=\{f_n\colon a_n\to b_n:n\in\N\}$ such that $r_n(f_n)\circ\alpha_{n+1}=\beta_{n+1}\circ f_{n+1}$, for all $n\in\N$.
\end{itemize}
The dual of a category of towers is said to be a {\bf category of telescopes}, and denoted by $\Tel(\M_\bullet')$.
\end{defn}

If each $\M_n$ in the above definition is a $\Dia$-bicomplete category, then one can construct limits and colimits component-wise in $\Tow(\M_\bullet)$, so, under these hypotheses, the category of towers is $\Dia$-bicomplete. 

\begin{prop}\label{model_torre}{\rm~\cite[Prop.\,4.3]{chacholski2017relative}}
Let $\M_\bullet=\{(\M_n,\W_n,\C_n,\F_n):n\in\N\}$ be a sequence of model categories connected with adjunctions 
\[
l_{n+1}:  \xymatrix{\M_{n+1}\ar@<2pt>[r]\ar@<-2pt>@{<-}[r]&\M_{n}} :r_n
\]
and suppose that each $r_n$ preserves fibrations and acyclic fibrations. Then there exists a model structure $(\W_\Tow,\C_\Tow,\F_\Tow)$ on the category of towers $\Tow(\M_\bullet)$, such that $\W_\Tow=\{f_\bullet:f_n\in\W_n \,,\ \forall n\in\N\}$ and
%\item $\C_\Tow=\{f_\bullet:f^*_n\in\C_n \,,\ \forall n\in\N\}$, where $f^*_n$ is constructed as follows. First we define an object $(p_\bullet, \pi_\bullet)$ in $\Tow(\M_\bullet)$ where each $p_n$ comes from a pull-back diagram
%\[
%\xymatrix{
%p_{n}\ar@{}[rrd]|{\rm P.B.}\ar[d]_{\bar\beta_n}\ar[rr]^{\bar f_{n-1}}& &    b_n\ar[d]^{\beta_n}\\
%r_{n-1}(a_{n-1})\ar[rr]_{r_{n-1}(f_{n-1})}& &         r_{n-1}(b_{n-1})
%}
%\]
%and $\pi_n=r_{n-1}(f_{n-1})\circ\bar\beta_n$. Then,  $f^*_n:a_n\to p_n$ is defined, using the universal property of the pull-back, as the unique morphism such that $\bar\beta_nf^*_n=\alpha_n$ and $\bar f_{n-1}f^*_n=f_n$. 
 $\F_\Tow=\{f_\bullet:f_n\in\F_n \,,\ \forall n\in\N\}$.
%Then, $(\Tow(\M_\bullet),\W_\Tow,\C_\Tow,\F_\Tow)$ is a model category.
\end{prop}

The concept of model approximation was introduced by Chach{\'o}lski and Scherer in order to circumvent the difficulties in constructing homotopy limits (see~\cite{PS}):

\begin{defn}\label{mod_appr} Let $(\M',\W')$ be a category with weak equivalences, that is, $\W'$ is a class o morphisms in $\M'$, containing all the isos and with the $3$-for-$2$ property. A {\bf right model approximation} for $(\M',\W')$ is a model category $(\M,\W,\C,\F)$ and a pair of functors 
\[
l :  \xymatrix{\M'\ar@<2pt>[r]\ar@<-2pt>@{<-}[r]&\M}  : r
\]
satisfying the following conditions:
\begin{enumerate}
\item[\rm ({MA}.1)] $l$ is left adjoint to $r$;
\item[\rm ({MA}.2)] if $\f\in\W'$, then $l(\f)\in\W$; 
\item[\rm ({MA}.3)] if $\psi$ is a weak equivalence between fibrant objects, then $r(\psi)\in\W'$;
\item[\rm ({MA}.4)] if $l(X)\to Y$ is a weak equivalence in $\M$ with $X$ fibrant, the adjoint morphism $X \to r(Y)$ is in $\W'$.
\end{enumerate}
One defines dually {\bf left model approximations} for $(\M',\W')$.
\end{defn}

In fact, the notion of model approximation is as good as the notion of model category in order to construct homotopy categories. Indeed, let $l:(\M',\W')\rightleftarrows (\M,\W,\C,\F):r$ be a right model approximation, then the category $\M'[\W'{}^{-1}]$ is locally small as it is equivalent to the category $\Ho(\M')$, which is constructed as follows: objects of $\Ho(\M')$ are those of $\M'$ and, given $M_1,\, M_2\in\M'$, 
\[
\Ho(\M')(M_1,M_2):=\Ho(\M)(l(M_1),l(M_2)),
\]
see~\cite[Prop.\,5.5]{PS}.
The category $\Ho(\M')$ is said to be the {\bf homotopy category} of the above model approximation. It is easily seen that it is equivalent to a full subcategory of $\Ho(\M)$.

\subsection{Generalities on (pre)derivators}

We denote by $\Cat$ the $2$-category of small categories and by $\Cat^{\op}$ the $2$-category obtained by reversing the direction of the functors in $\Cat$ (but letting the direction of natural transformations unchanged). Similarly, we  denote by $\CAT$ the $2$-``category" of all categories. This, when taken literally, may cause some set-theoretical problems that, for our constructions, can be safely ignored: see the discussion after~\cite[Def.\,1.1]{Moritz}.

\begin{defn}
A {\bf category of diagrams} is a full $2$-subcategory $\Dia$ of $\Cat$, such that
\begin{enumerate}
\item[(Dia.1)] all finite posets, considered as categories, belong to $\Dia$;
\item[(Dia.2)] given $I\in \Dia$ and $i\in I$, the slices  $I_{i/}$ and $I_{/i}$ belong to $\Dia$;
\item[(Dia.3)] if $I\in \Dia$, then $I^\op\in \Dia$;
\item[(Dia.4)] for every Grothendieck fibration $u : I\to J$, if all fibers $I_{j}$, for $j\in J$,  and the base $J$ belong to $\Dia$, then so does $I$. 
\end{enumerate}
\end{defn}

\begin{eg}
A category $I$ is said to be a {\bf finite directed category} if it has a finite number of objects and morphisms, and if there is no directed cycle in the quiver whose vertices are the objects of $I$ and the arrows are the non-identity morphisms in $I$. Equivalently, the nerve of $I$ has a finite number of non-degenerate simplices. We denote by $\fincat$ be the $2$-category of finite directed categories (categories whose nerve has a finite number of non-degenerate simplices). This $2$-category is a category of diagrams.
\end{eg}

Given a category of diagrams $\Dia$, a {\bf prederivator of type} $\Dia$ is a strict $2$-functor
\[
\D\colon \Dia^\op\to \CAT.
\]
All along this paper, we will follow the following notational conventions:
\begin{itemize}
\item the letter $\D$ will always denote a (pre)derivator;
\item for a prederivator $\D\colon \Dia^{\op}\to \CAT$ and a small category $I\in\Dia$, we denote by 
\[
\D^I\colon \Dia^{\op}\to\CAT
\] 
the {\bf shifted prederivator} such that $\D^I(J):=\D(I\times J)$;
\item  for any natural transformation $\alpha\colon u\to v\colon J\to I$ in $\Dia$, we will always use the notation $\alpha^*:=\D(\alpha)\colon u^*\to v^*\colon \D(I)\to \D(J)$. Furthermore, we denote respectively by $u_!$ and $u_*$ the left and the right adjoint to $u^*$ (whenever they exist), and call them respectively the {\bf left} and {\bf right homotopy Kan extension of $u$};
\item the letters $K$, $U$, $V$, $W$, $X$, $Y$, $Z$,  will be used either for objects in the base $\D(\bbone)$ or for (incoherent) diagrams on $\D(\bbone)$, that is, functors $I\to \D(\bbone)$, for some small category $I$;
\item the letters $\mathscr K$, $\mathscr U$, $\mathscr V$, $\mathscr W$, $\mathscr X$, $\mathscr Y$, $\mathscr Z$,  will be used for objects in some image $\D(I)$ of the derivator, for $I$ a category (possibly) different from $\bbone$. Such objects will be usually referred to as {\bf coherent diagrams of shape $I$};
\item given $I\in \Dia$, consider the unique functor $\pt_I\colon I\to \bbone$. We usually denote by $\hocolim_I\colon \D(I)\to \D(\bbone)$ and $\holim_I\colon \D(I)\to \D(\bbone)$ respectively the left and right homotopy Kan extensions of $\pt_I$; these functors are called respectively {\bf homotopy colimit} and {\bf homotopy limit}. 
\end{itemize}

For a given object $i\in I$, we also denote by $i$ the inclusion $i\colon\bbone\to I$ such that $0\mapsto i$. We obtain an evaluation functor $i^*\colon \D(I)\to \D(\bbone)$. For an object $\mathscr X\in \D(I)$, we let $\mathscr X_i:=i^*\mathscr X$. Similarly, for a morphism $\alpha\colon i\to j$ in $I$, one can interpret $\alpha$ as a natural transformation from $i\colon \bbone \to I$ to $j\colon \bbone \to I$. In this way, to any morphism $\alpha$ in $I$, we can associate $\alpha^*\colon i^*\to j^*$. For an object $\mathscr X\in \D(I)$, we let $\mathscr X_\alpha:=\alpha^*_{\mathscr X}\colon \mathscr X_i\to \mathscr X_j$.
For  $I$ in $\Dia$, we denote by 
\[
\dia_I\colon \D(I)\to \D(\bbone)^{I}
\]
the {\bf diagram functor}, such that, given $\mathscr X\in \D(I)$, $\dia_I(\mathscr X)\colon I\to \D(\bbone)$ is defined by $\dia_I(\mathscr X)(i\overset{\alpha}{\to} j)=(\mathscr X_i\overset{\mathscr X_\alpha}{\to} \mathscr X_j)$. We will refer to $\dia_I(\mathscr X)$ as the {\bf underlying} (incoherent) {\bf diagram} of the coherent diagram $\mathscr X$.

\begin{eg}\label{description_constant_diagram}
Let $\D\colon \Dia^{\op}\to \CAT$ be a prederivator. Given $I\in \Dia$, consider the unique functor $\pt_I\colon I\to \bbone$, let $X\in \D(\bbone)$ and consider $\mathscr X:=\pt_I^*X\in \D(I)$. Then the underlying diagram $\dia_I(\mathscr X)\in \D(\bbone)^I$ is constant, that is, $\mathscr X_i= X$ for all $i\in I$, and the map $\mathscr X_\alpha\colon \mathscr X_i\to \mathscr X_j$ is the identity of $X$ for all $(\alpha\colon i\to j)\subseteq I$.
\end{eg}

\begin{defn}
A {\bf derivator of type $\Dia$}, for a given category of diagrams $\Dia$, is a prederivator $\D\colon \Dia^\op\to\CAT$ satisfying the following axioms:
\begin{enumerate}
\item[(Der.$1$)] if $\coprod_{i\in I} J_i$ is a disjoint union in $\Dia$, then the canonical functor $\D(\coprod_I J_i) \to \prod_I\D(J_i)$ is an equivalence of categories;
\item[(Der.$2$)] for any $I\in\Dia$ and a morphism $f\colon \mathscr X \to \mathscr Y$ in $\D(I)$, $f$ is an isomorphism if and only if $i^*(f)\colon i^*(\mathscr{X}) \to i^*(\mathscr{Y})$ is an isomorphism in $\D(\bbone)$ for each $i\in I$;
\item[(Der.$3$)] the following conditions hold true:
\begin{enumerate}
\item[(L.Der.$3$)] for each functor $u\colon I \to J$ in $\Dia$, the functor $u^*$ has a left adjoint $u_!$ (i.e.\  left homotopy Kan extensions are required to exist);
\item[(R.Der.$3$)] for each functor $u\colon I \to J$ in $\Dia$, the functor $u^*$ has a right adjoint $u_*$ (i.e.\  right homotopy Kan extensions are required to exist);
\end{enumerate}
\item[(Der.$4$)] the following conditions hold true:
\begin{enumerate}
\item[(L.Der.$4$)] the left homotopy Kan extensions can be computed pointwise in that for each $u\colon I \to J$ in $\Dia$ and $j\in J$, there is a canonical isomorphism $\hocolim_{u/j} p^*(\mathscr X) \cong (u_!\mathscr X)_j$, where $p\colon u/j\to I$ is the canonical functor from the slice category;
\item[(R.Der.$4$)] the right homotopy Kan extensions can be computed pointwise in that for each $u\colon I \to J$ in $\Dia$ and $j\in J$, there is a  canonical isomorphism $(u_*\mathscr X)_j \cong \holim_{j/u } q^*(\mathscr X)$ in $\D(\bbone)$, where $q\colon j/u\to I$ is the canonical functor from the slice category.
\end{enumerate}
\end{enumerate}
 If $\Dia$ is not explicitly mentioned we just assume that $\Dia=\Cat$. 
\end{defn}

We refer to~\cite{Moritz} for a detailed discussion, as well as for the precise definitions of {\bf pointed} derivators ($\D(\bbone)$ has a zero object), and {\bf strong} derivators (the partial diagram functors $\D(\bbtwo\times I) \to \D(I)^\bbtwo$ are full and essentially surjective for each $I\in\Cat$).  

For a fixed category of diagrams $\Dia$, the prederivators of type $\Dia$ form a $2$-category, that we denote by $\PDer_\Dia$, where $0$-cells are the $2$-functors $\Dia^{\op}\to \CAT$ (i.e., the prederivators), $1$-cells are $2$-natural transformations among these functors, and $2$-cells are modifications. We will usually refer to the $1$-cells of $\PDer_\Dia$ as {\bf morphisms of prederivators} or, abusing terminology, {\bf functors} between prederivators, while we will usually refer to the $2$-cells of $\PDer_\Dia$ as {\bf natural transformations}.

\begin{eg}
Given a prederivator $\D\colon \Dia^\op\to \CAT$ and a morphism $u\colon J\to I$ in $\Dia$, for any $K\in \Dia$, one can consider $u^*\colon \D^I(K)\to \D^J(K)$. These functors can be assembled together to form a morphism of prederivators that, abusing notation, we denote again by $u^*\colon \D^I\to \D^J$.\\ If $\D$ is a derivator, then one can construct similarly two morphisms of derivators $u_!,u_*\colon \D^J\to \D^I$. Of course there are adjunctions $(u_!,u^*)$ and $(u^*,u_*)$ in the $2$-category $\PDer_\Dia$.
\end{eg}

\subsection{Derivators induced by model categories}

In this subsection we introduce our main source for explicit examples of derivators. Indeed, for a given $\Dia$-bicomplete model category $\M$ whose class of weak equivalences is denoted by $\W$,  the following results are proved in~\cite{Cis03}:
\begin{itemize}
\item for any $I\in \Dia$, let $\W_I$ be the class of morphisms in $\C^I$ which belong pointwise to $\W$. Then,~\cite[Thm.\,1]{Cis03} tells us that the category of fractions $\M^I[\W_I^{-1}]$ can always be constructed (in the same universe);
\item the assignment $I\mapsto \M^{I}[\W_I^{-1}]$ underlies a derivator \[\D_{(\M,\W)}\colon \Dia^{\op}\to \CAT;\]
\item the derivator $\D_{(\M,\W)}$ is strong and it is pointed if $\M$ is pointed.
\end{itemize}

In fact, there is another approach to the proof of the above facts that uses model approximations. This is based on the following theorem that collects several facts proved in~\cite{PS}:

\begin{thm}\label{derivator_from_approximation}
Let $(\M',\W')\rightleftarrows (\M,\W,\C,\F)$ be a left model approximation. For any small category $I$, let $(\W')^I$ be the class of maps in $(\M')^I$ that belong pointwise to $\W'$, then there exists a model category $(\M_I,\W_I,\C_I,\F_I)$ and a left model approximation $((\M')^I,(\W')^I)\rightleftarrows (\M_I,\W_I,\C_I,\F_I)$. Hence, the category $\Ho((\M')^I)$ has small hom-sets and we can define a prederivator 
\begin{align*}
\D_{(\M',\W')}\colon \Cat^\op&\to \CAT\\
I&\mapsto \Ho((\M')^I).
\end{align*}
Furthermore, the above approximations are ``good for left Kan extensions", so they can be used to show that the prederivator $\D_{(\M',\W')}$ satisfies (Der.$1$), (Der.$2$), (L.Der.$3$) and (L.Der.$4$). 

As a consequence, if $(\M',\W')$ admits both a left and a right model approximation (e.g., if $\M'$ admits a model structure on it for which the class of weak equivalences is exactly $\W'$), then $\D_{(\M',\W')}$ is a derivator.
\end{thm}
%
%In what follows we mention some classes of examples of (strong and stable) derivators that will appear frequently in the rest of the paper:
%\begin{eg} \label{expl.derived categories}
%Let $(\C,\mathcal W,\mathcal B,\mathcal F)$ be a model category. For any small category $I$,  For such derivator, homotopy co/limits and, more generally, homotopy Kan extensions, are just the total derived functors of the usual co/limit and Kan extension functors. 
%
%Given a Grothendieck category $\G$, %(resp., a small dg algebra $\A$)
%we refer to the strong and stable derivator arising as above from the injective model structure on $\Ch(\G)$ %(resp., the projective model structure on $\C(\A)$),
%as the {\bf canonical derivator enhancing the derived category} $\sfD(\G)$. %(resp., $\Der(\A)$).
%\end{eg}

\subsection{Stable derivators}\label{prelim_stab_der}

Let $\bigsquare:=\bbtwo\times \bbtwo$, with the following labels for vertices and arrows:
\[
\xymatrix@R=10pt@C=18pt{
(0,0) \ar[r]^{N}\ar[dd]_W&(0,1)\ar[dd]^E\\
&&\\
(1,0)\ar[r]_S&(1,1).
}
\]
Then we let $\iota_{\bullet}\colon \bbone\to \bigsquare$ (with $\bullet\in \{(0,0),(0,1),(1,0),(1,1)\}$) and $\iota_{\bullet}\colon \bbtwo\to \bigsquare$ (with $\bullet\in \{N,S,W,E\}$) be the inclusion of the corresponding vertex and arrow, respectively. Furthermore, we consider the inclusion of the two obvious subcategories $\bigulangle \overset{\iota_{\ulangle}}{\longrightarrow}\bigsquare\overset{\iota_{\drangle}}{\longleftarrow}\bigdrangle$,
\[
\xymatrix@R=10pt@C=18pt{
(0,0)\ar[r]\ar[dd]&(0,1)&&(0,0) \ar[r]\ar[dd]&(0,1)\ar[dd]&&&(0,1)\ar[dd]\\
&\ar@{.>}[rr]|{\iota_{\ulangle}}&&&&&\ar@{.>}[ll]|{\iota_{\drangle}}&\\
(1,0)&&&(1,0)\ar[r]&(1,1)&&(1,0)\ar[r]&(1,1).
}
\]
Let also $\iota^{\ulangle}_{\bullet}\colon \bbone\to \bigulangle$ (with $\bullet\in \{(0,0),(0,1),(1,0)\}$), $\iota^\ulangle_{\bullet}\colon \bbtwo\to \bigulangle$ (with $\bullet\in \{N,W\}$), $\iota^{\drangle}_{\bullet}\colon \bbone\to \bigdrangle$ (with $\bullet\in \{(1,0),(0,1),(1,1)\}$) and $\iota^\drangle_{\bullet}\colon \bbtwo\to \bigdrangle$ (with $\bullet\in \{S,E\}$) be the inclusion of the corresponding vertex and arrow, respectively. 

\medskip
Let $\Dia$ be some category of diagrams and let $\D\colon \Dia^\op\to \CAT$ be a pointed derivator. The {\bf suspension functor} $\Sigma\colon \D\to \D$ and the {\bf loop functor} $\Omega\colon \D\to \D$ are defined as follows:
\[
\Sigma:=\iota_{(1,1)}^*\circ (\iota_{\ulangle})_!\circ (\iota^\ulangle_{(0,0)})_*\qquad\Omega:=\iota_{(0,0)}^*\circ (\iota_{\drangle})_*\circ (\iota^\drangle_{(1,1)})_!
\]
and they form an adjoint pair $(\Sigma,\Omega)\colon \D\rightleftarrows \D$, which is an equivalence if and only if $\D$ is a {\bf stable derivator}. Furthermore, let 
\[
\mathrm C:=\iota_{E}^*\circ (\iota_{\ulangle})_!\circ (\iota^\ulangle_N)_*\qquad\mathrm F:=\iota_{W}^*\circ (\iota_{\drangle})_*\circ (\iota^\drangle_{S})_!
\]
this gives an adjoint pair $(\mathrm S,\mathrm F)\colon \D^{\bbtwo}\rightleftarrows \D^{\bbtwo}$ which, again, is an equivalence if and only if $\D$ is stable. We define, respectively, the {\bf cone functor} $\cone\colon \D^\bbtwo\to \D$ and the {\bf fiber functor} $\fib\colon \D^\bbtwo\to \D$ as 
\[
\cone:=\iota_{(1,1)}^*\circ (\iota_{\ulangle})_!\circ (\iota^\ulangle_N)_*\qquad\fib:=\iota_{(0,0)}^*\circ (\iota_{\drangle})_*\circ (\iota^\drangle_{S})_!.
\]
The key fact about stability, which can be found in~\cite[Thm.\,4.16 and Cor.\,4.19]{Moritz}, is that given a strong and stable derivator $\D$, the above structure can be used to endow each $\D(I)$ with a canonical triangulated structure, with respect to which, all the $u^*$ and all the Kan extensions $u_!$ and $u_*$ are naturally triangulated functors.

\begin{eg}
Let $(\C,\mathcal W,\mathcal B,\mathcal F)$ be a model category and consider the associated derivator  $\D_{(\C,\W)}\colon \Cat^{\op}\to \CAT$. If $(\C,\mathcal W,\mathcal B,\mathcal F)$ is stable in the sense of model categories, then $\D_{(\C,\W)}$ is strong and stable. 
\end{eg}

\subsection{Localization of derivators}

In this subsection we recall the definition and some basic facts about co/reflections in $\PDer_\Dia$. This concept is studied, for example, in~\cite{heller,cis08,GLV,coley}: 
\begin{defn}\label{defn:refl-loc}
Let $L :\D\rightleftarrows\mathbb E : R$ be an adjunction in $\PDer_\Dia$.
\begin{enumerate}
\item[$\bullet$] The adjunction $(L,R)$ is a \textbf{reflection} if $R$ is fully faithful, i.e.\, the counit $\varepsilon \colon LR\to\id$ is invertible.
\item[$\bullet$] The adjunction $(L,R)$ is a \textbf{coreflection} if $L$ is fully faithful, i.e.\, the unit $\eta \colon \id\to RL$ is invertible.
\end{enumerate}
\end{defn}

\begin{lem}\label{loc_is_der}{\rm~\cite[Prop.\,3.12]{GLV}}
Let $L :\D\rightleftarrows\mathbb E : R$ be a reflection in $\PDer_\Dia$. If $\D$ is a (pointed, strong) derivator of type $\Dia$ for some category of diagrams $\Dia$, so is $\mathbb E$. Furthermore, given $u\colon I\to J$ in $\Dia$, the homotopy Kan extensions $u_!,u_*\colon \mathbb E(I)\to \mathbb E(J)$ can be computed, respectively, as follows:
\[
L_Ju_!R_I :\E(I)\to\E(J)\qquad\text{and}\qquad L_Ju_* R_I :\E(I)\to\E(J).
\]
where the $u_!$ and $u_*$ appearing in the above equation are the homotopy Kan extensions in $\D$.
\end{lem}

In general a co/reflection of a stable derivator need not be stable. However the following result allows to easily check when a reflection of a stable derivator is again stable:

\begin{lem}\label{stability_of_loc}
Let $L :\D\rightleftarrows\mathbb E : R$ be a reflection in $\PDer_\Dia$ and suppose that $\D$ is strong and stable. Denote by $(\bar \Sigma, \bar \Omega)$ the suspension-loop adjunction in $\mathbb E$. Then $\bar \Sigma\cong L_{}\Sigma R_{}$ and $R_{}\bar \Omega\cong \Omega R_{}$.
\end{lem}
\begin{proof}
By~\cite[Prop.\,2.11]{Moritz}, left/right exact (in particular right/left adjoint) morphisms of pointed derivators commute with loops/suspensions. In particular, for any $I\in\Dia$, $L_I\Sigma R_{I}\cong \bar\Sigma  L_{I} R_{I}\cong \bar \Sigma $, while $\Omega R_{I} \cong R_{I}\bar \Omega $.
\end{proof}
If we consider $\E$ as a sub-derivator of $\D$, then the above corollary means that $\bar \Omega$ is just a restriction of $\Omega$ to $\E$, while applying $\bar\Sigma$ is the same as first applying $\Sigma$ and then reflecting onto $\E$. In particular, $\E$ is closed under loops in $\D$ and, being $\Omega$ an equivalence, $\bar\Omega$ is always fully faithful, so it is an equivalence if and only if it is essentially surjective.

\subsection{Pseudo-colimits of prederivators}\label{ps_colims_in_pder}

Let $\C$ be a $2$-category (e.g., $\C=\CAT$ or $\C=\PDer_\Dia$) and consider a functor $F\colon \N\to \C$, which is determined by a sequence of ($1$-)morphisms in $\C$:
\[
\xymatrix{
F(0):=C_0\ar[r]^{F_0}&F(1):=C_1\ar[r]^{F_1}&\cdots\ar[r]&F(n):=C_n\ar[r]^{F_n}&\cdots
}
\]
and let $F_{n,m}:=F_m\circ\ldots\circ F_{n-1}$ for $m<n$. The {\bf pseudo-colimit}
\[
C:=\pcolim F
\]
of this diagram is defined by a universal property expressed by the following natural isomorphism of categories:
\[
\C^{\N}(F,\Delta_A)\cong \C(\pcolim F,A)
\]
for any $A\in \C$, where $\Delta_A\colon \N\to \C$ is the functor that takes constantly the value $A$. In the following lemma we give an explicit construction of a pseudo-colimit as above, in the particular case when $\C=\CAT$:

\begin{lem}
Consider a functor $F\colon \N\to \CAT$. The pseudo-colimit $\C:=\pcolim F$ exists and it can be constructed as follows: the class of objects of $\C$ is the disjoint union of the objects of the $\C_n$, that we can write as $\Ob(\C)=\{(X,n):X\in \C_n\}$, while the morphism sets are:
\begin{equation}\label{homs_in_pcolim}
\C((X,n),(Y,m))=\colim_{k\geq n\lor m} \C_k(F_{k,n}X,F_{k,m}Y).
\end{equation}
where this direct limit is taken in $\Set$.
\end{lem}

If in the above lemma all the $F_n$ are  fully faithful, then the direct limit \eqref{homs_in_pcolim} is eventually constant, that is $\C((X,n),(Y,m))\cong\C_k(F_{k,n}X,F_{k,m}Y)$, where $k \geq n\lor m $ is any positive integer.

The above construction of sequential pseudo-colimits in $\CAT$ can be used to give an explicit construction of sequential pseudo-colimits in $\PDer_\Dia$:

\begin{lem}
Consider a functor $F\colon \N\to \PDer_\Dia$. The pseudo-colimit $\D:=\pcolim F$ in $\PDer_\Dia$ exists and, furthermore, $\D(I)$ is the pseudo-colimit $\pcolim F_I$ in $\CAT$, where $F_I\colon \N\to \CAT$ is the functor $F_I(n):=(F(n))(I)$.  
\end{lem}
\begin{proof}
Let us define a prederivator $\D\colon \Dia^{\op}\to \CAT$ as follows:
\begin{itemize}
\item for  $I\in\Dia$, let $\D(I):=\pcolim F_I$;
\item for a functor $u\colon I\to J$, the functor $u^*\colon \D(J)\to \D(I)$ is the unique object in the category on the left-hand side
\begin{equation}\label{univ_prop_pcolim_dim}
\CAT(\D(J),\D(I))\cong \CAT^{\N}(F_J,\Delta_{\D(I)})
\end{equation}
corresponding to an object $F_J\to \Delta_{\D(I)}$ in the category on the right-hand side, whose $n$-th component is the composition $F_J(n)\to F_I(n)\to \D(I)$ (where $F_J(n)\to F_I(n)$ is $u^*$ for the derivator $F(n)$, while $\varepsilon^I_n\colon F_I(n)\to \D(I)$ is the canonical map to the pseudo-colimit);
\item for a natural transformation $\alpha\colon u\to v\colon I\to J$, the natural transformation $\alpha^*\colon u^*\to v^*\colon \D(J)\to \D(I)$ is the unique morphism $\alpha^*$ in the category on the left-hand side in \eqref{univ_prop_pcolim_dim}, corresponding to a morphism in $\CAT^{\N}(F_J,\Delta_{\D(I)})$ whose $n$-th component is 
\[
\xymatrix{
F_J(n)\ar@/_-10pt/[rr]^{u^*}\ar@/_10pt/[rr]_{v^*}&\Downarrow{\text{\scriptsize $\alpha^*$}}&F_I(n)\ar[rr]^{\varepsilon^I_n}&&\D(I).
}
\]
\end{itemize}
With this explicit definition, it is not difficult to deduce from the universal property of pseudo-colimits in $\CAT$ that $\D$ is the pseudo-colimit of $F$ in $\PDer_\Dia$, that is, that there is an isomorphism of categories 
\[
\PDer_\Dia(\D,\E)\cong (\PDer_\Dia)^{\N}(F,\Delta_{\E}),
\] for any prederivator $\E$. 
\end{proof}

\newpage
\section{Preliminaries about derived categories}\label{prelim_der_sec}

In this section we start with an Abelian category $\A$ and we study the standard prederivator whose base is the category of bounded (resp., left-bounded, right-bounded, unbounded) cochain complexes $\Ch^b(\A)$ on $\A$ (resp., $\Ch^{+}(\A)$, $\Ch^-(\A)$, $\Ch(\A)$). In particular, we show how to obtain bounded complexes as a (double) pseudo-colimit of uniformly bounded complexes. Furthermore, we study the localizations of these prederivators with respect to weak equivalences, obtaining the natural prederivators enhancing the bounded (resp., left-bounded, right-bounded,  unbounded) derived category of $\A$. Using model approximations, we give standard assumptions on $\A$ for the unbounded derived category to have small hom-sets, and to  embed it as a full subcategories of the homotopy category of towers/telescopes over $\A$. In this way, the concepts of left-complete, right-complete, and two-sided complete derived category naturally arise. 

\medskip\noindent
{\bf Setting for Section \ref{prelim_der_sec}.}
We fix through this section a category of diagrams $\Dia$ and a $\Dia$-bicomplete Abelian category $\A$. When working with categories of half-bounded or unbounded complexes we will need $\Dia$ to contain the discrete category on a countable set, that is, we need $\A$ to be countably bicomplete.

\subsection{Categories of bounded complexes}

Let us start introducing notations for the following prederivators:
\begin{itemize}
\item $\Ch_\A\colon \Dia^\op\to \CAT$ is the discrete Abelian derivator represented by  the category of unbounded cochain complexes $\Ch(\A)$;
\item $\Ch_\A^b\colon \Dia^\op\to \CAT$ is the full sub-prederivator of $\Ch_\A$ such that $\Ch^b_\A(I):=\Ch^b(\A^I)$, the bounded complexes on $\A^I$;
\item $\Ch_\A^{(b,+)}\colon \Dia^\op\to \CAT$ is the full sub-prederivator of $\Ch^b_\A$ such that $\Ch^{(b,+)}_\A(I):=\Ch^{(b,+)}(\A^I)$, the category of those bounded complexes $X\in \Ch^b(\A^I)$ such that $X^h=0$ for all $h<0$;
\item given $n\in\N$, $\Ch_\A^{\ordn}\colon \Dia^\op\to \CAT$ is the full sub-prederivator of $\Ch^{(b,+)}_\A$ such that $\Ch^{\ordn}_\A(I):=\Ch^{\ordn}(\A^I)$, the category of those bounded complexes $X\in \Ch^b(\A^I)$ such that $X^h=0$ for all $h<0$ and all $h\geq n$.
\end{itemize}

\begin{rmk}
Notice that $\Ch_\A^{\ordn}$ is again a discrete Abelian derivator, but the same is not true for $\Ch_\A^b$ and $\Ch_\A^{(b,+)}$. In fact, while it is true that the associated diagram functors are fully faithful, given $I\in \Dia$, in general the inclusions $\Ch^{b}(\A^I)\subseteq \Ch^b(\A)^I$ and $\Ch^{(b,+)}(\A^I)\subseteq\Ch^{(b,+)}(\A)^I$ may be strict.
\end{rmk}

Consider the following sequence of prederivators, where the maps are the obvious inclusions:
\[
\Ch_\A^{\bbone}\to \Ch_\A^{\bbtwo}\to \cdots \to \Ch_\A^{\ordn}\to \cdots.
\]
and let $\pcolim_n \Ch_\A^{\ordn}$ be the pseudo-colimit in $\PDer_\Dia$, as described in Subsection \ref{ps_colims_in_pder}.
Since these inclusions are componentwise fully faithful one can give a straightforward proof of the following lemma:

\begin{lem}
In the above notation, there is an equivalence of prederivators 
\[
F\colon \pcolim_n \Ch_\A^{\ordn}\to \Ch_{\A}^{(b,+)},
\] 
such that, for all $n\in \N_{>0}$, $I\in \Dia$ and $X\in \Ch_\A^{\ordn}(I)$, $(X,n)\mapsto X$.
\end{lem}

Let us conclude this subsection showing that the prederivator $\Ch_\A^b$ can be seen as a pseudo-colimit of copies of $\Ch_{\A}^{(b,+)}$. Indeed, define the following endofunctors:
\[
\Omega\colon \Ch_\A\to \Ch_\A \quad (\text{resp., } \Sigma\colon \Ch_\A\to \Ch_\A)
\]
such that, given $X\in \Ch_\A(I)=\Ch(\A^I)$, we have $(\Omega X)^h:=X^{h-1}$ (resp., $(\Sigma X)^h:=X^{h+1}$), with the usual sign conventions for differentials. Notice that the above use of the symbols $\Sigma$ and $\Omega$ does not coincide with the general definition for pointed derivators given in Subsection \ref{prelim_stab_der}. 

\begin{lem}
In the above notation, consider a diagram $F\colon \N\to \PDer_\Dia$ such that $F(i)=\Ch_{\A}^{(b,+)}$ for all $i\in \N$ and $F(i)\to F(i+1)$ is a suitable restriction of $\Omega\colon \Ch_\A\to \Ch_\A$. Then there is an equivalence of prederivators
\[
\pcolim F=\pcolim_n \Ch^{(b,+)}_\A\to \Ch_{\A}^{b}.
\]
\end{lem}

\subsection{Naive and smart truncations of complexes}
Let us introduce the so-called {\bf smart truncations} of a complex. Indeed, given an integer $k\in\N$ and $X\in \Ch(\A)$, we let
\[
\tau^{\leq k}X:\quad(\xymatrix@C=10pt{\cdots\ar[r]&X^{k-2}\ar[r]&X^{k-1}\ar[r]&\ker(d^{k})\ar[r]&0\ar[r]&\cdots});
\]
\[
\tau^{\geq k}X:\quad(\xymatrix@C=10pt{\cdots\ar[r]&0\ar[r]&\mathrm{CoKer}(d^{k-1})\ar[r]&X^{k+1}\ar[r]&X^{k+2}\ar[r]&\cdots}).
\]
Similarly, the {\bf naive truncations} of $X$ are defined as follows
\[
X^{<k}:\quad(\xymatrix@C=10pt{\cdots\ar[r]&X^{k-2}\ar[r]&X^{k-1}\ar[r]&0\ar[r]&\cdots});
\]
\[
X^{>k}:\quad(\xymatrix@C=10pt{\cdots\ar[r]&0\ar[r]&X^{k+1}\ar[r]&X^{k+2}\ar[r]&\cdots}).
\]
Notice that there is a short exact sequence: 
\[
0\to X^{>k-1}\to X\to X^{<k}\to0.
\] 
Furthermore, $X$ is an exact complex if and only if both $\tau^{\leq k}X$ and $\tau^{\geq k+1}X$ are exact.

\begin{lem}\label{devissage_of_der}
Let $\A$ be an Abelian category and let $\sfD^b(\A)$, $\sfD^+(\A)$, $\sfD^-(\A)$ and $\sfD(\A)$ be its bounded, left-bounded, right-bounded and unbounded derived category, respectively. Suppose furthermore that these derived categories all have small hom-sets. Then, 
\begin{enumerate}
\item $\sfD^b(\A)$ satisfies the principle of d\'evissage with respect to $\A$, embedded in $\Ch^{b}(\A)$ as the class of complexes concentrated in degree $0$;
\item if $\A$ has exact coproducts (resp., products), then $\sfD^+(\A)$,  $\sfD^-(\A)$, and $\sfD(\A)$ satisfy the principle of infinite (dual) d\'evissage with respect to $\A$.
\end{enumerate}
\end{lem}
\begin{proof}
(1) is clear, since any bounded complex can be constructed inductively taking extensions of shifted copies of objects in $\A$. 

(2) We just deal with the principle of infinite d\'evissage since the infinite dual d\'evissage can be proved by dual arguments. Let us start with $X\in \Ch^+(\A)$, and consider the direct system $(\tau^{\leq k}X)_{k\in\N}$ in $\Ch^{b}(\A)$, whose direct limit (in $\Ch^+(\A)$) is isomorphic to $X$. There is a short exact sequence of the form:
\[
0\to \bigoplus_{k\in\N}\tau^{\leq k}X\to \bigoplus_{k\in\N}\tau^{\leq k}X\to \colim_{k\in\N}\tau^{\leq k}X(\cong X)\to 0.
\]
Since coproducts are exact in $\A$, the above coproducts are also coproducts in $\sfD^+(\A)$ and the above short exact sequence induces a triangle in $\sfD^+(\A)$. Hence, $X$ belongs to $\mathrm{Loc}(\sfD^b(\A))$ and, by part (1), $\mathrm{Loc}(\sfD^b(\A))=\mathrm{Loc}(\mathrm{thick}(\A))=\mathrm{Loc}(\A)$.
\\
Let  now $X\in \Ch^-(\A)$, and consider the direct system $(X^{>-k})_{k\in\N}$ in $\Ch^{b}(\A)$, whose direct limit (in $\Ch^-(\A)$) is isomorphic to $X$. There is a short exact sequence of the form:
\[
0\to \bigoplus_{k\in\N}X^{>-k}\to \bigoplus_{k\in\N}X^{>-k}\to \colim_{k\in\N}X^{>-k}(\cong X)\to 0.
\]
Since coproducts are exact in $\A$, the above coproducts are also coproducts in $\sfD^-(\A)$ and the above short exact sequence induces a triangle in $\sfD^-(\A)$. Hence, $X$ belongs to $\mathrm{Loc}(\Ch^b(\A))$ and, by part (1), $\mathrm{Loc}(\Ch^b(\A))=\mathrm{Loc}(\mathrm{thick}(\A))=\mathrm{Loc}(\A)$.
\\
Finally, suppose $X\in \Ch(\A)$. Then $X$ is an extension of a complex in $\Ch^{-}(\A)$ and one in $\Ch^{+}(\A)$, hence $X\in \mathrm{Loc}(\sfD^+(\A)\cup\sfD^{-}(\A))=\mathrm{Loc}(\A)$, by what we have just proved for half-bounded complexes.
\end{proof}

\subsection{Towers and telescopes of complexes}\label{complete_subsection_tower_tel_complexes}

Let us start introducing the following sequence of model categories on categories of uniformly left/right bounded complexes:

\begin{lem}\cite[Thm.\,5.2]{chacholski2017relative}\label{model_troncati}
Suppose $\A$ has enough injectives and let $n\in\N$. Then there is a model category $(\Ch^{\geq -n}(\A),\W^{\geq -n},\C^{\geq -n},\F^{\geq -n})$, where:
\begin{itemize}
\item $\W^{\geq -n}=\{\phi^\bullet:\phi^\bullet \text{ is a quasi-isomorphism}\}$;
\item $\C^{\geq -n}=\{\phi^\bullet:\phi^k \text{ is an epi with injective kernel, for all $k\geq-n$}\}$;
\item $\F^{\geq -n}=\{\phi^\bullet:\phi^k\text{ is a mono for all $k\geq -n$}\}$.
\end{itemize}
Furthermore, there is an adjunction
\begin{equation}\label{adj_n}
l_{n}:  \xymatrix{\Ch^{\geq -n-1}(\A)\ar@<4pt>[r]\ar@<-2pt>@{<-}[r]&\Ch^{\geq -n}(\A)} :r_n,
\end{equation}
where $l_{n}$ is the obvious inclusion while $r_n$ is the (smart) truncation functor. In this situation, $r_n$ preserves fibrations and acyclic fibrations.

Dually, if $\A$ has enough projectives, then for any $n\in\N$ there is a model category $(\Ch^{\leq n}(\A),\W^{\leq n},\C^{\leq n},\F^{\leq n})$, where:
\begin{itemize}
\item $\W^{\leq n}=\{\phi^\bullet:\phi^\bullet \text{ is a quasi-isomorphism}\}$;
\item $\C^{\leq n}=\{\phi^\bullet:\phi^k\text{ is an epi for all $k\leq n$}\}$;
\item $\F^{\leq n}=\{\phi^\bullet:\phi^k \text{ is a mono with projective coker, for all $k\leq n$}\}$.
\end{itemize}
Furthermore, there is an adjunction
\begin{equation}\label{adj_n}
l'_{n}:  \xymatrix{\Ch^{\leq n}(\A)\ar@<4pt>[r]\ar@<-2pt>@{<-}[r]&\Ch^{\leq n+1}(\A)} :r'_n,
\end{equation}
where $l'_{n}$ is the (smart) truncation functor and $r'_n$ is the  obvious inclusion.  In this situation, $l'_n$ preserves cofibrations and acyclic cofibrations.
\end{lem}

Lemma \ref{model_troncati} shows that the adjunctions $l_{n}\colon\Ch^{\geq -n-1}(\A)\rightleftarrows \Ch^{\geq -n}(\A)\colon r_n$ give a tower of model categories (see Definition \ref{tower_of_models}), while the adjunctions $l'_{n}\colon\Ch^{\leq n}(\A)\rightleftarrows \Ch^{\leq n+1}(\A)\colon r'_n$ give a telescope. Hence, we can construct a category of towers $\Tow(\A)$ and a category of telescopes $\Tel(\A)$. A typical object $X^{\bullet}_\bullet$ of $\Tow(\A)$ can be thought as a commutative diagram of the form
\[
\xymatrix{
\vdots\ar[d] & \vdots\ar[d] & \vdots\ar[d] & \vdots\ar[d] & \vdots\ar[d] & \vdots\ar[d] \\
0\ar[r] & X_{2}^{-2}\ar[d]\ar[r]^{d^{-2}_{2}} & X_{2}^{-1}\ar[d]^{t^{-1}_{2}} \ar[r]^{d^{-1}_{2}}& X_{2}^{0}\ar[d]^{t^{0}_{2}} \ar[r]^{d^{0}_{2}}& X_{2}^{1}\ar[d]^{t^{1}_{2}}\ar[r]^{d^{1}_{2}} & X_{2}^{2} \ar[d]^{t^{2}_{2}}\ar[r]^{d^{2}_{2}}& \dots \\
   & 0 \ar[r] & X_{1}^{-1}\ar[d] \ar[r]^{d^{-1}_{1}}& X_{1}^{0}\ar[d]^{t^{0}_{1}}\ar[r]^{d^{0}_{1}} & X_{1}^{1}\ar[d]^{t^{1}_{1}} \ar[r]^{d^{1}_{1}}& X_{1}^{2}\ar[d]^{t^{2}_{1}} \ar[r]^{d^{2}_{1}}& \dots \\
   &     & 0\ar[r]  & X_{0}^{0}\ar[r]^{d^{0}_{0}} & X_{0}^{1}\ar[r]^{d^{1}_{0}} & X_{0}^{2}\ar[r]^{d^{2}_{0}} & \dots 
}
\]
where $(X^\bullet_n,d_n^\bullet)$ is a cochain complex for all $n\in\N$, and $X^m_n=0$ for all $m<-n$. An analogous description holds for $\Tel(\A)$.

\begin{defn}
Let $\A$ be an Abelian complete category with enough injectives, and let $\M_\bullet=\{(\Ch^{\geq -n}(\A),\W^{\geq -n},\C^{\geq -n},\F^{\geq -n}):n\in\N\}$ be the tower of model categories defined in Lemma \ref{model_troncati}. We denote by 
\[
(\Tow(\A),\W_\Tow,\C_\Tow,\F_\Tow)
\] 
the induced model category constructed as in Proposition \ref{model_torre}. Furthermore, we denote by 
\[
\tow:\Ch(\A)\to \Tow(\A)
\] 
the so-called {\bf tower functor}, which sends a complex $X$ to the sequence of its successive truncations $\dots\to \tau^{\geq -n}X\to \tau^{\geq -n+1}X\to \dots\to \tau^{\geq0}X$ and acts on morphisms in the obvious way (see~\cite{chacholski2017relative}).

Dually, let $\A$ be an Abelian cocomplete category with enough projectives, and let $\M'_\bullet=\{(\Ch^{\leq n}(\A),\W^{\leq n},\C^{\leq n},\F^{\leq n}):n\in\N\}$ be the telescope of model categories defined in Lemma \ref{model_troncati}. We denote by 
\[
(\Tel(\A),\W_\Tow,\C_\Tow,\F_\Tow)
\] 
the induced model category constructed as in Proposition \ref{model_torre}. Furthermore, we denote by 
\[
\tel:\Ch(\A)\to \Tow(\A)
\] 
the so-called {\bf telescope functor}, which sends a complex $X$ to the sequence of its successive truncations $\tau^{\leq0}X\to \tau^{\leq 1}X\to \dots\to \tau^{\leq n}X\to \dots$ and acts on morphisms in the obvious way.
\end{defn}

In the same setting of the above definition, notice that the category $\Tow(\A)$ can be seen as a full subcategory of the category $\Ch(\A)^{\N^{\op}}$ of functors $\N^{\op}\to \Ch(\A)$ and so we can restrict the usual limit functor to obtain a functor $\lim_{\N^\op}\colon\Tow(\A)\to \Ch(\A)$ which is a right adjoint for the tower functor. If we endow $\Ch(\A)$ with the class of weak equivalences given by the quasi-isomorphisms of complexes, we would like to verify that $(\lim_{\N^{\op}},\tow)$ is a right model approximation. A similar question holds for the dual adjunction $(\tel, \colim_\N)$. In fact, this is not true in general. A sufficient condition for this adjunction to be an approximation is provided by the following definition:

\begin{defn}\cite{Roos}
If $\A$ has enough injectives, we say that it is (Ab.$4^*$)-$k$ if, for any
countable family of objects $\{B_n\}_{n\in\N}$ and any choice of injective resolutions $B_n\to E_n$ with $E_n\in\Ch^{\geq0}(\A)$, we have that $H^h(\prod_{n\in\N}E_n)=0$ for all $h>k$.

Dually, if $\A$ has enough projectives, it is (Ab.$4$)-$k$ if, for any
countable family of objects $\{B_n\}_{n\in\N}$ and any choice of projective resolutions $B_n\to E_n$ with $E_n\in\Ch^{\leq0}(\A)$, we have that $H^h(\coprod_{n\in\N}E_n)=0$ for all $h<-k$.
\end{defn}

\begin{prop}\label{prop_tow_tel_approx}\cite[Thm.\,6.4]{chacholski2017relative}
Given $k\in\N$, suppose that $\A$ has enough injectives and it is (Ab.$4^*$)-$k$. Then the following adjunction is a right model approximation, 
\begin{equation}\label{tower_adj_to}
\tow:  \xymatrix{(\Ch(\A),\W)\ar@<4pt>[r]\ar@<-2pt>@{<-}[r]&(\Tow(\A),\W_\Tow,\C_\Tow,\F_\Tow)} :\lim_{\N^{\op}}.
\end{equation}
Dually, if $\A$ has enough projectives and it is (Ab.$4$)-$k$, then the following adjunction is a left model approximation, 
\begin{equation}\label{tel_adj_to}
\colim_\N:  \xymatrix{(\Tel(\A),\W_\Tel,\C_\Tel,\F_\Tel)\ar@<4pt>[r]\ar@<-2pt>@{<-}[r]& (\Ch(\A),\W)}:\tel.
\end{equation}
\end{prop}

Hence, if $\A$ satisfies one of the two following sets of hypotheses
\begin{itemize}
\item enough injectives and (Ab.$4^*$)-$k$ for some $k\in\N$,
\item enough projectives and (Ab.$4$)-$k$ for some $k\in\N$,
\end{itemize}
a consequence of the above Proposition \ref{prop_tow_tel_approx} and Theorem \ref{derivator_from_approximation} is that the derived category $\sfD(\A^I)$ has small hom-sets for any small category $I\in \Dia$. In particular, we can safely define a prederivator 
\[
\sfD_\A\colon \Dia^{\op}\to \CAT,\qquad I\mapsto \sfD(\A^I),
\]
enhancing the derived category of $\A$. Notice that all the images of $\sfD_\A$ are canonically triangulated but, without more hypotheses on $\A$, $\sfD_\A$ is just a pre-derivator. Anyway, one can verify that the homotopy Kan extensions involved in the definition of cones, fibers, suspensions and loops, always exist in $\sfD_\A$ and that the definitions given classically coincide with those in Subsection \ref{prelim_stab_der}. Let us remark that one can safely define the above prederivator under even more general hypotheses, as explained in~\cite{saneblidze2007derived}.

%\medskip
%Denote by $\t_\A=(\sfD(\A)^{\leq0},\sfD(\A)^{\geq0})$ the {\bf natural $t$-structure} on $\sfD(\A)$, given by the usual (smart) left and right truncation of complexes. Notice that this $t$-structure lifts to $\sfD(\A^I)$ for any $I\in \Dia$, just because, defining $(\t_\A)_I$ as in the statement of Proposition \ref{lift_tstructure_general}, we get $(\t_\A)_I=\t_{\A^I}$ is the natural $t$-structure on the derived category $\sfD(\A^I)$. Hence, we can define sub-prederivators $\sfD_\A^{\leq0}$ and $\sfD_\A^{\geq0}$, which are clearly a coreflection and a reflection, respectively, of the prederivator $\sfD_\A$.
%

\subsection{Two-sided completion of the derived category}\label{two_sided_derived_completion}
Let us define the following categories with weak equivalences: 
\begin{enumerate}
\item the full subcategory $\widehat{\Ch(\A)}\subseteq \Ch(\A)^{\N^{\op}}$ spanned by those diagrams $X\colon \N^{\op}\to\Ch(\A)(I)$ such that, for all $n\in\N$,
\begin{itemize}
\item $X(n)\in \Ch^{\geq-n}(\A)$;
\item the map $X(n)\leftarrow X(n+1)$ induces a q.i. $X(n)\leftarrow \tau^{\geq-n}X(n+1)$.
\end{itemize}
Consider also the class of weak equivalences 
\[
\widehat\W:=\{\phi\colon X\Rightarrow X':\text{$\phi_{n}\colon X(n)\to X'(n)$ is a q.i.},\, \forall n\in\N\};
\] 
\item the full subcategory $\widecheck{\Ch(\A)}\subseteq \Ch(\A)^{\N}$ spanned by those diagrams $Y\colon \N\to\Ch(\A)(I)$ such that, for all $n\in\N$,
\begin{itemize}
\item $Y(n)\in \Ch^{\leq n}(\A)$;
\item the map $Y(n)\to Y(n+1)$ induces a q.i. $Y(n)\to \tau^{\leq n}Y(n+1)$.
\end{itemize}
Consider also the class of weak equivalences 
\[
\widecheck\W:=\{\phi\colon Y\Rightarrow Y':\text{$\phi_{n}\colon Y(n)\to Y'(n)$ is a q.i.},\, \forall n\in\N\};
\] 
\item the full subcategory $\Ch(\A)^{\diamond}\subseteq \Ch(\A)^{\N\times \N^{\op}}$ spanned by those $X\colon \N\times \N^{\op}\to \Ch(\A)$ such that, for all $n$ and $m\in\N$,
\begin{itemize}
\item $X(m,-)\in \widehat{\Ch(\A)}$;
\item $X(-,n)\in \widecheck{\Ch(\A)}$.
\end{itemize}
Consider also the class of weak equivalences 
\[
\W_{\diamond}:=\{\phi\colon X\Rightarrow X':\text{$\phi_{(m,n)}\colon X(m,n)\to X'(m,n)$ is a q.i.},\, \forall n,\, m\in\N\}.
\] 
\end{enumerate}

There are obvious pairs of adjoint morphisms of prederivators:
\[
\colim_\N:  \xymatrix{\Ch_\A^{\diamond}\ar@<4pt>[r]\ar@<-2pt>@{<-}[r]&\widehat{\Ch_\A}} :\widehat\tel\qquad\text{and}\qquad\widecheck\tow:  \xymatrix{\widecheck{\Ch_\A}\ar@<4pt>[r]\ar@<-2pt>@{<-}[r]&\Ch_\A^{\diamond}}:\lim_{\N^{\op}}
\]
where, given $Y\in \widehat{\Ch_\A}(I)$, $\widehat\tel(Y)(m,n):=\tau^{\leq m}Y(n)$ while, given $X\in \widecheck{\Ch_\A}(I)$, $\widecheck\tow(X)(m,n):=\tau^{\geq -n}X(m)$. 

Remember that, given two categories with weak equivalences $(\C_1,\W_1)$ and $(\C_2,\W_2)$, a functor $F\colon \C_1\to \C_2$ is said to be {\bf homotopically invariant} if $F(w)\in \W_2$ for all $w\in \W_1$. In this case, the functor $F$ induces a functor $\C_1[\W_1^{-1}]\to \C_2[\W_2^{-1}]$ that acts like $F$ on objects.

\begin{prop}\label{filt_is_an_eq}
Let $k\in\N$ and suppose that $\A$ has enough injectives, it is (Ab.$4^*$)-$k$ and it has exact coproducts. 
Then, the localizations $\widehat {\Ch(\A)}[\widehat\W^{-1}]$ and $\Ch(\A)^{\diamond}[\W_{\diamond}^{-1}]$ exist (i.e., they are locally small). Furthermore, the functors $\colim_\N$ and $\widehat\tel$ are homotopically invariant and they induce an equivalence $\sfD(\A)\cong\widehat {\Ch(\A)}[\widehat\W^{-1}]\cong\Ch(\A)^{\diamond}[\W_{\diamond}^{-1}]$.

Dually, if $\A$ has enough projectives, it is (Ab.$4$)-$k$ and it has exact products,
then, the localizations $\widecheck {\Ch(\A)}[\widecheck\W^{-1}]$ and $\Ch(\A)^{\diamond}[\W_{\diamond}^{-1}]$ exist (i.e., they are locally small). Furthermore, the functors $\lim_{\N^{\op}}$ and $\widecheck\tow$ are homotopically invariant and they induce an equivalence $\sfD(\A)\cong\widecheck {\Ch(\A)}[\widecheck\W^{-1}]\cong\Ch(\A)^{\diamond}[\W_{\diamond}^{-1}]$.
\end{prop}
\begin{proof}
We prove just the first half of the statement, as the rest of the proof is completely dual. Consider the right model approximation 
\begin{equation}\label{tower_adj_to}
\tow:  \xymatrix{(\Ch(\A),\W)\ar@<4pt>[r]\ar@<-2pt>@{<-}[r]&(\Tow(\A),\W_\Tow,\C_\Tow,\F_\Tow)} :\lim_{\N^{\op}}
\end{equation}
constructed in Proposition \ref{prop_tow_tel_approx}. Then, $\sfD_\A(I)$ is equivalent to its  image via (the derived functor of) $\tow$ in the homotopy category $\Ho_{\W_\Tow}(\Tow(\A^I))$; this essential image is clearly equivalent to $\widehat {\Ch(\A^I)}[\widehat\W^{-1}]$, so we get an equivalence:
\[
\tow:  \xymatrix@C=12pt{\sfD(\A^I)\ar@<3.5pt>[rr]\ar@<-3.5pt>@{<-}[rr]&&\widehat {\Ch(\A^I)}[\widehat\W^{-1}]} :\holim_{\N^{\op}}.
\]
Using the same argument of~\cite[\href{http://stacks.math.columbia.edu/tag/05RW}{Tag 05RW}]{stacks-project} we can embed $\Ch^{\diamond}(\A^I)[\W_\diamond^{-1}]$ into $\Ho(\Tow(\A)^\N)$, so this category is locally small. Furthermore, we can consider the following equivalence of categories
\[
\widehat\tel:  \xymatrix@C=12pt{\widehat{\Ch(\A)}\ar@<3.5pt>[rr]\ar@<-3.5pt>@{<-}[rr]&&\Ch(\A)^{\diamond}} :\colim_{\N}
\]
Now notice that $\widehat\tel$ is exact, so it is homotopically invariant and, since $\A$ is (Ab.4), one can use an argument dual to~\cite[Remark 2.3]{BN}
to show that the above restriction of the functor $\colim_\N$ is homotopically invariant. Hence, the above equivalence of categories induces an equivalence of the  localizations. 
\end{proof}

\newpage
\section{Lifting $t$-structures algong diagram functors}\label{Sec:lifting}

In this section we begin our study of $t$-structures on a strong and stable derivator $\D$. Using one of the main results of~\cite{SSV}, we show that a $t$-structure on the base $\D(\bbone)$ of our derivator naturally induces $t$-structures on each of the categories $\D(I)$. Furthermore, $t$-structures in $\D(\bbone)$ are in bijection with suitable families of co/localizations of $\D$. We also give an extremely short and self-contained argument for the abelianity of the heart of a $t$-structure. In the last subsection we show how a $t$-structure on $\D(\bbone)$ induces two special $t$-structures on $\D(\N)$ and $\D(\N^{\op})$ which are different from the ones obtained in~\cite{SSV}.

\medskip\noindent
{\bf Setting for Section \ref{Sec:lifting}.}
We fix through this section a category of diagrams $\Dia$ and a strong and stable derivator $\D\colon \Dia^\op\to \CAT$. In Subsections \ref{subs_epiv} and \ref{subs_other_t} we will need to assume that $\N\in\Dia$.

\subsection{$t$-Structures on stable derivators}
\label{sec:loc-coloc}

Let us start with the following definition

\begin{defn}
A $t$-\textbf{structure} $\t=(\D(\bbone)^{\leq0},\D(\bbone)^{\geq0})$ on $\D$ is, by definition, a $t$-structure on the triangulated category $\D(\bbone)$. 
\end{defn}

The following proposition is one of the main results of~\cite{SSV}:
\begin{prop}\label{lift_tstructure_general}
Let $\t=(\D(\bbone)^{\leq0},\D(\bbone)^{\geq0})$ be a $t$-structure on $\D$, and let $I\in \Dia$. Letting 
\begin{align*}
\D(I)^{\leq0}&:=\{\mathscr X\in \D(I):\mathscr X_i\in \D(\bbone)^{\leq0},\, \forall i\in I\},\\
\D(I)^{\geq0}&:=\{\mathscr X\in \D(I):\mathscr X_i\in \D(\bbone)^{\geq0},\, \forall i\in I\},
\end{align*}
$\t_I:=(\D(I)^{\leq0},\D(I)^{\geq0})$ is a $t$-structure on $\D(I)$. Furthermore, 
\[
\dia_I\colon \D(I)\to \D(\bbone)^I
\] 
induces an equivalence $\A_{I}\cong \A^I$ between the heart $\A_{I}$ of $\t_I$ and the category $\A^I$ of diagrams of shape $I$ in the heart $\A$ of $\t$.
\end{prop}

\begin{cor}\label{lift_t_struc}
Given a $t$-structure  $\t=(\D(\bbone)^{\leq0},\D(\bbone)^{\geq0})$ on $\D$, 
\begin{enumerate}
\item the prederivator $\D^{\le 0}\colon I\mapsto \D(I)^{\le 0}$ is a coreflection of $\D$;
\item the prederivator $\D^{\ge 0}\colon I\mapsto \D(I)^{\ge 0}$ is a reflection of $\D$;
\item the prederivator $\D^{\heartsuit}\colon I\mapsto\D(I)^{\le 0}\cap \D(I)^{\ge 0}$ can be seen either as a reflection of $\D^{\le 0}$, or as a coreflection of $\D^{\ge 0}$.
\end{enumerate}
In particular, $\D^{\le 0}$, $\D^{\ge 0}$, and $\D^{\heartsuit}$ are pointed, strong derivators of type $\Dia$. 
\end{cor}
\begin{proof}
By Proposition \ref{lift_tstructure_general}, $\t_I$ is a $t$-structure on $\D(I)$ and $\dia_I\colon \D(I)\to \D(\bbone)^{I}$ induces an equivalence of categories $\D^{\heartsuit}(I)\cong \D^{\heartsuit}(\bbone)^I$, for all $I$ in $\Dia$. The proofs of (1), (2) and (3) are very similar so we just verify (1) and leave the other two statements to the reader. First of all we need to verify that $\D^{\le 0}\colon I\mapsto \D^{\le 0}(I)$ and $\D^{\ge 0}\colon I\mapsto \D^{\ge 0}(I)$ are sub-prederivators of $\D$. For that one should verify that, given $u\colon I\to J$ in $\Dia$,
\[\begin{split}
u^*\mathscr X\in \D^{\leq0}(I)\quad\text{ for all $\mathscr X\in \D^{\leq 0}(J)$}\\
u^*\mathscr Y\in \D^{\geq0}(I)\quad\text{ for all $\mathscr Y\in \D^{\geq 0}(J)$}.\\
\end{split}\]
But this is true since $\mathscr X\in \D^{\leq 0}(J)$ if and only if $\mathscr X_j\in\D^{\leq0}(\bbone)$, while $u^*\mathscr X\in \D^{\leq0}(I)$ if and only if $i^*(u^*\mathscr X)=(u(i))^*\mathscr X\in \D^{\leq0}(\bbone)$ for all $i\in I$. The proof of the second statement is analogous. 

To conclude the proof of (1) we need to verify that the obvious embedding $\D^{\le0}\to \D$ has a right adjoint. Since each $\t_I$ is a $t$-structure, each component $\D^{\le0}(I)\to \D(I)$ has a right adjoint so, by~\cite[Lem.2.10]{Moritz} there is a unique way to organize these right adjoints into a lax morphism of prederivators $\D\to \D^{\le0}$. Consider a functor $u\colon I\to J$ in $\Dia$; given $\mathscr X\in \D^{\leq 0}(I)$, let us show that $u_!\mathscr X\in \D^{\leq0}(J)$, equivalently, we can show that $\D(J)(u_!\mathscr X,\mathscr Y)=0$ for all $\mathscr Y\in  \D^{>0}(J)$, but this is clear since $\D(J)(u_!\mathscr X,\mathscr Y)=\D(I)(\mathscr X,u^*\mathscr Y)=0$, as we have already seen that $u^*\mathscr Y\in \D^{>0}(I)$. Hence, using the same argument of the proof of~\cite[Prop.2.11]{Moritz}, one shows that the morphism $\D\to \D^{\le0}$ is strict.

The last statement is a consequence of Lemma \ref{loc_is_der}.
\end{proof}

Recall that a derivator $\D\colon \Dia^\op\to \CAT$ is said to be {\bf discrete} if, for any $I\in \Dia$, the diagram functor $\dia_I\colon \D(I)\to \D(\bbone)^I$ is an equivalence. In particular, this means that the homotopy (co)limits and homotopy Kan extensions in such a derivator coincide with the classical (i.e., categorical) (co)limits and Kan extensions. A discrete derivator is {\bf Abelian} if $\D(\bbone)$ is an Abelian category. 

\begin{cor}\label{the_heart_is_Abelian}
Given a $t$-structure  $\t=(\D(\bbone)^{\leq0},\D(\bbone)^{\geq0})$ on $\D$, $\D^{\heartsuit}$ is a discrete, Abelian derivator.
\end{cor}
\begin{proof}
$\D^{\heartsuit}$ is a derivator by Corollary \ref{lift_t_struc} and it is discrete by Proposition \ref{lift_tstructure_general}. Hence we have just to verify that $\D^\heartsuit(\bbone)$ is an Abelian category. In fact, we already know that it is an additive and $\Dia$-bicomplete category, so it is enough to show that any monomorphism (epimorphism) is a (co)kernel.

Let us start with a morphism $\phi\colon X\to Y$ in $\D^\heartsuit(\bbone)$. By Lemma \ref{loc_is_der}, the kernel of $\phi$ can be constructed as follows: first we notice that such kernel is the limit (which is the same as ``homotopy colimit" since $\D^{\heartsuit}$ is discrete) of the following diagram
\[
\xymatrix{
&0\ar[d]\\
X\ar[r]^\phi&Y.
}
\]
This limit can be computed first as a homotopy limit in $\D^{\geq0}$ or, equivalently, in $\D$, and then coreflected in $\D^{\heartsuit}$ (see Lemma \ref{loc_is_der}). This shows that $\Ker(\phi)$ is the coreflection of $K$, where $K$ is the fiber of $\phi$, that is, there is a triangle in $\D(\bbone)$:
\begin{equation}\label{triangle_for_normality}
K\to X\to Y\to \Sigma K.
\end{equation}
This shows that $\phi$ is a monomorphism if and only if $\Ker(\phi)=0$ if and only if $K\in \D^{\geq1}(\bbone)$. This means exactly that in the triangle \eqref{triangle_for_normality}, $X, Y$ and $\Sigma K\in \D^\heartsuit(\bbone)$. By the above discussion, we obtain that $\phi\colon X\to Y$ represents a kernel for $Y\to \Sigma K$ in $\D^\heartsuit(\bbone)$. In particular any monomorphism in $\D^\heartsuit(\bbone)$ is a kernel; one verifies similarly that all the epimorphisms are cokernels.
\end{proof}

We have seen above that a $t$-structure on the base of a stable derivator can be used to produce a reflection and a coreflection. In the following theorem we prove that, in fact, all the extension-closed (co)reflections arise this way from a $t$-structure. 

\begin{thm}\label{t_structures_are_co_localizations}
There are bijections among the following classes:
\begin{enumerate}
\item $t$-structures in $\D(\bbone)$; 
\item extension-closed reflective sub-derivators of $\D$;
\item extension-closed coreflective sub-derivators of $\D$.
\end{enumerate}
\end{thm}
\begin{proof}
By Proposition \ref{lift_tstructure_general}, a $t$-structure $\t=(\D(\bbone)^{\leq0},\D(\bbone)^{\geq0})$ on $\D(\bbone)$ can be lifted to $t$-structures $\t_I=(\D(I)^{\leq0},\D(I)^{\geq0})$ in $\D(I)$, for any $I$ in $\Dia$. Then, the derivator $\D^{\le 0}\colon \Dia^\op\to \CAT$, such that $\D^{\le 0}(I) \coloneqq \D(I)^{\le 0}$, is coreflective in $\D$ and, similarly, the derivator $\D^{\ge 0}\colon \Dia^\op\to \CAT$, such that $\D^{\ge 0}(I) \coloneqq \D(I)^{\ge 0}$, is reflective in $\D$ (see Corollary \ref{lift_t_struc}). 

On the other hand, given an extension-closed coreflective sub-derivator $\sU\to \D$, it is clear that $\sU(\bbone)$ is an aisle in $\D(\bbone)$ (use~\cite[Prop.\,3.21]{Moritz}). By Proposition \ref{lift_tstructure_general}, we can define a coreflective sub-derivator $\sU'$ of $\D$ letting
\[
\sU'(I) \coloneqq \{\mathscr X\in \D(I):\mathscr X_i\in \sU(\bbone),\forall i\in I\}.
\]
Thus, $\sV(\bbone) \coloneqq \Sigma^{-1}(\sU(\bbone)^{\perp})$ is a co-aisle and, again by Proposition \ref{lift_tstructure_general}, we can define a reflective sub-derivator $\sV$ of $\D$ letting
\[
\sV(I) \coloneqq \{\mathscr X\in \D(I):\mathscr X_i\in \sV(\bbone),\forall i\in I\}.
\]
Now, for any $I\in \Dia$, we have two t-structures: $(\sU(I),\Sigma^{-1}(\sU(I)^{\perp}))$ and $(\sU'(I),\sV(I))$. But clearly, $\sU(I)\subseteq \sU'(I)$ and $\Sigma^{-1}(\sU(I)^{\perp})\subseteq \sV(I)$, which implies that these two $t$-structures coincide. In particular, $\sU=\sU'$.
Hence, $\sU$ is completely determined by $\sU(\bbone)$, so the assignment $\sU\mapsto (\sU(\bbone),\Sigma^{-1}(\sU(\bbone)^{\perp}))$ is a bijection between the classes described in parts (3) and (1) of the statement. The bijection between (2) and (1) is proved similarly.
\end{proof}

Due to the above theorem, from now on, we will indifferently denote a $t$-structure $\t$ on a stable derivator $\D$, either as aisle and coaisle $(\D(\bbone)^{\leq0}, \D(\bbone)^{\geq0})$ in $\D(\bbone)$, or as a pair of coreflective and reflective subderivators $(\D^{\leq0},\D^{\geq0})$ of $\D$.

\subsection{Epivalence of certain diagram functors}\label{subs_epiv}

The following lemma is a refinement of ~\cite[Thm.\,3.1]{SSV} in the case for the special small category $\N$. We will just give a sketch of the proof since similar statements have already appeared in~\cite{tderiv} (for finite ordinals) and in the appendix of~\cite{KN}. 

\begin{lem}\label{lem_epivalence_N_op}
The diagram functor
\[
\dia_{\N}\colon \D(\N)\to \D(\bbone)^{\N}\quad \text{(resp., $\dia_{\N^{\op}}\colon \D(\N^{\op})\to \D(\bbone)^{\N^{\op}}$)}
\]
is full and essentially surjective. Furthermore, given $\mathscr X \in \D(\N)$ (resp., $\D(\N^{\op})$), there is a triangle
\begin{align*}
\bigoplus_{n\in \N}(n+1)_!\mathscr X_n\to \bigoplus_{n\in \N}n_!\mathscr X_n\to \mathscr X\to \Sigma \bigoplus_{n\in \N}(n+1)_!\mathscr X_n\\
\left (\text{resp., $\mathscr Y\to \prod_{n\in \N}n_*\mathscr Y_n\to \prod_{n\in \N}(n+1)_*\mathscr Y_n\to \Sigma \mathscr Y$}\right).
\end{align*}
Finally, given a second object  $\mathscr Y\in \D(\N)$ the diagram functor induces an isomorphism
\begin{align*}
\D(\N)(\mathscr X,\mathscr Y)\cong \D(\bbone)^{\N}(\dia_{\N}\mathscr X,\dia_{\N}\mathscr Y)\\
(\text{resp., $\D(\N^{\op})(\mathscr X,\mathscr Y)\cong \D(\bbone)^{\N^{\op}}(\dia_{\N^{\op}}\mathscr X,\dia_{\N^{\op}}\mathscr Y)$})
\end{align*}
provided the following map is surjective (where $\mathcal D:=\D(\bbone)$):
\[\begin{matrix}
\prod_{n\in\N}\mathcal D(\Sigma \mathscr X_n,\mathscr Y_n)&\to& \prod_{n\in\N}\mathcal D(\Sigma \mathscr X_n,\mathscr Y_{n+1})\\ (\phi_n)_\N&\mapsto& (d_{\mathscr Y}^n\phi_n-\phi_{n+1}\Sigma d_{\mathscr X}^n)_\N
\end{matrix}\]
\[\begin{pmatrix}
\prod_{n\in\N}\mathcal D(\Sigma \mathscr X_n,\mathscr Y_n)&\to& \prod_{n\in\N}\mathcal D(\Sigma \mathscr X_{n+1},\mathscr Y_{n})\\ (\phi_n)_\N&\mapsto& (d_{\mathscr Y}^{n+1}\phi_{n+1}-\phi_{n}\Sigma d_{\mathscr X}^{n+1})_\N
\end{pmatrix}.\]
This happens, for example, when the Toda condition $\D(\bbone)(\Sigma \mathscr X_n,\mathscr Y_{n+1})=0$ (resp., $\D(\bbone)(\Sigma \mathscr X_{n+1},\mathscr Y_{n})=0$), is verified for all $n\in \N$.
\end{lem}
\begin{proof}[Sketch of the proof.]
Let $X\in \D(\bbone)^{\N}$. For any $n\in\N$, consider the following maps
\[
\alpha_n\colon (n+1)_!X_n\to n_{!}X_n \qquad \beta_n\colon (n+1)_!X_n\to (n+1)_{!}X_{(n+1)},
\]
where $\alpha_n$ is the image of the identity via the following series of natural isomorphism $\D(\bbone)(X_n,X_n)\cong \D(\bbone)(X_n,(n+1)^*n_!X_n) \cong\D(\N)((n+1)_!X_n, n_{!}X_n)$, while $\beta_n$ is induced by the unique natural transformation $b_n\colon (\bbone\overset{n}\to \N)\Rightarrow(\bbone\overset{n+1}\to \N)$ (hence, $\beta_n=(n+1)_{!}b_n^*$). Then define a morphism in $\D(\N)$
\[
\Phi_X\colon \bigoplus_{n\in\N} (n+1)_!X_n\to \bigoplus_{n\in\N} n_{!}X_n
\]
where the component of $\Phi_X$ relative to $(n+1)_!X_n$ is exactly 
\[
(\beta_n,-\alpha_n)^t\colon (n+1)_!X_n\to n_{!}X_n\oplus (n+1)_{!}X_{(n+1)}.
\]
Letting $\mathscr X$ be the cone of $\Phi_X$, one can verify that $\dia_\N(\mathscr X)\cong X$. Furthermore, given $\mathscr Y\in \D(\N)$, we can find a presentation of $\mathscr Y$ as the one above and, applying the functor $(-,\mathscr Z):=\D(\N)(-,\mathscr Z)$ we get a long exact sequence
\[
\xymatrix@C=16pt@R=2pt{
\cdots\ar[r]&\prod_{n\in\N} \D(\bbone)(\Sigma \mathscr Y_n,\mathscr Z_n)\ar[r]^(.45){(*)}&\prod_{n\in\N} \D(\bbone)(\Sigma \mathscr Y_n,\mathscr Z_{n+1})\ar[r]&(\mathscr Y,\mathscr Z)\ar[r]&\\
\ar[r]&\prod_{n\in\N} \D(\bbone)(\mathscr Y_n,\mathscr Z_n)\ar[r]^(.45){(**)}&\prod_{i\in\N} \D(\bbone)(\mathscr Y_n,\mathscr Z_{n+1})\ar[r]&\cdots
}
\]
where the kernel of $(**)$ is $\D(\bbone)^{\N}(\dia_{\N}\mathscr Y,\dia_{\N}\mathscr Z)$. This shows that the map $(\mathscr Y,\mathscr Z)\to \D(\bbone)^{\N}(\dia_{\N}\mathscr Y,\dia_{\N}\mathscr Z)$ is always surjective and that, when the map $(*)$ is surjective (so for example when $\D(\bbone)(\Sigma \mathscr Y_n,\mathscr Z_{n+1})=0$ for all $n\in\N$), then it is also injective.
\\
To prove the statement for $\N^{\op}$ one uses the same argument with a ``dual" resolution. Indeed, given $Y\in \D(\bbone)^{\N^{\op}}$ one can construct a map
\[
\Psi_Y\colon \prod_{n\in \N}n_*Y_n\to \prod_{n\in \N}(n+1)_*Y_n.
\]
Then, calling $\mathscr Y$ the fiber of the above map, $\dia_{\N^{\op}}(\mathscr Y)\cong Y$. Furthermore, given $\mathscr X\in \D(\N^{\op})$, we can find a presentation of $\mathscr Y$ as the one above. Applying the functor $(\mathscr Z,-):=\D(\N^{\op})(\mathscr Z,-)$ we get a long exact sequence with which we can conclude analogously to the first half of the proof.
\end{proof}

\subsection{$t$-structures on $\D(\N)$ and $\D(\N^{\op})$}\label{subs_other_t}
Let $\t=(\D^{\leq0},\D^{\geq0})$ be a $t$-structure on $\D$. By Proposition \ref{lift_tstructure_general}, $\t_\N=(\D^{\leq0}(\N),\D^{\geq0}(\N))$ (resp., $\t_{\N^{\op}}=(\D^{\leq0}(\N^{\op}),\D^{\geq0}(\N^{\op}))$) is a $t$-structure on the triangulated category $\D(\N)$ (resp., $\D(\N^{\op})$). In the following lemmas we introduce a different kind of $t$-structure on the same triangulated categories:

\begin{lem}\label{stair_t_structure_op}
Consider the following full subcategory of $\D(\N^{\op})$:
\begin{align*}
D_\tow^{\leq 0}&:=\{\mathscr X\in \D(\N^{\op}):\mathscr X_n\in \D^{\leq -n}(\bbone)\},\\
D_\tow^{\geq 0}&:=\{\mathscr X\in \D(\N^{\op}):\mathscr X_n\in \D^{\geq -n}(\bbone)\}.
\end{align*}
Then, $\t_{\tow}:=(D_\tow^{\leq0},D_\tow^{\geq 0})$ is a $t$-structure in $\D(\N^{\op})$. 
\end{lem} 
\begin{proof}
It is clear that $D_\tow^{\leq0}$ and $D_\tow^{\geq 0}$ have the required closure properties. Let now $\mathscr X\in D_\tow^{\leq0}$ (i.e., $\mathscr X_n\in \D(\bbone)^{\leq-n}$, for all $n\in\N$) and $\mathscr Y\in D_\tow^{\geq 1}$ (i.e., $\mathscr Y_n\in \D(\bbone)^{\geq-n+1}$, for all $n\in \N$). Then, $\D(\bbone)(\Sigma \mathscr X_{n},\mathscr Y_{n+1})=0$ since $\Sigma \mathscr X_n\in \Sigma \D(\bbone)^{\leq-n}=\D(\bbone)^{\leq-n-1}$ and $\mathscr Y_{n+1}\in \D(\bbone)^{\geq-n}$ so, by Lemma \ref{lem_epivalence_N_op}, 
\[
\D(\N^{\op})(\mathscr X,\mathscr Y)\cong\D(\bbone)^{\N^{\op}}(\dia_{\N^{\op}}\mathscr X,\dia_{\N^{\op}}\mathscr Y).
\] 
On the other hand, 
\[
\D(\bbone)^{\N^{\op}}(\dia_{\N^{\op}}\mathscr X,\dia_{\N^{\op}}\mathscr Y)\subseteq \prod_{i\in \N}\D(\bbone)(\mathscr X_n,\mathscr Y_n)=0,
\] 
since $\mathscr X_n\in \D(\bbone)^{\leq-n}$ and $\mathscr Y_n\in \D(\bbone)^{\geq-n+1}$. This shows that $\D(\N^{\op})(\mathscr X,\mathscr Y)=0$.
Finally, let $\mathscr X\in \D(\N^{\op})$ and consider its diagram $X:=\dia_{\N^{\op}}\mathscr X\in \D(\bbone)^{\N^{\op}}$. Then there is an object $X^{\leq 0}\in \D(\bbone)^{\N^{\op}}$ and a morphism $\phi\colon X^{\leq0}\to X$ with the property that, for any $n\in\N$, $X^{\leq0}(n)\in \D^{\leq -n}(\bbone)$ and the cone of $\phi_n$ belongs to $\D^{\geq -n+1}(\bbone)$. Using that $\dia_{\N^{\op}}$ is full and essentially surjective, one can lift $\phi$ to a morphism $\Phi\colon \mathscr X^{\leq0}\to \mathscr X$. Completing this morphism to a triangle
\[
\mathscr X^{\leq 0}\to \mathscr X\to \mathscr Y\to \Sigma \mathscr X^{\leq 0}
\]
we obtain that, by construction, $\mathscr X^{\leq 0}\in D_\tow^{\leq0}$ and $\mathscr Y\in D_\tow^{\geq1}$.
\end{proof}

We omit the proof of the following lemma since it is completely dual to that of the above lemma.

\begin{lem}\label{stair_t_structure}
Consider the following full subcategory of $\D(\N)$:
\begin{align*}
D_\tel^{\leq 0}&:=\{\mathscr X\in \D(\N):\mathscr X_n\in \D^{\leq n}(\bbone)\},\\
D_\tel^{\geq 0}&:=\{\mathscr X\in \D(\N):\mathscr X_n\in \D^{\geq n}(\bbone)\}.
\end{align*}
Then, $\t_{\tel}:=(D_\tel^{\leq0},D_\tel^{\geq 0})$ is a $t$-structure in $\D(\N)$. 
\end{lem}

\newpage
\section{The bounded realization functor}\label{Sec:Beilinson}

The construction of realization functors (at least in the bounded context) is classically done using the filtered derived category and, more generally, $f$-categories, see~\cite{BBD,Beilinson,Jorge_e_Chrisostomos}. For a given derivator $\D$, one can define a ``bounded filtered prederivator" $\FD$ such that $\FD(\bbone)$ has a natural $f$-structure, whose core is equivalent to $\D(\bbone)$. As a consequence, any $t$-structure on $\D(\bbone)$ can be lifted to a compatible $t$-structure on $\FD(\bbone)$. In this section we give a direct and detailed proof of this lifting property for $t$-structures from $\D$ to $\FD$ in the setting of stable derivators. Moreover, we give a construction of a morphism of prederivators
\[
\real^b_\t\colon \sfD_\A^b\to \D,
\]
where $\A$ is the heart of a given $t$-structure on $\D$. In the final part of the section we show that this morphism is equivalent to the one we obtain via $f$-categories.

\medskip\noindent
\textbf{Setting for Section \ref{Sec:Beilinson}}.
Fix through this section a category of diagrams such that $\N\in\Dia$, a strong and stable derivator $\D\colon \Dia^\op\to \CAT$, a $t$-structure $\t=(\D^{\leq0},\D^{\geq0})$ on $\D$, and let $\A=\D^{\leq0}(\bbone)\cap \D^{\geq0}(\bbone)$ be its heart. Just for the final Proposition \ref{real_is_tot_prop}, we will need to assume that $\N\in\Dia$.

\subsection{The bounded filtered prederivator $\FD$}
Let us start with the following definition:

\begin{defn}\label{filtered_der_cat}
Given $n\in\Z$, let $[n\to n+1]\colon \bbtwo\to \Z$ be the functor sending $0\mapsto n$ and $1\mapsto n+1$. Define
\[
\xymatrix{
\gr_n\colon \D^\Z\ar[rr]^(.58){[n\to n+1]^*}&&\D^\bbtwo\ar[rr]^{\mathrm{cone}}&& \D.
}
\]
Given $ \mathscr X\in \D^\Z(I)$
\begin{itemize}
\item $\mathscr X$ is  {\bf bounded above} if there exists $a\in\Z$ such that $\mathscr X_{a-n}=0$ for all $n\in \N$. If $\mathscr X$ is bounded above, we let 
\[
a(\mathscr X):=\sup\{a\in\Z:\mathscr X_{a-n-1}=0,\ \forall n\in \N\};
\]
\item $\mathscr X$ is {\bf bounded below} if there exists $b\in\Z$ such that $\gr_{b+n}\mathscr X=0$ for all $i\in \N$. If $\mathscr X$ is bounded below, we let 
\[
b(\mathscr X):=\inf\{b\in\Z:\gr_{b+n}\mathscr X=0,\ \forall n\in \N\};
\]
\item $\mathscr X$ is  {\bf bounded} if it is both bounded above and below.  If $\mathscr X$ is bounded, we define its {\bf length} $\ell(\mathscr X)=-1$ if $\mathscr X=0$ and otherwise
$$\ell(\mathscr X)=|b(\mathscr X)-a(\mathscr X)|\,.$$
\end{itemize} 
We define the {\bf bounded filtered prederivator} 
\[
\FD\colon \Dia^{\op}\to \CAT,
\] 
as a full sub-prederivator of $\D^\Z$, where $\FD(I)$ is the category of bounded objects in $\D^\Z(I)$.
\end{defn}

In the following remark we try to give some motivation and intuition for the above definition:
\begin{rmk}
Let $\mathscr X\in \D(\Z)$ and consider its underlying diagram:
\[
\dia_\Z(\mathscr X):\xymatrix@C=20pt{\cdots\ar[r]&\mathscr X_{-2}\ar[r]^{d_{-2}}&\mathscr X_{-1}\ar[r]^{d^{-1}}&\mathscr X_{0}\ar[r]^{d^{0}}&\mathscr X_1\ar[r]^{d^1}&\mathscr X_2\ar[r]^{d^2}&\cdots}
\]
The object $\mathscr X$ is then bounded above if $\mathscr X_{k}=0$ for all $k<<0$ and it is bounded below if $d^k$ is an isomorphism for all $k>>0$. In particular, when $\mathscr X$ is bounded, we can interpret it as the following (coherent) finite filtration of the objects $\mathscr X_{b(\mathscr X)}$:
\[
0=\mathscr X_{a(\mathscr X)}\to \mathscr X_{a(\mathscr X)+1}\to \mathscr X_{a(\mathscr X)+2}\to \cdots\to \mathscr X_{b(\mathscr X)-1}\to \mathscr X_{b(\mathscr X)}.
\]
Hence, the motivation for calling $\FD$ the bounded filtered derivator of $\D$ is that its objects are exactly the coherent finite filtrations of the objects in $\D$. Another motivation is the following: if we take a Grothendieck category $\G$ and the associated stable derivator $\sfD_\G$, the base of the corresponding filtered derivator $\mathbb F^b\sfD_\G$ is (equivalent to) what is usually called the filtered derived category of $\G$.

A final remark is that one would expect that the definition of ``bounded above" object of $\D(\Z)$ is what we call here ``bounded below", and viceversa. The reason for this counterintuitive choice of words will be clarified in the proof of Proposition \ref{heart_equiv_prop}.
\end{rmk}

Notice that, in general $\FD$ is not a derivator but, for any $I\in \Dia$, $\FD(I)$ is a full triangulated subcategory of $\D^\Z(I)$. In particular, it makes sense to consider triangles and to speak about $t$-structures in $\FD(I)$ for some $I\in \Dia$. The following result follows by Lemma \ref{lem_epivalence_N_op}.

\begin{cor}\label{lift_Z}
For any $I\in\Dia$, the functor
\[
\dia_{\Z}\colon \FD(I)\to \D(I)^{(b,\Z)}
\]
is full and essentially surjective, where $\D(I)^{(b,\Z)}$ is the full subcategory of $\D(I)^{\Z}$ spanned by those $X\colon \Z\to \D(I)$ such that there exists $a\leq b\in\Z$ for which $X_{a-i}=0$ and $X_{b+i}\to X_{b+i+1}$ is an isomorphism, for all $i\in \N$. Furthermore, given $\mathscr Y \in \FD(I)$, there is a triangle
\[
\mathscr Y\to \prod_{i\leq b(\mathscr Y)}i_*\mathscr Y_i\to \prod_{i\leq b(\mathscr Y)}(i-1)_*\mathscr Y_i\to \Sigma \mathscr Y
\]
Finally, given a second object  $\mathscr X\in \FD(I)$, the diagram functor induces an isomorphism
\[
\FD(I)(\mathscr X,\mathscr Y)\cong \D(I)^{\Z}(\dia_{\Z}\mathscr X,\dia_{\Z}\mathscr Y)
\]
provided the following map is surjective:
\begin{align*}
\prod_{i\leq b}\D(I)(\Sigma \mathscr X_i,\mathscr Y_i)&\to \prod_{i\leq b}\D(I)(\Sigma \mathscr X_{i-1},\mathscr Y_{i})\\
(\phi_i)_{i\leq b}&\mapsto (d_{\mathscr Y}^{i-1}\phi_{i-1}-\phi_{i}\Sigma d_{\mathscr X}^{i-1})_{i\leq b},
\end{align*}
where $b:=\max\{b(\mathscr X),b(\mathscr Y)\}$.
\end{cor}

\begin{lem}\label{hom=hocolim}
Given $I\in\Dia$ and $\mathscr X,\,\mathscr Y\in \FD(I)$ such that $a(\mathscr X)\geq b(\mathscr Y)$, then 
\[
\FD(I)(\mathscr X,\mathscr Y)=\D(I)(\hocolim_\Z(\mathscr X),\hocolim_\Z(\mathscr Y)).
\]
\end{lem}
\begin{proof}
It is easy to show that the map 
\begin{align*}
\prod_{i\leq b(\mathscr X)}\D(I)(\Sigma \mathscr X_i,\mathscr Y_i)&\to \prod_{i\leq b(\mathscr X)}\D(I)(\Sigma \mathscr X_{i-1},\mathscr Y_{i})
\end{align*}
is surjective, so that $\FD(I)(\mathscr X,\mathscr Y)\cong\D(I)^\Z(\dia_\Z \mathscr X,\dia_\Z \mathscr Y)$.
Consider now the underlying diagrams:
\[
\xymatrix@R=15pt@C=20pt{
\cdots \ar[r]&0\ar[r]\ar[d]&\mathscr X_{a(\mathscr X)}\ar[r]\ar[d]&\mathscr X_{a(\mathscr X)+1}\ar[r]\ar[d]&\cdots\ar[r]&\mathscr X_{b(\mathscr X)}\ar[d]\ar[r]^\cong&\cdots\\
\cdots \ar[r]&\mathscr Y_{{a(\mathscr X)}-1}\ar[r]&\mathscr Y_{a(\mathscr X)}\ar[r]^{\cong}&\mathscr Y_{{a(\mathscr X)}+1}\ar[r]^{\cong}&\cdots\ar[r]^{\cong}&\mathscr Y_{b(\mathscr X)}\ar[r]^\cong&\cdots\\
}
\]
Clearly, $\hocolim_\Z \mathscr X\cong \mathscr X_{b(\mathscr X)}$, $\hocolim_\Z \mathscr Y\cong \mathscr Y_{b(\mathscr Y)}\cong \mathscr Y_{b(\mathscr X)}$ and 
\[
\D(I)^\Z(\dia_\Z\mathscr  X,\dia_\Z \mathscr Y)\cong \D(I)(\mathscr X_{b(\mathscr X)},\mathscr Y_{b(\mathscr X)}).\qedhere
\]
\end{proof}

\subsection{The Beilinson $t$-structure}

\begin{prop}\label{heart_is_Chb}
Consider the following two classes of objects in $\FD(\bbone)$:
\begin{align*}
\FD^{\leq0}(\bbone)&=\{\mathscr X\in \FD(\bbone):\gr_n\mathscr X\in \D^{\leq -n}(\bbone),\, \forall n\},\\
\FD^{\geq0}(\bbone)&=\{\mathscr X\in \FD(\bbone):\gr_n\mathscr X\in \D^{\geq -n}(\bbone),\, \forall n\}.
\end{align*}
Then $\t_{B}:=(\FD^{\leq0}(\bbone),\FD^{\geq0}(\bbone))$ is a $t$-structure on $\FD(\bbone)$.
\end{prop}
\begin{proof}
The closure properties of $\FD^{\leq0}(\bbone)$ and $\FD^{\geq0}(\bbone)$ follow by the fact that all the $\gr_n$ are triangulated functors. 
Now let $\mathscr X\in \FD^{\leq0}(\bbone)$, $\mathscr Y\in \FD^{\geq1}(\bbone)$, and let us show that $\D(\Z)(\mathscr X,\mathscr Y)=0$. 
Consider the distinguished triangle 
\[
a(\mathscr X)_!\mathscr X_{a(\mathscr X)}\to \mathscr X\to \mathscr X'\to\Sigma a(\mathscr X)_!\mathscr X_{a(\mathscr X)},
\]
where $\ell(a(\mathscr X)_!\mathscr X_{a(\mathscr X)})=0$, $\ell(\mathscr X')\leq\ell(\mathscr X)-1$, and $a(\mathscr X)_!\mathscr X_{a(\mathscr X)},\, \mathscr X'\in \FD^{\leq0}(\bbone)$. The following exact sequence
\[
\cdots \to\D(\Z)(\mathscr X',\mathscr Y)\to \D(\Z)(\mathscr X,\mathscr Y)\to \D(\Z)(a_!\mathscr X_a,\mathscr Y)\to \cdots
\]
shows that $\D(\Z)(\mathscr X,\mathscr Y)=0$ provided $\D(\Z)(\mathscr X',\mathscr Y)=0=\D(\Z)(a_!\mathscr X_a,\mathscr Y)$. We can use this trick to reduce (by induction on $\ell(\mathscr X)$) to the case when $\ell(\mathscr X)=0$. Similarly, 
there is a distinguished triangle $a(\mathscr Y)_!\mathscr Y_{a(\mathscr Y)}\to \mathscr Y\to \mathscr Y'\to \Sigma a(\mathscr Y)_!\mathscr Y_{a(\mathscr Y)}$, where $\ell(a(\mathscr Y)_!\mathscr Y_{a(\mathscr Y)})=0$, $\ell(\mathscr Y')\leq\ell(\mathscr Y)-1$, and $a(\mathscr Y)_!\mathscr Y_{a(\mathscr Y)},\, \mathscr Y'\in D^{\geq0}_b$. 
Similarly to what we did for $\mathscr X$, we can reduce (by induction on $\ell(\mathscr Y)$) to the case when $\ell(\mathscr Y)=0$. Hence, suppose that $\ell(\mathscr X)=\ell(\mathscr Y)=0$ and consider the following cases:\\
$\bullet$ if $a(\mathscr X)\geq a(\mathscr Y)$ then, by adjunction,
\begin{align*}
\D(\Z)(\mathscr X,\mathscr Y)&\cong \D(\Z)(a(\mathscr X)_!\mathscr X_{a(\mathscr X)},a(\mathscr Y)_!\mathscr Y_{a(\mathscr Y)})\\
&\cong \D(\bbone)(\mathscr X_{a(\mathscr X)},a(\mathscr X)^*a(\mathscr Y)_!\mathscr Y_{a(\mathscr Y)})\\
&\cong \D(\bbone)(\mathscr X_{a(\mathscr X)},\mathscr Y_{a(\mathscr Y)})=0
\end{align*}
where the last equality holds since $\mathscr X_{a(\mathscr X)}\cong\gr_{a(\mathscr X)-1}\mathscr X\in \D^{\leq -a(\mathscr X)+1}(\bbone)$, while $\mathscr Y_{a(\mathscr Y)}\cong \gr_{a(\mathscr Y)-1}\mathscr Y\in \D^{\geq -a(\mathscr Y)+2}(\bbone)\subseteq \D^{\geq -a(\mathscr X)+2}(\bbone)$;\\
$\bullet$ if $a(\mathscr X)<a(\mathscr Y)$ then, by adjunction,
\begin{align*}
\D(\Z)(\mathscr X,\mathscr Y)&\cong \D(\Z)(a(\mathscr X)_!\mathscr X_{a(\mathscr X)},a(\mathscr Y)_!\mathscr Y_{a(\mathscr Y)})\\
&\cong \D(\bbone)(\mathscr X_{a(\mathscr X)},a(\mathscr X)^*a(\mathscr Y)_!\mathscr Y_{a(\mathscr Y)})\\
&\cong \D(\bbone)(\mathscr X_{a(\mathscr X)},0)=0.
\end{align*}

Given $\mathscr X\in \FD(\bbone)$, we need to show that there exists a triangle 
\[
\mathscr X^{\leq 0}\to \mathscr X\to \mathscr X^{\geq 1}\to \Sigma(\mathscr X^{\leq0})
\]
such that $\mathscr X^{\leq0}\in \FD^{\leq0}(\bbone)$ and $\mathscr X^{\geq1}\in \FD^{\geq1}(\bbone)$. If $\ell(\mathscr X)=-1$ then we can take everything to be $0$. Otherwise, we proceed by induction on $\ell(\mathscr X)\in \N$ to find  such $\mathscr X^{\leq 0}$ and $\mathscr X^{\geq 1}$ with the extra conditions that $\hocolim_\N \mathscr X^{\leq0}\in \D^{\leq -a(\mathscr X)}(\bbone)$ and $\hocolim_\N \mathscr X^{\geq1}\in \D^{\geq -b(\mathscr X)+1}(\bbone)$. Indeed:\\
$\bullet$ If $\ell(\mathscr X)=0$, say $\mathscr X=a_!\mathscr X_a$, then it is enough to take $\mathscr X^{\leq0}:=a_!(\mathscr X_a^{\leq -a})$ and $\mathscr X^{\geq1}:=a_!(\mathscr X_a^{\geq -a+1})$ with the obvious maps.\\ 
$\bullet$ If $\ell(\mathscr X)>0$ and $a:=a(\mathscr X)$, then there is a distinguished triangle in $\FD(\bbone)$, $a_!\mathscr X_a\to \mathscr X\to \mathscr X'\to \Sigma a_!\mathscr X_a$ and we can consider the following diagram
\[
\xymatrix@R=15pt{
\Sigma^{-1}(\mathscr X'^{\leq 0})\ar@{.>}[r]\ar[d]&(a_!\mathscr X_a)^{\leq0}\ar[d]\ar@{.>}[r]&\mathscr A\ar@{.>}[d]\ar@{.>}[r]&\mathscr X'^{\leq 0}\ar[d]\\
\Sigma^{-1}\mathscr X'\ar[r]\ar[d]&a_!\mathscr X_a\ar[d]\ar[r]&\mathscr X\ar@{.>}[d]\ar[r]&\mathscr X'\ar[d]\\
\Sigma^{-1}(\mathscr X'^{\geq 1})\ar@{.>}[r]\ar[d]&(a_!\mathscr X_a)^{\geq1}\ar@{.>}[r]\ar[d]&\mathscr B\ar@{.>}[r]\ar@{.>}[d]&\mathscr X'^{\geq 1}\ar[d]\\
\mathscr X'^{\leq 0}\ar@{.>}[r]&\Sigma ((a_!\mathscr X_a)^{\leq0})\ar@{.>}[r]&\Sigma \mathscr A\ar@{.>}[r]&\Sigma(\mathscr X'^{\leq 0})
}
\]
where all rows and columns are triangles, and all the small squares commute, but the bottom right square that anti-commutes. This diagram is constructed as follows:
first we should complete the following diagram to a morphism of triangles
\[
\xymatrix@R=15pt{
\Sigma^{-1}(\mathscr X'^{\leq 0})\ar[r]\ar@{.>}[d]&\Sigma^{-1}\mathscr X'\ar[d]\ar[r]&\Sigma^{-1}(\mathscr X'^{\geq 1})\ar@{.>}[d]\ar[r]&\mathscr X'^{\leq 0}\ar@{.>}[d]\\
(a_!\mathscr X_a)^{\leq0}\ar[r]&a_!\mathscr X_a\ar[r]&(a_!\mathscr X_a)^{\geq1}\ar[r]&\Sigma ((a_!\mathscr X_a)^{\leq0})
}
\]
where the solid map is given by a rotation of $a_!\mathscr X_a\to \mathscr X\to \mathscr X'\to \Sigma a_!\mathscr X_a$. In fact, using~\cite[Prop.\,1.1.9]{BBD}, the following two vanishing conditions imply that the above solid diagram can be completed in a unique way to a morphism of triangles:
\begin{itemize}
\item using Lemma \ref{hom=hocolim}, we get that $\D(\Z)(\Sigma^{-1}(\mathscr X'^{\leq 0}),(a_!\mathscr X_a)^{\geq1})$ is isomorphic to $\D(\bbone)(\hocolim_\Z\Sigma^{-1}(\mathscr X'^{\leq 0}),\hocolim_\Z (a_!\mathscr X_a)^{\geq1})$ which is trivial since $\hocolim_\Z\Sigma^{-1} (\mathscr X'^{\leq 0})\in \D^{\leq -a(\mathscr X')+1}(\bbone)\subseteq  \D^{\leq -a(\mathscr X)}(\bbone)$, while $\hocolim_\Z (a_!\mathscr X_a)^{\geq1}=\mathscr X_a^{\geq -a(\mathscr X)+1}\in \D^{\geq -a(\mathscr X)+1}(\bbone)$;
\item using Lemma \ref{hom=hocolim}, we get that $\D(\Z)((\mathscr X')^{\leq 0},(a_!\mathscr X_a)^{\geq1})$ is isomorphic to $\D(\bbone)(\hocolim_\Z (\mathscr X')^{\leq 0},\hocolim_\Z (a_!\mathscr X_a)^{\geq1})=0$, which is trivial since $\hocolim_\Z (\mathscr X')^{\leq 0}\in \D^{\leq -a(\mathscr X')}(\bbone)\subseteq \D^{\leq -a(\mathscr X)-1}(\bbone)$, whereas $\hocolim_\Z a_!\mathscr X_a^{\geq1}=\mathscr X_a^{\geq -a(\mathscr X)+1}\in \D^{\geq -a(\mathscr X)+1}(\bbone)$.
\end{itemize}
 Since there exists a unique way to complete the diagram to a morphism of triangles, this morphism of triangles can be completed to the desired $3\times 3$ diagram (for this use~\cite[Thms.\,1.8 and 2.3]{neeman1991some}, these results say respectively, in the language of that paper, that any commutative square can be completed to a ``good" morphism of triangles and that any such morphism gives rise to a $3\times 3$ diagram as above. Since we have proved that there is a unique way to complete the above diagram to a morphism of triangles, this unique way is the ``good" one).
Now, letting $\mathscr X^{\leq0}:=\mathscr A$ and $\mathscr X^{\geq1}:=\mathscr B$, it is not difficult to verify that they satisfy the desired properties.
Hence, $\t_B$ is a $t$-structure on $\FD(\bbone)$.
\end{proof}

\begin{defn}
The $t$-structure $\t_B:=(\FD^{\leq0},\FD^{\geq0})$ on $\FD$ described in the above proposition is said to be the {\bf Beilinson $t$-structure} induced by $\t$.
\end{defn}

After the following technical lemma we will show that the heart of the Beilinson $t$-structure is precisely the category of bounded complexes over the heart of the original $t$-structure $\t$.

\begin{lem}\label{full_implies_faith}
Let $F\colon \A\to \A'$ be an exact functor between Abelian categories. If $F$ reflects $0$-objects (i.e., $F(X)=0$ implies $X=0$), then $F$ is faithful.
\end{lem}
\begin{proof}
Consider a morphism $\phi\colon A\to B$ in $\A$ such that $F(\phi)=0$. Let $\pi\colon A\to A/\ker(\phi)$ be the obvious projection and let $\bar\phi\colon A/\ker(\phi)\to B$ be the unique morphism such that $\bar\phi\pi=\phi$. Notice that $\pi$ is an epimorphism and $\bar\phi$ is a monomorphism so, by the exactness of $F$, $F(\pi)$ is an epimorphism and $F(\bar \phi)$ is a monomorphism. One can now conclude as follows: $0=F(\phi)=F(\bar\phi\pi)=F(\bar\phi)F(\pi)$ implies that $F(\bar\phi)$ is $0$ as we can cancel the epimorphism $F(\pi)$, but $F(\bar\phi)$ is also a monomorphism, so $F(A/\Ker(\phi))=0$, which means that $A=\Ker(\phi)$, since $F$ reflects $0$-objects.
\end{proof}

Given an Abelian category $\A$, we denote by $\Ch^b(\A)$ the {\bf category of bounded complexes} over $\A$. For a non-trivial complex $0\neq X^\bullet\in\Ch^b(\A)$, we let:
\begin{enumerate}
\item[--] $b(X^{\bullet}):=\max\{n\in\Z:X^{-n}\neq 0\}$,   $a(X^{\bullet}):=\min\{n\in\Z:X^{-n}\neq 0\}$;
\item[--] $\ell(X^{\bullet}):=|a(X^{\bullet})-b(X^{\bullet})|$.
\end{enumerate}
Let also $\ell(\mathscr X)=-1$ if $\mathscr X=0$.

\begin{prop}\label{heart_equiv_prop}
Let $\A$ be the heart of $\t$ and $\H_{B}\subseteq \FD(\bbone)$  the heart of the Beilinson $t$-structure $\t_B$. Then, there is an equivalence 
\[
\xymatrix{
F\colon \H_{B}\ar[rr]^\cong&&\Ch^b(\A)
}
\]
that satisfies the following properties for any $\mathscr X\in \H_B$:
\begin{enumerate}
\item if $\ell(\mathscr X)=0$, then $F\mathscr X$ is a complex concentrated in degree $-a+1$, whose unique non-zero component is $\Sigma^{a-1}\gr_{a-1}\mathscr X$;
\item if $\ell(\mathscr X)=1$, then $F\mathscr X$ is a complex concentrated in degrees $-a+1$ and $-a$, whose non-zero part is the map $\Sigma^{-a}\gr_a\mathscr X\to \Sigma^{-a+1}\gr_{a-1}\mathscr X$;
\item  in general, $\ell(F\mathscr X)=\ell(\mathscr X)$. 
\end{enumerate}
\end{prop}
\begin{proof}
We start constructing a functor $F\colon \H_{B}\to \Ch^b(\A)$
as follows: to an object $\mathscr X\in \H_{B}$ we associate the complex $F\mathscr X^\bullet$, where $F\mathscr X^{-n}=\Sigma^{-n}\gr_{n}\mathscr X$ and with the differential $d^{-n-1}_{F\mathscr X}\colon F\mathscr X^{-n-1}\to F\mathscr X^{-n}$, such that $\Sigma^{n+1}d_{F\mathscr X}^{-n-1}$ fits in the following octahedron:
\[
\xymatrix@R=15pt{
\mathscr X_n\ar@{=}[d]\ar[r]& \mathscr X_{n+1}\ar[r]\ar[d]& \gr_n\mathscr X\ar[d]\ar[r] &\Sigma \mathscr X_n\ar@{=}[d]\\
\mathscr X_n\ar[r]&\mathscr  X_{n+2}\ar[r]\ar[d]& \mathscr X_{[n,n+2]}\ar[d]\ar[r] &\Sigma \mathscr X_n\\
&\gr_{n+1}\mathscr X\ar[d]\ar@{=}[r]&\gr_{n+1} \mathscr X\ar[d]^{\Sigma^{n+1}d^{-n-1}_{F\mathscr X}}\\
&\Sigma \mathscr X_{n+1}\ar[r]&\Sigma \gr_n\mathscr X
}
\]
With this definition, the properties (1--3) in the statement are trivially verified, let us prove that $F$ is an equivalence:\\
$\bullet$ {\bf $F$ is exact}. Consider a short exact sequence $0\to \mathscr X\to \mathscr Y\to \mathscr Z\to 0$ in $\H_{B}$.
The short exact sequences in $\H_{B}$ are the triangles of $\FD(\bbone)$ that happen to lie in $\H_{B}$, in particular, there exists a map $\mathscr Z\to \Sigma \mathscr X$ such that $\mathscr X\to \mathscr Y\to \mathscr Z\to \Sigma \mathscr X$ is a triangle. We have to prove that $0\to F\mathscr X\to F\mathscr Y\to F\mathscr Z\to 0$ is a short exact sequence in $\Ch^b(\A)$, but this means exactly that $0\to F\mathscr X^n\to F\mathscr Y^n\to F\mathscr Z^n\to 0$ is a short exact sequence in $\A$ for any $n\in\Z$. This means that there is a triangle $\gr_n\mathscr X\to \gr_n\mathscr Y\to \gr_n\mathscr Z\to \Sigma \gr_n\mathscr X$ in $\D(\bbone)$, which is true since $\gr_n$ is a triangulated functor.\\
$\bullet$ {\bf $F$ reflects $0$-objects}. Let $\mathscr X\in \FD(\bbone)$ and let us define a sequence of objects $(\mathscr X'(n))_{n\in\N}$ in $\FD(\bbone)$. Indeed, let $a:=a(\mathscr X)$ and 
\begin{enumerate}
\item[$(*)$] define $\mathscr X'(0)$ to be the cone of the counit $a_!\mathscr X_a\to \mathscr X$ (here $\mathscr X_a\cong\gr_{a-1}\mathscr X$);
\item[$(*)$] notice that $\mathscr X'(0)_a=0$ and $\mathscr X'(0)_{a+1}=\gr_a\mathscr X$, so there is a natural map $(a+1)_!\gr_a(\mathscr X)\to \mathscr X'(0)$, and we define $\mathscr X'(1)$ to be the cone of this map;
\item[$(*)$] we define $\mathscr X'(n+1)$ as the cone of the counit $(n+a+1)_!\gr_{n+a}\mathscr X\to \mathscr X'(n)$. 
\end{enumerate}
One can show that $\mathscr X'(n)=0$ for any $n> \ell(\mathscr X)$. This filtration shows that $\mathscr X$ belongs in the smallest triangulated subcategory of $\FD(\bbone)$ that contains the objects of the form $(n+1)_!\gr_n(\mathscr X)$. Thus, if $\gr_n\mathscr X=0$ for all $n\in\Z$ (i.e., if $F\mathscr X=0$ in $\Ch^b(\A)$), then $\mathscr X=0$.\\
$\bullet$ {\bf $F$ is fully faithful}. By Lemma \ref{full_implies_faith} and the properties of $F$ that we have already verified, it is enough to show that $F$ is full. Indeed, let $\mathscr X,\, \mathscr Y\in \FD(\bbone)$ and let $\phi\in \Ch^b(\A)(F\mathscr X,F\mathscr Y)$. As in the proof that $F$ reflects $0$-objects, one can construct $\mathscr X$ and $\mathscr Y$ as extensions of shifts of objects of the form $(n+1)_!\gr_n\mathscr X$ and $(n+1)_!\gr_n\mathscr Y$. Of course the map $\phi$ induces maps $(n+1)_!(\Sigma^n\phi_n)\colon (n+1)_!\gr_n\mathscr X\to (n+1)_!\gr_n\mathscr Y$, for any $n\in\Z$, and these maps can be glued together to construct a new map $\psi\colon \mathscr X\to \mathscr Y$ such that $F(\psi)=\phi$.\\
$\bullet$ {\bf $F$ is essentially surjective}. Given a complex $C^\bullet\in \Ch^b(\A)$, say
\[
\xymatrix@C=25pt{\cdots\ar[r]&0\ar[r]&C^{-b}\ar[r]^{d^{-b}}&C^{-b+1}\ar[r]^(.65){d^{-b+1}}&\cdots \ar[r]^{d^{-a-1}}&C^{-a}\ar[r]&0\ar[r]&\cdots},
\]
we have to construct and object $\mathscr X\in \H_{B}$ such that $F\mathscr X\cong C^\bullet$. Suppose, for simplicity, that $a=0$, the general case can be handled analogously by shifting the indices. So our complex becomes:
\[
\xymatrix@C=25pt{\cdots\ar[r]&0\ar[r]&C^{-b}\ar[r]^{d^{-b}}&C^{-b+1}\ar[r]^(.65){d^{-b+1}}&\cdots \ar[r]^{d^{-1}}&C^{0}\ar[r]&0\ar[r]&\cdots}.
\]
We construct a diagram in $\D(\bbone)^{\Z}$
\[
\xymatrix@R=0pt{
\cdots\ar[r]&0\ar[r]&X(0)\ar[r]^{x_0}&X(1)\ar[r]^{x_1}&\cdots\ar[r]&X(b+1)\ar[r]&0\ar[r]&\cdots}
\]
(one can successively lift it to a coherent diagram $\mathscr X$ in $\FD(\bbone)$, using Corollary \ref{lift_Z}). We define:
\begin{enumerate}
\item[$(*)$] for $n\leq 0$, we let $X(n)=0$; 
\item[$(*)$] $X(1):=C^0\in \D^{\geq0}(\bbone)$ and let $\xymatrix@C=12pt{
X(0)\ar[rr]^{x_0=0}&& X(1)\ar[rr]^{\id_{C^0}=y_0}&& C^{0}\ar[rr]^{z_0:= d^{0}}&& 0}$;
\item[$(*)$] $X(2)\in \D^{\geq-1}$ and $x_1\colon X(1)\to X(2)$ are defined as part of a triangle:
\[
\xymatrix@C=20pt{X(1)\ar[r]^{x_1}& X(2)\ar[r]^{y_1}& \Sigma C^{-1}\ar[rrr]^{z_1:=\Sigma d^{-1}}&&& \Sigma X(1)}
\]
so that $\Sigma y_0\circ z_{1}=\id_{C^0}\circ\Sigma d^{-1} =\Sigma d^{-1}$;
\item[$(*)$] then we proceed inductively. Suppose that we have already constructed two triangles with $X(i)\in \D^{\geq-i+1}(\bbone)$ (so that $X(i+1)\in \D^{\geq-i}(\bbone)$ and $X(i+2)\in \D^{\geq-i-1}(\bbone)$)
\[
\xymatrix@R=0pt{
X(i)\ar[r]^{x_i}& X(i+1)\ar[r]^{y_i}& \Sigma^{i}C^{-i}\ar[r]^{z_i}& \Sigma X(i)\\
X(i+1)\ar[r]^{x_{i+1}}& X(i+2)\ar[r]^{y_{i+1}}& \Sigma^{i+1}C^{-i-1}\ar[r]^{z_{i+1}}& \Sigma X(i+1)
}
\]
and such that $\Sigma y_i\circ z_{i+1}=\Sigma^{i+1}d^{-i-1}$.
We define $X(i+3)\in \D^{\geq-i-2}$ and $x_{i+2}\colon X(i+2)\to X(i+3)$ as parts of a triangle
\[
\xymatrix@R=0pt{
X(i+2)\ar[r]^{x_{i+2}}& X(i+3)\ar[r]^{y_{i+2}}& \Sigma^{i+2}C^{-i-2}\ar[r]^{z_{i+2}}& \Sigma X(i+2)
}
\]
Notice first that $\Sigma y_i\circ z_{i+1}\circ \Sigma^{i+1}d^{-i-2}=\Sigma^{i+1}(d^{-i-1}d^{-i-2})=0$, so that we have a commutative diagram
\[
\xymatrix{
0\ar[r]\ar[d]&\Sigma^{i+1}C^{-i-2}\ar@{=}[r]\ar[d]&\Sigma^{i+1}C^{-i-2}\ar[r]\ar[d]|{z_{i+1} \Sigma^{i+1}d^{-i-2}}&0\ar[d]\\
\Sigma^{i}C^{-i}\ar[r]^{z_i}&\Sigma X(i)\ar[r]^{\Sigma x_i}& \Sigma X(i+1)\ar[r]^{\Sigma y_i}&\Sigma^{i+1}C^{-i}.}
\]
Now,  $\Sigma^{i+1}C^{-i-2}\in \D^{\leq -i-1}(\bbone)$ and $\Sigma X(i)\in \D^{\geq -i}(\bbone)$, so 
\[
\D(\bbone)(\Sigma^{i+1}C^{-i-2},\Sigma X(i))=0,
\] 
showing that $z_{i+1} \Sigma^{i+1}d^{-i-2}=0$. Hence, we can define $z_{i+2}$ in order to make the following diagram commute
\[
\xymatrix{
0\ar[r]\ar[d]&\Sigma^{i+1}C^{-i-2}\ar@{=}[r]\ar[d]|{\Sigma^{-1}z_{i+2}}&\Sigma^{i+1}C^{-i-2}\ar[r]\ar[d]|{\Sigma^{i+1}d^{-i-2}}&0\ar[d]\\
X(i+1)\ar[r]^{x_{i+1}}& X(i+2)\ar[r]^{y_{i+1}}& \Sigma^{i+1} C^{-i-1}\ar[r]^{z_{i+1} }&\Sigma X(i+1).}
\]
We define $X(i+3)$, $x_{i+2}$ and $y_{i+2}$ completing $z_{i+2}$ to a triangle.\qedhere
\end{enumerate}
\end{proof}

\subsection{Totalization of bounded complexes}
\label{sec:ex-real}

If $n>1$, we define the following morphism of derivators
\[
\Tr_n\colon\Ch^{\ordn}_\A\to (\Ch^{\ordn-\bbone}_\A)^{\bbtwo}
\]
such that, given $I\in \Dia$ and a complex $X\in \Ch^{\ordn}_\A(I)$, 
\[
(\Tr_nX\colon X^{<n-1}\to \Omega^{n-2} X^{n-1}):=
\begin{pmatrix}\xymatrix@C=10pt@R=7pt{
\vdots\ar[d]&\vdots\ar[d]\\
0\ar[d]\ar[r]&0\ar[d]\\
X^0\ar[d]\ar[r]&0\ar[d]\\
\vdots\ar[d]&\vdots\ar[d]\\
X^{n-2}\ar[r]\ar[d]&X^{n-1}\ar[d]\\
0\ar[r]\ar[d]&0\ar[d]\\
\vdots&\vdots
}\end{pmatrix}
\]
We can now introduce the main definition of this subsection:
\begin{defn}
Given $n\in \N_{>0}$ we define, inductively, a morphism of prederivators, called {\bf totalization morphism},
\[
\Tot_\t^n \colon \Ch_\A^{\ordn}\to \D
\]
$\bullet$ for $n=1$, $\Ch_\A^{\bbone}=y(\A)=\D^{\heartsuit}$ so we let $\Tot^1_\t$ be the inclusion $\D^{\heartsuit}\to \D$;
\\
$\bullet$ in general we define $\Tot_\t^{n+1}$ to be the following composition
\[
\Tot_\t^{n+1}:=\mathrm{fib}\circ(\Tot_\t^{n})^{\bbtwo}\circ \Tr_n \colon\Ch_\A^{\ordnp}\to(\Ch_\A^{\ordn})^{\bbtwo}\to  \D^\bbtwo \to \D.
\]
\end{defn}

In the following lemma we collect some natural properties of the totalization morphisms:

\begin{lem}\label{basic_prop_tot_n}
Given $n\in\N_{>0}$, the following statements hold true:
\begin{enumerate}
\item $\Tot_\t^n$ sends short exact sequences of complexes to triangles;
\item $\Tot_\t^{n}\Omega^{n-1}A\cong \Omega^{n-1}A$ for any $A\in\A$ (where the $\Omega^{n-1}A$ on the left hand side is the complex concentrated in degree $n-1$, with $A$ as the unique non-zero component);
\item given $X\in \Ch^{\ordn}(\A)\subseteq \Ch^{\ordnp}(\A)$, $\Tot_\t^{n+1}X\cong \Tot_\t^n X$;
\item given $X\in \Ch^{\ordn+\bbone}(\A)$, there is a triangle in $\D(\bbone)$:
\[
\xymatrix@C=10pt{
\Tot_\t^{n+1}X\ar[r]&\Tot_\t^{n}X^{<n}\ar[r]&\Omega^{n-1}X^n\ar[r]&\Sigma\Tot_\t^{n+1}X;
}
\]
\item given $X\in \Ch^{\ordn}(\A)$, $\Tot_\t^{n+1}\Omega X\cong \Omega\Tot_\t^nX$.
\end{enumerate}
\end{lem}
\begin{proof}
(1) The proof goes by induction on $n$. For $n=1$, short exact sequences in $y(\A)\subseteq \D$ are clearly sent to triangles in $\D$ by the inclusion. On the other hand, $\Tot_\t^{n+1}$ is, by definition, a composition $\mathrm{fib}\circ(\Tot_\t^{n})^{\bbtwo}\circ \Tr_n$. Now, $\Tr_n$ is an exact functor, $(\Tot_\t^{n})^{\bbtwo}$ sends short exact sequences to triangles by inductive hypothesis, and $\mathrm{fib}\colon \D^{\bbtwo}\to \D$ is an exact morphism of stable derivators (so it is, component-wise, a triangulated functor).

(2) For $n=1$ the statement is clear. Suppose we have already proved our statement for some $n>0$, then the statement boils down to the usual fact that the fiber of the morphism $0\to \Omega^{n-1}A$ in the triangulated category $\D(\bbone)$ is isomorphic to $\Omega^nA$.

(3) This boils down to the usual fact that, in the triangulated category $\D(\bbone)$, the fiber of the morphism $\Tot_\t^n X\to 0$ is isomorphic to $\Tot_\t^n X$.

(4) By parts (2) and (3), $\Tot_\t^{n+1}X^{<n}\cong \Tot_\t^n X^{<n}$, while $\Tot_\t^{n+1}\Omega^{n}X^n\cong \Omega^{n}X^n$. Consider now the following short exact sequence in $\Ch^{\ordnp}(\A)$:
\[
\xymatrix{
0\ar[r]& \Omega^nX^n\ar[r]& X\ar[r]&X^{<n}\ar[r]&0
}
\]
Applying $\Tot_\t^{n+1}$ we get (a rotation of) the triangle in the statement.

(5) For $n=1$ the statement reduces to part (2). If $n>1$, let $Y:=\Omega X\in \Ch^{\ordnp}(\A)$ and consider the triangle of part (4),
\[
\xymatrix@C=10pt{
\Tot_\t^{n+1}Y\ar[r]&\Tot_\t^{n}Y^{<n}\ar[r]&\Omega^{n-1}Y^n\ar[r]&\Sigma\Tot_\t^{n+1}Y.
}
\]
Notice that $Y^{<n}=\Omega X^{<n-1}$ and $Y^n=X^{n-1}$. Hence there is a diagram, which is commutative by inductive hypothesis, whose first line is the above triangle while the second line is the triangle given by part (4) to which we have applied $\Omega$:
\[
\xymatrix@C=10pt{
\Tot_\t^{n+1}\Omega X\ar[r]&\Tot_\t^{n}\Omega X^{<n-1}\ar[r]\ar[d]|{\cong}&\Omega^{n-1}X^{n-1}\ar[r]\ar[d]|{\cong}&\Sigma\Tot_\t^{n+1}\Omega X\\
\Omega\Tot_\t^{n} X\ar[r]&\Omega\Tot_\t^{n-1} X^{<n-1}\ar[r]&\Omega^{n-1}X^{n-1}\ar[r]&\Tot_\t^{n+1} X.
}
\]
This clearly implies that  $\Tot_\t^{n+1}\Omega X\cong\Omega\Tot_\t^{n} X$.
\end{proof}

Before proceeding further we want to show how the totalization morphisms can be used to detect the exactness of a bounded complex: 

\begin{prop}\label{description_partial_tots}
Given $X\in\Ch^{\ordn}(\A)$, $\Tot_\t^n X\in \D^{\leq n-1}(\bbone)\cap \D^{\geq 0}(\bbone)$. Furthermore, the following are equivalent:
\begin{enumerate}
\item  $X$ is an exact complex;
\item  $\Tot_\t^{k+1}(\tau^{\leq k}X)=0$ for all $k< n$.
\end{enumerate}
\end{prop}
\begin{proof}
The first statement follows by induction using the above lemma. Indeed, if $n=1$, then $\Tot_\t^1X\in \A=\D^{\heartsuit}(\bbone)$. For $n>1$, let $m:=n-1$,  there is a triangle $\Omega^mX^m\to \Tot_\t^{n}X\to \Tot_\t^mX^{<m}\to\Omega^{m-1}X^m$ and, by inductive hypothesis,  $\Tot_\t^mX^{<m}\in \D^{\leq m-1}(\bbone)\cap \D^{\geq0}(\bbone)\subseteq \D^{\leq n-1}(\bbone)\cap \D^{\geq0}(\bbone)$, while $\Omega^mX^m\in \D^{\leq m}(\bbone)\cap \D^{\geq m}(\bbone)\subseteq \D^{\leq n-1}(\bbone)\cap \D^{\geq 0}(\bbone)$. Since aisle and co-aisle are extension-closed, we get that $\Tot_\t^n X\in \D^{\leq n-1}(\bbone)\cap \D^{\geq 0}(\bbone)$.

Let us now show the equivalence of the conditions (1) and (2); we proceed by induction on $n$. Indeed, if $n=1$, then $X$ is exact if and only if $X=0$, if and only if $\Tot_\t^1X\cong X^0=0$. If $n=2$, $X$ is exact if and only if $d^0$ is an isomorphism, if and only if $\Tot_\t^2(X)=\mathrm{fib}(d^0)=0$.\\ 
For $n>2$, $X$ is an exact complex if and only if $\tau^{\leq n-2}X$ and $\tau^{\geq n-1}X$ are exact. Using the inductive hypothesis we get:
\begin{itemize}
\item $\tau^{\leq n-2}X$ is exact iff $\Tot_\t^{k+1}(\tau^{\leq k}X)=0$ for all $k<n-1$;
\item $\tau^{\geq n-1}X$ is exact iff $\Tot_\t^{n}(\tau^{\geq n-1}X)=\Tot_\t^{n}(\tau^{\leq n-1}\tau^{\geq n-1}X)=0$.
\end{itemize}
These two conditions together are clearly equivalent to (2).
\end{proof}

We have constructed a commutative diagram (up to natural iso):
\[
\xymatrix@R=6pt{
\Ch^{\bbone}_\A\ar@/^0.8pc/[rrrddd]^{\Tot_\t^1}\ar[dd]\\
\\
\Ch^{\bbtwo}_\A\ar[rrrd]^{\Tot_\t^2}\ar[dd]|{\vdots}\\
&&&\D\\
\Ch^{\ordn}_\A\ar@/_0.8pc/[rrru]^{\Tot_\t^n }\ar[d]\\
\vdots
}
\]
The above diagram describes an object of $(\PDer_\Dia)^\N(F,\Delta_\D)$, where $F(n)=\Ch^{\ordn}_\A$. By the universal property of pseudo-colimits, this corresponds to a unique object in $\PDer_\Dia(\pcolim F,\D)$, that is, a morphism of derivators
\[
\Tot_\t^{(b,+)}\colon \Ch_\A^{(b,+)}\cong\pcolim F\to \D.
\]
Similarly, using Lemma \ref{basic_prop_tot_n}, one shows that the following diagram is commutative up to natural isomorphisms:
\[
\xymatrix@R=10pt{
\Ch^{(b,+)}_\A\ar@/^1.5pc/[rrrddd]|{\Tot_\t^{(b,+)}}\ar[dd]_\Omega\\
\\
\Ch^{(b,+)}_\A\ar[rrrd]|{\Sigma\Tot_\t^{(b,+)}}\ar[dd]|{\vdots}\\
&&&\D\\
\Ch^{(b,+)}_\A\ar@/_1.2pc/[rrru]|{\Sigma^{n-1}\Tot_\t^{(b,+)}}\ar[d]\\
\vdots
}
\]
Hence, using that $\Ch_\A^{b}$ is a sequential pseudo-colimit of copies of $\Ch^{(b,+)}_\A$, we can uniquely extend $\Tot_\t^{(b,+)}$ to a  morphism 
\[
\Tot_\t^b\colon \Ch_\A^{b}\to \D.
\]

\begin{prop}\label{qi_to_iso}
Consider the above morphism $\Tot_\t^b\colon \Ch^{b}_\A\to \D$. For any $I\in \Dia$ and $X\in \Ch^{b}_\A(I)$ we have the following properties:
\begin{enumerate}
\item $\Tot_\t^b(\Omega X)\cong \Omega\Tot_\t^b(X)$ and $\Tot_\t^b(\Sigma X)\cong \Sigma\Tot_\t^b(X)$;
\item $\Tot_\t^b\colon \Ch^{b}_\A(I)\to \D(I)$ short exact sequences to triangles;
\item $\Tot_\t^bX=0$, provided $X$ is an exact complex.
\end{enumerate}
\end{prop}
\begin{proof}
One can check that, given $(X,n)\in \Ch^{(b,+)}_\A(I)$, then $\Tot^b_\t(X,n)\cong \Tot_\t^n (X)$, so the statement follows by Proposition \ref{description_partial_tots}.  
\end{proof}

\subsection{The bounded realization functor}

Consider the stable derivator $\sfD_\A^b\colon \Cat^\op\to \CAT$,  such that $\sfD_\A^b(I)=\sfD^b(\A^I)$ is the bounded derived category of the Abelian category $\A^I$. In the previous section we have constructed a morphism $\Tot_\t^b\colon \Ch^b_\A\to \D$ that sends short exact sequences to triangles and that vanishes on exact complexes. Such a morphism factors uniquely through the quotient  $\Ch^b_\A\to \sfD_\A^b$ giving us an exact (and $t$-exact) morphism of derivators, called ({\bf bounded}) {\bf realization functor},
\[
\real^b_\t \colon \sfD_\A^b\to \D
\] 
extending the inclusion $y(\A)\to \D$.

\begin{thm}
The bounded totalization morphism $\Tot_\t^b\colon \Ch^b_\A\to \D$ factors uniquely through the quotient $\Ch^b_\A\to \sfD^b_\A$, giving rise to an exact and $t$-exact morphism of derivators
\[
\real_\t^b\colon \sfD^b_\A\to \D,
\]
such that $\real_\t^b\restriction_{y(\A)}$ is naturally isomorphic to the  inclusion $y(\A)=\D^{\heartsuit}_\t\to \D$.
\end{thm}

We conclude this section about bounded realization functor showing that, when $\N\in\Dia$ and so it makes sense to consider the bounded filtered prederivator $\FD$,  the above construction of $\real_\t^b$ via totalization of complexes coincides with the realization functor constructed via the Beilinson $t$-structure on $\FD$.

\begin{prop}\label{real_is_tot_prop}
Suppose that $\N\in\Dia$. Then, given $\mathscr X\in \H_{B}$ (the heart of the Beilinson $t$-structure in $\FD$), there is a natural isomorphism 
\[
\Tot_\t^b(F(\mathscr X))\cong \hocolim_\Z\mathscr X,
\] 
where $F\colon \H_{B}\to \Ch^b(\A)$ is the equivalence of Proposition \ref{heart_equiv_prop}.
\end{prop}
\begin{proof}
We proceed by induction on $\ell(\mathscr X)\geq -1$.\\ 
If $\ell(\mathscr X)=-1$, that is, $\mathscr X=0$ then the statement is trivial.\\ 
If $\ell(\mathscr X)=0$, then  $\hocolim_\Z\mathscr X\cong \mathscr X_{a}\cong \gr_{a-1}\mathscr X$, where $a:=a(\mathscr X)$
while $\Tot_\t^b(F\mathscr X)\cong \Sigma^{-a+1}(\Sigma^{a-1}\gr_{a-1}\mathscr X)\cong \hocolim_\Z\mathscr X$ (see Proposition \ref{heart_equiv_prop} (1)) as desired.\\ 
If $\ell(\mathscr X)= 1$ then $\hocolim_\N\mathscr X\cong \mathscr X_{a+1}$, while $\Tot_\t^bF\mathscr X$ is the fiber of the map $\gr_a\mathscr X\to \Sigma\gr_{a-1}\mathscr X$ (see Proposition \ref{heart_equiv_prop}); the triangle $\mathscr X_a(=\gr_{a-1}\mathscr X)\to \mathscr X_{a+1}\to \gr_a\mathscr X\to \Sigma\mathscr X_a$ shows that such a fiber is exactly $\mathscr X_{a+1}$.
\\
For $n>1$, let $a:=a(\mathscr X)$, $b:=b(\mathscr X)$, consider $F\mathscr X\in \Ch^{b}(\A)$ and the following short exact sequence in $\Ch^b(\A)$:
\[
0\to A:=\tau^{\leq-b+2}F\mathscr X\to  F\mathscr X\to B\to0.
\]
Then, there are triangles: $\Tot_\t^bA\to \Tot_\t^bF\mathscr X\to \Tot_\t^bB\to\Sigma \Tot_\t^bA$ and 
\[
\hocolim_\Z F^{-1}A\to \hocolim_\Z\mathscr X\to \hocolim_\Z F^{-1}B\to\Sigma \hocolim_\Z F^{-1}A.
\]
Since $\ell(F^{-1}A)\leq 1$ and $\ell(F^{-1}B)<\ell(\mathscr X)$, there are isomorphisms $\Tot_\t^bA\cong \hocolim_\Z F^{-1}A$ and $ \Tot_\t^bB\cong\hocolim_\Z F^{-1}B$ by inductive hypothesis. Hence $\Tot_\t^bF\mathscr X\cong \hocolim_\Z\mathscr X$.
\end{proof}

Remember that $\t$ is said to be {\bf cosmashing} (resp., {\bf smashing}) if its aisle is closed under products (resp., if its coaisle is closed under coproducts). It is well-known that for a cosmashing $t$-structure $\t$, the heart has exact products, while for a smashing $t$-structure the heart has exact coproducts (this can be either verified by hand or it can be deduced from Corollary \ref{lift_t_struc} and Lemma \ref{loc_is_der}).

\begin{cor}\label{real_b_commutes_prod_coprod}
Let $\{X_i\}_{i\in I}$ be a uniformly bounded family of complexes in $\Ch^b(\A)$, that is, there exist $b\leq a\in \Z$ such that $X_i^n=0$ for all $i\in I$ and for all $n< b$ or $n>a$. The following statements hold true:
\begin{enumerate}
\item if $\t$ is smashing, then $\real_\t^b\left(\coprod_iX_i\right)\cong\coprod_i\real_\t^b X_i$, where the coproduct on the left-hand side is taken in $\sfD^b(\A)$, while the one on the right-hand side is taken in $\D(\bbone)$;
\item if $\t$ is cosmashing, then $\real_\t^b\left(\prod_iX_i\right)\cong\prod_i\real_\t^b X_i$, where the product on the left-hand side is taken in $\sfD^b(\A)$, while the one on the right-hand side is taken in $\D(\bbone)$.
\end{enumerate}
\end{cor}
\begin{proof}
(1) Let us start noticing that, if $\t$ is smashing, so is $\t_B$ (use that each functor of the form $\gr_n$ ($n\in\Z$) commutes with coproducts). Hence, coproducts are the same in $\H_B$ and in $\FD(\bbone)$. Furthermore, since coproducts in $\A$ are exact, the coproduct $\coprod_iX_i$ in $\Ch^b(\A)$ also represents a coproduct in $\sfD^b(\A)$. Consider the equivalence $F\colon \H_B\to \Ch^b(\A)$. Then, 
\begin{align*}
\real_\t^b\left(\coprod_iX_i\right)&=\hocolim_\Z F^{-1}\left(\coprod_iX_i\right)=b^*F^{-1}\left(\coprod_iX_i\right)\\
&=\coprod_ib^*F^{-1}X_i=\coprod_i\hocolim_\Z F^{-1}X_i=\coprod_i\real_\t^b X_i.
\end{align*}

(2) is completely analogous. 
\end{proof}

\subsection{Effa\c{c}ability of $t$-structures}

In this last subsection we give some classical conditions on $\t$ for $\real_\t^b$ to be fully faithful. 

\begin{defn}
We say that $\t=(\D^{\leq0},\D^{\geq0})$ is {\bf effa\c{c}able} provided the following condition holds for all $I\in\Dia$:
\begin{itemize}
\item given objects $X$ and $Y$ in $\A^I$, $n>0$ and  $\phi\colon X\to \Sigma^nY$ in $\D(I)$, there is
an object $Z$ in $\A^I$ and an epimorphism $\psi\colon Z\to X$ in $\A^I$, such that $\phi\psi = 0$.
\end{itemize}
Dually, $\t$ is said to be {\bf co-effa\c{c}able} provided the following condition holds for all $I\in \Dia$:
\begin{itemize}
\item given objects $X$ and $Y$ in $\A^I$, $n >0$ and  $\phi\colon X\to \Sigma^nY$ in $\D(I)$, there is
an object $Z$ in $\A^I$ and a monomorphism $\psi\colon Y\to Z$ in $\A$, such that $\Sigma^n(\psi) \phi= 0$.
\end{itemize}
\end{defn}

\begin{prop}\label{effacable}
The following conditions are equivalent:
\begin{enumerate}
\item[\rm (1)] $\t$ is effa\c{c}able;
\item[\rm (2)] $\t$ is co-effa\c{c}able;
\item[\rm (3)] $\real_\t^b$ is fully faithful.
\end{enumerate}
If the above equivalent conditions hold true, then the essential image of $\real_\t^b$ is exactly $\D^b_\t:=\bigcup_{n\in\N}\D^{\leq n}\cap \D^{\geq-n}$, so that $\real_\t^b$ induces an equivalence
\[
\real_\t^b\colon \sfD^b_\A\overset{\cong}{\longrightarrow} \D^b_\t.
\] 
\end{prop}
\begin{proof}
The equivalence of conditions (1), (2) and (3) is proved in~\cite[Prop.\,3.1.16]{BBD}, see also~\cite[Thm.\,3.10]{Jorge_e_Chrisostomos}, where some parts of the proof are explained in better detail.
\end{proof}

\newpage
\section{Half-bounded realization functor}\label{Sec:half}

In this section we introduce the notion of left and right $\t$-complete derivator, analogously to what has been done in the setting of triangulated categories by Neeman~\cite{Neeman_left_complete} and, in the setting of $\infty$-categories, by Lurie~\cite[Chapter 2]{Lurie_higher_algebra}. We use these notions to extend the bounded realization functor $\real_\t^b\colon \sfD_\A^b\to \D$ to the ``half-bounded" realization functors
\[
\real_\t^-\colon \sfD^-_\A\to \widehat \D\quad\text{and}\quad\real_\t^+\colon \sfD^+_\A\to \widecheck \D,
\]
where $\widehat \D$ and $\widecheck \D$ are the left and right $\t$-completion of $\D$, respectively, while $\sfD^-_\A$ and $\sfD^+_\A$ are the natural prederivators enhancing the right- and left-bounded derived category of $\A$, respectively. We also give sufficient conditions on $\t$ for $\real_\t^-$ and $\real_\t^+$ to commute with co/products and to be fully faithful.

\medskip\noindent
\textbf{Setting for Section \ref{Sec:half}}.
Fix throughout this section a category of diagrams $\Dia$ such that $\N\in\Dia$, a strong and stable derivator $\D\colon \Dia^\op\to \CAT$, a $t$-structure $\t=(\D^{\leq 0},\D^{\geq0})$ on $\D$, and let $\A = \D^\heartsuit(\bbone)$ be its heart. 

\subsection{Left and right $\t$-completeness}\label{derived_complete_subsection}

Let us start with the following definition:

\begin{defn}\label{def_co_completion}
The {\bf left $\t$-completion} $\widehat \D$(resp., the {\bf right $\t$-completion} $\widecheck \D$)  of $\D$ is the sub-prederivator 
\[
\widehat \D\colon \Dia^{\op}\to \CAT\qquad \text{(resp., $\widecheck \D\colon \Dia^{\op}\to \CAT$)}
\]
of $\D^{\N^{\op}}$ (resp., of $\D^{\N}$) such that, for any $I\in \Dia$, $\widehat \D(I)\subseteq \D^{\N^{\op}}(I)$ (resp., $\widecheck \D(I)\subseteq \D^{\N}(I)$) is the full subcategory of those objects $\mathscr X$ such that
\begin{itemize}
\item $\mathscr X_n\in \D^{\geq -n}(I)$ (resp., $\mathscr X_n\in \D^{\leq n}(I)$) for all $n\in\N$;
\item the canonical map $\mathscr X_{n}\to \mathscr X_{n-1}$ (resp., $\mathscr X_n\to \mathscr X_{n+1}$) induces an isomorphism $(\mathscr X_n)^{\geq -n+1}\to \mathscr X_{n-1}$ (resp., $\mathscr X_n\to (\mathscr X_{n+1})^{\leq n}$) in $\D(I)$.
\end{itemize}
\end{defn}

In what follows we are going to give conditions under which $\widehat \D$ is a reflective localization and $\widecheck \D$ is a coreflective colocalization of $\D$.

\begin{lem}\label{tow_and_tel}
The restriction $\holim_{\N^{\op}}\colon \widehat\D\to \D$ (resp., $\hocolim_{\N}\colon \widecheck \D\to \D$) has a left (resp., right) adjoint $\tow\colon \D\to \widehat\D$ (resp., $\tel\colon \D\to \widecheck\D$). 
\end{lem}
\begin{proof}
We give a complete proof for the left completion, the argument for the right completion is completely dual. Consider first the adjoint pair, where $\pt\colon \N^{\op}\to \bbone$ is the obvious functor,
\[
\pt^*\colon \D\rightleftarrows \D^{\N^{\op}}\colon \holim_{\N^{\op}}.
\]
In the notation of Lemma \ref{stair_t_structure_op}, the inclusion $\iota_{\tow}\colon \D_\tow^{\geq 0}\to \D^{\N^{\op}}$ has a left adjoint $\tau^{\geq 0}_\tow$, so we obtain a second pair of adjoints:
\[
\tau^{\geq 0}_\tow\colon \D^{\N^{\op}}\rightleftarrows \D_\tow^{\geq0}\colon \iota_{\tow}.
\]
Putting together the above two adjunctions, we obtain an adjunction 
\[
\D\rightleftarrows \D_\tow^{\geq0},
\] 
where the right adjoint is a restriction of the homotopy limit. It is enough to notice that the essential image of the left adjoint is contained in $\widehat \D$.
\end{proof}

\begin{lem}\label{bounded_in_completion}
Let $\D^b_\t:=\bigcup_{n\in\N}\D^{\leq n}\cap \D^{\geq-n}$; the restriction $\tow\colon \D_\t^b\to \widehat\D$ is fully faithful. Furthermore, consider the full sub-prederivator $\widehat\D_\t^b\subseteq \widehat\D$ such that, given $I\in \Dia$, an object $\mathscr X\in \widehat\D(I)$ belongs to $\widehat\D_\t^b(I)$, if and only if 
\begin{enumerate}
\item $\mathscr X_0\in \D^b_\t$;
\item there exists $n\in\N$, such that $\mathscr X_{k+1}\to \mathscr X_k$ is an iso for all $k\geq n$.
\end{enumerate}
Then, $\widehat \D^b_\t$ is exactly the essential image $\tow(\D_\t^b)$.
\end{lem}
\begin{proof}
It it clear that, given $\mathscr Y\in \D_\t^b(I)$, $\tow(\mathscr Y)$ satisfies the conditions (1) and (2). Furthermore, $\holim_{\N^{\op}}\tow (\mathscr Y)\cong (\tow\mathscr Y)_n\cong \mathscr Y$ (for $n$ as in condition (2) in the statement). So the unit of the adjunction 
\[
\tow: \D_\t^b\rightleftarrows \widehat\D_\t^b:\holim_{\N^{\op}}
\]
is a natural equivalence. To show that also the counit is an equivalence, let $\mathscr X\in  \widehat\D_\t^b(I)$. Then, for $n$ as in condition (2), $\holim_{\N^{\op}}\mathscr X=\mathscr X_n$ and, clearly, $\mathscr X_k=(\mathscr X_n)^{\geq-k}$, for all $k\in\N$. Hence, $\mathscr X\cong \tow\holim_{\N^{\op}}(\mathscr X)$.
\end{proof}

By a dual argument, we can embed $\D_\t^b$ as a full sub-prederivator $\widecheck \D_\t^b$ of $\widecheck \D$.

\begin{defn}
We say that $\t$ is {\bf (countably) $k$-cosmashing} for a non-negative integer $k$ if, for any  (countable) family $\{ X_\alpha\}$ of objects in $\D^{\leq 0}(\bbone)$, the product $\prod  X_\alpha$ belongs to $\D^{\leq k}(\bbone)$. Dually, $\t$ is {\bf (countably) $k$-smashing} if, for any  (countable) family $\{ X_\alpha\}$ of objects in $\D^{\geq 0}(\bbone)$, the coproduct $\coprod X_\alpha$  belongs to $\D^{\geq -k}(\bbone)$.
\end{defn}

In particular, a (countably) $0$-cosmashing $t$-structure is a $t$-structure whose right truncation functor commutes with (countable) products. Similarly, a (countably) $0$-smashing $t$-structure is a $t$-structure whose left truncation functor commutes with (countable) coproducts.

\begin{lem}\label{describe_truncated_product}
Suppose that $\t$ is (countably) $k$-cosmashing and consider a (countable) family $\{X_i\}_{i \in I}$ in $\D(\bbone)$, then
\begin{enumerate}
\item[\rm (1)] $(\prod_{I}X_i)^{\geq-n}\cong (\prod_I X_i^{\geq-n-k})^{\geq-n}$, 
where the product on the left is taken in $\D(\bbone)$, while the product on the right can be taken equivalently in $\D(\bbone)$ or in $\D^+(\bbone):=\bigcup_{n\in\N}\D^{\geq-n}$, since the coaisle $\D^{\geq-n-k}$ is closed under taking products, so $\prod_IX_i^{\geq-n-k}\in \D^{\geq-n-k}(\bbone)$;
\item[\rm (2)] $(\coprod_{I}X_i)^{\geq-n}\cong (\coprod_IX_i^{\geq-n})^{\geq-n}$, where the coproduct on the right can be taken in $\D^+(\bbone)$.
\end{enumerate}
\end{lem}
\begin{proof}
For each $i\in I$, consider the following approximation triangle 
\[
X_i^{<-n-k}\to X_i\to X_i^{\geq-n-k}\to \Sigma X_i^{<-n-k}.
\]
Taking the product of the above triangles, we get 
\[
\prod_IX_i^{<-n-k}\to\prod_I X_i\to \prod_IX_i^{\geq-n-k}\to \prod_I\Sigma X_i^{<-n-k}.
\]
Since $\t$ is $k$-cosmashing, $\prod_IX_i^{<-n-k}\in \D^{\leq-n}(\bbone)$, so Lemma \ref{general_iso_truncations} applies to show that $(\prod_I X_i)^{\geq-n}\cong(\prod_IX_i^{\geq-n-k})^{\geq-n}$. The second statement is easier and follows by the fact that $(-)^{\geq-n}$, being a left adjoint, commutes with coproducts (the coproduct in $\D^{\geq-n}$ is $\tau^{\geq-n}\coprod$, where $\coprod$ here denotes the coproduct in $\D(\bbone)$).
\end{proof}

If, in the above lemma, we start with a family $\{X_i\}_I$ in $\D^{-}(\bbone)$ which is uniformly right-bounded, that is, for which there exists $n\in \N$ such that $X_i^{\geq n}=0$ for all $i$, then its product exists in $\D^-(\bbone)$ (in fact $ \prod_iX_i\cong\prod_iX_i^{<n}\cong(\prod_iX_i)^{<n}\in \D^{<n}(\bbone)\subseteq \D^-(\bbone)$). Furthermore, the same argument of the lemma proves that
\[
\left(\prod_{I}X_i\right)^{\geq-n}\cong \left(\prod_I X_i^{\geq-n-k}\right)^{\geq-n},
\] 
where the product on the right can be taken in $\D^b(\bbone)$. Similarly, 
\[
\left(\coprod_{I}X_i\right)^{\geq-n}\cong \left(\coprod_IX_i^{\geq-n}\right)^{\geq-n}.
\]

\begin{lem}
Suppose that $\t$ is (countably) $k$-(co)smashing for some $k\in\N$ and let $I\in\Dia$, then $\t_I$ has the same property.
\end{lem}
\begin{proof}
Given a (countable) family $\{\mathscr X_\alpha\}$ of objects in $\D^{\leq 0}(I)$, consider the product $\mathscr X:=\prod \mathscr X_\alpha$ in $\D(I)$. Then, $\mathscr X$ belongs to $\D^{\leq k}(I)$ if and only if $i^*(\mathscr X)\in \D^{\leq k}(\bbone)$ for all $i\in I$, but this is true since $i^*\left(\prod \mathscr X_\alpha\right)=\prod i^*(\mathscr X_\alpha)\in \D^{\leq k}(\bbone)$ since $\t$ is (countably) $k$-cosmashing (here we used that, since $i^*$ is both a left and a right adjoint, it commutes with products and coproducts). A similar argument can be used in the smashing case. 
\end{proof}

\begin{lem}\label{lemma_for_completeness}
Suppose that $\t$ is countably $k$-cosmashing for some $k\in\N$ and let $\mathscr X\in \widehat\D(\bbone)$; for any $n\in\N$, there is a triangle
\[
\xymatrix{
 K_n\ar[r]& \holim_{\N^{\op}}\mathscr X\ar[r]& \mathscr X_n\ar[r]& K_n
}
\]
with $ K_n\in \D^{\leq-n-1}(\bbone)$.
\end{lem}
\begin{proof}
Given $n_1\leq n_2\in \N$, denote by $d_{n_1,n_2}\colon \mathscr X_{n_2}\to \mathscr X_{n_1}$ the map in the underlying diagram of $\mathscr X$. Notice that, for any $n\in \N$, there is a natural morphism
\[
\pi_n\colon \holim_{\N^{\op}}\mathscr X\to \mathscr X_n,
\]
obtained applying the functor $n^*$ to the counit $\pt^*_{\N^{\op}}\holim_{\N^{\op}}\mathscr X\to \mathscr X$ and, by construction, $\pi_{n_1}=d_{n_1,n_2}\pi_{n_2}$. Now, given $n\in\N$, we apply the octahedral axiom to the composition $\pi_{n}=d_{n,n+k+1}\pi_{n+k+1}$, obtaining the following diagram whose rows and columns are all triangles:
\[
\xymatrix@C=15pt@R=15pt{
&  (\mathscr X_{n+k+1})^{\leq -n-1}\ar[d]\ar@{=}[r]& (\mathscr X_{n+k+1})^{\leq -n-1}\ar[d]\\
\holim_{\N^{\op}} \mathscr X\ar[r]\ar@{=}[d]&\mathscr X_{n+k+1}\ar[r]\ar[d]&\Sigma  K_{n+k+1}\ar[d]\ar[r]&\Sigma \holim_{\N^{\op}} \mathscr X\ar@{=}[d]\\
\holim_{\N^{\op}} \mathscr X\ar[r]&\mathscr X_n\ar[d]\ar[r]&\Sigma  K_n\ar[d]\ar[r]&\Sigma \holim_{\N^{\op}}\mathscr  X\\
&\Sigma(\mathscr X_{n+k+1})^{\leq -n-1}\ar@{=}[r]&\Sigma (\mathscr X_{n+k+1})^{\leq -n-1}
}
\]
To conclude we should prove that $ K_n\in \D^{\leq-n-1}(\bbone)$ and, for that, it is enough to verify that $ K_{n+k+1}\in \D^{\leq-n-1}(\bbone)$. Consider now the obvious inclusion $\iota_{n+k+1}\colon (\ordn+\ordk+\bbone)^{\op}\to \N^{\op}$, let $\mathscr X_{(\leq n+k+1)}:=(\iota_{n+k+1})_*(\iota_{n+k+1})^*\mathscr X$, consider the unit $\Phi_{\mathscr X}\colon \mathscr X\to \mathscr X_{(\leq n+k+1)}$ and complete it to a triangle in $\D(\N^{\op})$:
\[
 {\mathscr K}\to \mathscr X\to \mathscr X_{(\leq n+k+1)}\to\Sigma {\mathscr K}
\]
Applying $\holim_{\N^{\op}}$ we get the following triangle 
\[
\holim_{\N^{\op}} {\mathscr K}\to \mathscr X\to \mathscr X_{n+k+1}\to\Sigma\holim_{\N^{\op}}  {\mathscr K},
\] 
in fact, $(\mathscr X_{(\leq n+k+1)})_h=\mathscr X_h$ provided $h\leq n+k$, while $(\mathscr X_{(\leq n+k+1)})_h=\mathscr X_{n+k+1}$, if $h\geq n+k+1$. We obtain an isomorphism $\holim_{\N^{\op}} {\mathscr K}\cong  K_{n+k+1}$, from which we deduce that there is a triangle of the form:
\[
 K_{n+k+1}\to \prod_{h\in \N}  {\mathscr K}_h\to \prod_{h\in \N} {\mathscr K}_h\to \Sigma  K_{n+k+1}
\]
Furthermore, for any $h\in\N$, there is a distinguished triangle $ {\mathscr K}_h\to \mathscr X_h\to (\mathscr X_{(\leq n+k+1)})_h\to  {\mathscr K}_h$, so that $ {\mathscr K}_h=0$ if $h\leq n+k+1$, and $ {\mathscr K}_h\cong \mathscr X_h^{\leq -n-k-2}$ otherwise. 
Since $\t$ is countably $k$-cosmashing, $\prod_{h\in \N}  {\mathscr K}_h\in \D^{\leq -n-2}(I)$ and so $K_{n+k+1}\in \D^{\leq -n-1}(\bbone)$, as desired.
\end{proof}

\begin{thm}
If $\t$ is countably $k$-cosmashing for some $k\in\N$, then $\widehat \D$ is a reflective localization of $\D$ and so it is itself a pointed, strong derivator. Furthermore, under these hypotheses, $\widehat\D$ is stable and, letting
\begin{align*}
\widehat\D^{\leq0}(\bbone)&:=\{\mathscr X\in \widehat\D(\bbone):\mathscr X_n\in\D^{\leq0}(\bbone),\ \forall n\in\N\}\quad\text{and}\\
\widehat\D^{\geq0}(\bbone)&:=\{\mathscr X\in \widehat\D(\bbone):\mathscr X_n\in\D^{\geq0}(\bbone),\ \forall n\in\N\},
\end{align*}
$\widehat\t:=(\widehat\D^{\leq0}(\bbone),\widehat\D^{\geq0}(\bbone))$ is a $t$-structure on $\widehat\D(\bbone)$.
\end{thm}
\begin{proof}
Denote by $\varepsilon\colon \tow \circ \holim_{\N^{\op}}\to \id_{\widehat \D}$  the counit of the adjunction $(\tow,\holim_{\N^{\op}})$. Given $I\in \Dia$, we have to show that $\varepsilon_I$ is a natural isomorphism. Indeed, given $\mathscr X\in \widehat\D(I)$ and $\varepsilon_{\mathscr X}\colon \tow(\holim_{\N^{\op}} \mathscr X)\to\mathscr X$, we have to check that $n^*(\varepsilon_{\mathscr X})$ is an isomorphism for all $n\in\N$, that is, the natural map $\pi_n\colon \holim_{\N^{\op}} \mathscr X\to \mathscr X_n$ induces an isomorphism $(\holim_{\N^{\op}} \mathscr X)^{\geq -n}\to \mathscr X_n$, which is clear by Lemma \ref{lemma_for_completeness}.

Denote by $(\Sigma, \Omega)$ (resp., $(\widehat \Sigma,\widehat \Omega)$, $(\Sigma_{\N^{\op}}, \Omega_{\N^{\op}})$) the suspension-loop adjunction in $\D(\bbone)$ (resp., $\widehat\D(\bbone)$, $\D({\N^{\op}})$). To show that $\widehat\D$ is stable we have to show that $\widehat \Omega$ is an equivalence and we know, by Lemma \ref{stability_of_loc}, that it is fully faithful. Let $\sigma\colon \N^{\op}\to \N^{\op}$ be the map $n\mapsto n+1$; we claim that $\sigma^*\Sigma_{\N^{\op}}\mathscr X\in \widehat\D(\bbone)$ and $\widehat \Omega \sigma^*\Sigma_{\N^{\op}}\mathscr X\cong \mathscr X$. In fact,
\begin{align*}
\widehat \Omega \sigma^*\Sigma_{\N^{\op}}\mathscr X&\cong \tow\, \Omega\, \holim_{\N^{\op}}\sigma^*\Sigma_{\N^{\op}}\mathscr X\\
&\cong \tow\, \Omega\, \holim_{\N^{\op}}\Sigma_{\N^{\op}}\mathscr X\\
&\cong \tow\, \Omega\, \Sigma\, \holim_{\N^{\op}}\mathscr X\\
&\cong \tow\, \holim_{\N^{\op}}\mathscr X\\
&\cong\mathscr  X.
\end{align*}
where $\holim_{\N^{\op}}\sigma^*\cong \holim_{\N^{\op}}$ by the dual of~\cite[Prop.\,1.24]{Moritz}, using that $\sigma$ is a left adjoint, while $\holim_{\N^{\op}}\Sigma_{\N^{\op}}\cong \Sigma\, \holim_{\N^{\op}}$ since left exact (e.g., right adjoint) morphisms of pointed derivators commute with suspensions. 

Finally, it is clear that $\widehat\D^{\leq0}(\bbone)$ is extension closed and that, given $\mathscr X\in \widehat\D^{\leq0}(\bbone)$ and $n\in\N$, $n^*\widehat\Sigma \mathscr X=n^*\sigma^*\Sigma_{\N^{\op}}\mathscr X=(n+1)^*\Sigma_{\N^{\op}}\mathscr X=\Sigma\mathscr X_{n+1}\in \D^{\leq-1}(\bbone)\subseteq \D^{\leq 0}(\bbone)$, so $\widehat\Sigma \widehat\D^{\leq0}(\bbone)\subseteq \widehat\D^{\leq0}(\bbone)$. Furthermore, let $\mathscr X\in \widehat\D^{\leq0}(\bbone)$ and $\mathscr Y\in \widehat\D^{\geq1}(\bbone)$, then
\begin{align*}
\widehat\D(\bbone)(\mathscr X,\mathscr Y)&=\D(\N^{\op})(\mathscr X,\mathscr Y)\\
&=\D(\bbone)^{\N^{\op}}(\dia_{\N^{\op}}\mathscr X,\dia_{\N^{\op}}\mathscr Y)\\
&=0
\end{align*}
where the second equality holds by Lemma \ref{lem_epivalence_N_op}.
Let now $\mathscr X\in \widehat \D(\bbone)$. Notice that the object $\mathscr X^{\geq1}:=0_*(\mathscr X_0)^{\geq1}\in \D(\N^{\op})$ still belongs to $\widehat\D(\bbone)$ and it has a canonical morphism $\mathscr X\to \mathscr X^{\geq 1}$ (just compose the unit $\mathscr X\to 0_*\mathscr X_0$ with the image under $0_*$ of the morphism $\mathscr X_0\to \mathscr X_0^{\geq1}$); we are finished if we can prove that the fiber of this morphism is in $\widehat\D^{\leq0}(\bbone)$, but this is clear since, for any $n\in\N$, $(\mathscr X_n)^{\geq1}\cong (\mathscr X_0)^{\geq1}$. 
\end{proof}

\begin{defn}
In the notation of Proposition \ref{tow_and_tel}, we say that $\D$ is {\bf left $\t$-complete} if the adjunction 
\[
\tow\colon \D\rightleftarrows \widehat\D\colon \holim_{\N^{\op}}
\]
is an equivalence. Dually, $\D$ is {\bf right $\t$-complete} if the adjunction
\[
\hocolim_{\N}\colon \widecheck\D\rightleftarrows \D\colon \tel
\] 
is an equivalence.
\end{defn}

Notice that the above definition is slightly stronger than that given by Neeman in~\cite{Neeman_left_complete}, in that Neeman requires the unit $\id_{\D(\bbone)}\to \holim_{\N^{\op}} \tow$ to be an iso, but does not include that the corresponding counit is an iso in the definition of left-completeness. Furthermore, in case $\t$ is countably $k$-cosmashing for some $k\in\N$, so that $\widehat \D$ is stable, then $\tow\colon \D\rightleftarrows \widehat\D\colon \holim_{\N^{\op}}$ is $t$-exact, that is, $\tow(\D^{\leq0})\subseteq \widehat \D^{\leq0}$, $\tow(\D^{\geq0})\subseteq \widehat \D^{\geq0}$ and so, whenever $\t$ is also a left-complete, we can deduce:  $\holim_{\N^{\op}}(\widehat\D^{\leq0})\subseteq  \D^{\leq0}$ and $\holim_{\N^{\op}}(\widehat\D^{\geq0})\subseteq  \D^{\geq0}$. Dual observations can be done about right-completeness.

\begin{prop}\label{complete_proposition}
Given $k\in \N$, suppose $\t$ is countably $k$-cosmashing. The following are equivalent:
\begin{enumerate}
\item[\rm (1)] $\D$ is left $\t$-complete;
\item[\rm (2)] $\t$ is left non-degenerate.
\end{enumerate}
\end{prop}
\begin{proof}
We prove first that (1) implies (2). Indeed, if $\t$ is left complete, then any object $X\in \D(\bbone)$ fits into a triangle $X\to \prod_{n\in\N} X^{\geq -n}\to \prod_{n\in\N} X^{\geq -n}\to \Sigma X$. Furthermore, given $X\in \bigcap_{n\in\N} \D^{\leq -n}(\bbone)$, clearly $X^{\geq -n}=0$ for all $n\in \N$ and so, by the above triangle, also $X=0$. 
\\
To show that (2) implies (1), denote by $\eta\colon \id_{\D(\bbone)}\to \holim_{\N^{\op}} \tow$ the unit of the adjunction $(\tow,\holim_{\N^{\op}})$, consider $X\in \D(\bbone)$ and let us show that $\eta_X\colon X\to \holim_{\N^{\op}}(\tow (X))$ is an isomorphism. We have already proved that $\tow(\eta_X)$ is an isomorphism, so that $\eta_X$ induces isomorphisms $X^{\geq-n}\to (\holim_{\N^{\op}}(\tow (X)))^{\geq-n}$ for all $n\in\N$. Thus, if we complete $\eta_X$ to a triangle
\[
K\to X\to \holim_{\N^{\op}}(\tow (X))\to \Sigma K,
\]
we get that $K^{\geq-n}=0$, for all $n\in\N$. As a consequence, $K\in \bigcap_{n\in\N}\D^{\leq -n}(\bbone)$, so $K=0$ by (2).
\end{proof}

\subsection{The half-bounded filtered prederivators}

Let us start with the following definition:

\begin{defn}\label{half_filtered_der_cat}
Given $i,\, j\in \N$, consider the following functors:
\begin{align*}
(i,\Z)\colon \bbone\times \Z&\to \N^{\op}\times \Z & (j,\Z)\colon \bbone\times \Z&\to \N\times \Z\\
(0,n)&\mapsto (j,n) & (0,n)&\mapsto (j,n).
\end{align*} 
(Notice that we are using the same notation for different functors, but the meaning of this notation will always be clear from the context).
Given $ \mathscr X\in \D^{\N^\op\times\Z}(I)$ and  $ \mathscr Y\in \D^{\N\times\Z}(I)$, we let $\mathscr X_{(i,\Z)}:=(i,\Z)^*\mathscr X\in \D^{\Z}(I)$ and $\mathscr Y_{(j,\Z)}:=(j,\Z)^*\mathscr Y\in \D^{\Z}(I)$.  We define the {\bf right-bounded} (resp., {\bf left-bounded}) {\bf filtered prederivator} 
\[
\FDp\colon \Dia^{\op}\to \CAT \qquad (\text{resp.,}\ \FDm\colon \Dia^{\op}\to \CAT )
\] 
as a full sub-prederivator of $\D^{\N^\op\times\Z}$ (resp., of $\D^{\N\times\Z}$) such that, given $I\in\Dia$, $\FDp(I)$ (resp., $\FDm(I)$) is the full category  spanned by those $\mathscr X$ such that, for any $i\in \N$, the object $\mathscr X_{(i,\Z)}$ belongs to $\FD(I)$. 
\end{defn}

As for $\FD$, in general, $\FDp$ and $\FDm$ are not derivators but, for any $I\in \Dia$, $\FDp(I)$ and $\FDm$ are full triangulated subcategories of $\D^{\N^\op\times\Z}(I)$ and $\D^{\N\times \Z}(I)$, respectively. In particular, it makes sense to consider triangles and to speak about $t$-structures in $\FDp(I)$ and $\FDm(I)$, for some $I\in \Dia$. 

\begin{cor}\label{lift_NZ}
Given $I\in\Dia$, the functor $\dia_{\N^\op}\colon \D^{\N^\op\times \Z}(I)\to (\D^{\Z}(I))^{\N^\op}$ induces a functor
\[
\dia_{\N^\op}\colon \FDp(I)\to (\FD(I))^{\N^\op}
\]
which is full and essentially surjective. Furthermore, given $\mathscr X,\,\mathscr Y \in \FDp(I)$, the diagram functor induces an isomorphism
\[
\FDp(I)(\mathscr X,\mathscr Y)\cong \FD(I)^{\N^\op}(\dia_{\N^\op}\mathscr X,\dia_{\N^\op}\mathscr Y)
\]
provided $\FD(I)(\Sigma \mathscr X_{(i+1,\Z)},\mathscr Y_{(i,\Z)})=0$ for all $i\in\N$.
\end{cor}
\begin{proof}
Applying Lemma \ref{lem_epivalence_N_op} to the shifted stable derivator $\D^\Z$, one can see that the functor $\dia_{\N^\op}\colon \D^{\N^\op\times \Z}(I)=(\D^\Z)^{\N^\op}(I)\to \D^\Z(I)^{\N^\op}$ is full and essentially surjective. Furthermore, the essentially image of the restriction of $\dia_{\N^\op}$ to $\FDp(I)$ is, by construction, exactly $(\FD(I))^{\N^\op}$. This proves the first part of the statement.

To verify the second part, let $\mathscr X,\,\mathscr Y \in \FDp(I)\subseteq  \D^{{\N^\op}\times \Z}(I)$, then, again by Lemma \ref{lem_epivalence_N_op}, the diagram functor induces an isomorphism
\begin{align*}
\FDp(I)(\mathscr X,\mathscr Y)&=\D^{\N^\op\times \Z}(I)(\mathscr X,\mathscr Y)\\
&\cong \D^\Z(I)^{\N^\op}(\dia_{\N^\op}\mathscr X,\dia_{\N^\op}\mathscr Y)\\
&= (\FD(I))^{\N^\op}(\dia_{\N^\op}\mathscr X,\dia_{\N^\op}\mathscr Y)
\end{align*}
if $\FD(I)(\Sigma \mathscr X_{(i+1,\Z)},\mathscr Y_{(i,\Z)})=\D^\Z(I)(\Sigma \mathscr X_{(i+1,\Z)},\mathscr Y_{(i,\Z)})=0$ for all $i\in\N$. 
\end{proof}

\begin{prop}\label{heart_is_Ch-}
Consider the following two classes of objects in $\FDp(\bbone)$:
\begin{align*}
\FDp^{\leq0}(\bbone)&=\{\mathscr X\in \FDp(\bbone):\mathscr X_{(i,\Z)}\in \FD^{\leq 0}(\bbone),\, \forall i\in \N\},\\
\FDp^{\geq0}(\bbone)&=\{\mathscr X\in \FDp(\bbone):\mathscr X_{(i,\Z)}\in \FD^{\geq 0}(\bbone),\, \forall i\in\N\}.
\end{align*}
Then $\t^-_{B}:=(\FDp^{\leq0}(\bbone),\FDp^{\geq0}(\bbone))$ is a $t$-structure on $\FDp(\bbone)$ whose heart is equivalent to $\Ch^b(\A)^{\N^\op}$.
\end{prop}
\begin{proof}
Let us verify the axioms of a $t$-structure for $\t^-_B$. Indeed, the classes $\FDp^{\leq0}(\bbone)$ and $\FDp^{\geq0}(\bbone)$ have the desired closure properties since all the $(i,\Z)^*$ are triangulated functors, and since $\FD^{\leq 0}(\bbone)$ and  $\FD^{\geq 0}(\bbone)$ are an aisle and a coaisle, respectively. Furthermore, given $\mathscr X\in \FDp^{\leq0}(\bbone)$ and $\mathscr Y\in \FDp^{\geq 1}(\bbone)$, then, by Corollary \ref{lift_NZ}, 
\[
\FDp(I)(\mathscr X,\mathscr Y)= \FD(I)^{\N^\op}(\dia_{\N^\op}\mathscr X,\dia_{\N^\op}\mathscr Y),
\]
in fact, $\FD(I)(\Sigma \mathscr X_{(i+1,\Z)},\mathscr Y_{(i,\Z)})=0$ for all $i\in\N$, since $\Sigma \mathscr X_{(i+1,\Z)}\in\FD^{\leq-1}$ and $\mathscr Y_{(i,\Z)}\in \FD^{\geq1}$. Furthermore,
\[
\FD(I)^{\N^\op}(\dia_{\N^\op}\mathscr X,\dia_{\N^\op}\mathscr Y)\subseteq \prod_{i\in\N}\FD(I)(\mathscr X_{(i,\Z)},\mathscr Y_{(i,\Z)})=0,
\]
since $\mathscr X_{(i,\Z)}\in\FD^{\leq 0}$ and $\mathscr Y_{(i,\Z)}\in \FD^{\geq1}$. This shows that $\FDp^{\leq0}(\bbone)$ is left orthogonal to $\FDp^{\geq1}(\bbone)$. Let now $\mathscr Z\in \FDp(I)$ and consider $\dia_{\N^\op}\mathscr Z\in \FD(I)^{\N^\op}$. Then, there is a map $\phi\colon X\to \dia_{\N^\op}\mathscr Z$ in $\FD(I)^{\N^\op}$ such that, for any $i\in \N$, $\phi_i\colon X(i)\to \mathscr Z_{(i,\Z)}$ is the left $\t_B^b$ truncation of $\mathscr Z_{(i,\Z)}$ in $\FD(I)$. By Corollary \ref{lift_NZ}, there is a morphism $\Phi\colon \mathscr X\to \mathscr Z$ such that $\dia_{\N^\op}(\Phi)=\phi$. It is now easy to show that the following diagram, obtained by completing $\Phi$ to a triangle:
\[
\mathscr X\overset{\Phi}\longrightarrow \mathscr Z\to \mathscr Y\to \Sigma \mathscr X,
\] 
is a truncation triangle of $\mathscr Z$ with respect to $\t_B^-$. This shows that $\t_B^-$ is a $t$-structure. 

To conclude we should verify that $\FDp^{\leq0}(\bbone)\cap \FDp^{\geq0}(\bbone)\cong \Ch^b(\A)^{\N^\op}$. By Proposition \ref{heart_equiv_prop}, $\FD^{\leq0}(\bbone)\cap \FD^{\geq0}(\bbone)\cong \Ch^b(\A)$, so clearly the functor $\dia_{\N^\op}$ induces an essentially surjective and full functor
\[
\FDp^{\leq0}(\bbone)\cap \FDp^{\geq0}(\bbone)\to \Ch^b(\A)^{\N^\op}.
\]
To show that this functor is  faithful notice that, given $\mathscr X,\, \mathscr Y\in \FDp^{\leq0}(\bbone)\cap \FDp^{\geq0}(\bbone)$, then $\Sigma \mathscr X_{(i+1,\Z)}\in\FD^{\leq-1}$ and $\mathscr Y_{(i,\Z)}\in \FD^{\geq0}$, so that 
\[\FD(\bbone)(\Sigma \mathscr X_{(i+1,\Z)},\mathscr Y_{(i,\Z)})=0\] for all $i\in\N$, showing that faithfulness follows by Corollary \ref{lift_NZ}.
\end{proof}

The above corollary and proposition can be easily dualized to obtain analogous results for $\FDm$. 

\begin{defn}
The $t$-structure $\t^-_B:=(\FDp^{\leq0},\FDp^{\geq0})$ on $\FDp$ (resp., $\t^+_B:=(\FDm^{\leq0},\FDm^{\geq0})$ on $\FDm$) described in the above proposition is said to be the {\bf right-bounded} (resp., {\bf left-bounded}) {\bf Beilinson $t$-structure} induced by $\t$. 
\end{defn}

\subsection{The half-bounded totalization functors}\label{half_bounded_subs}
Consider the full sub-precategory $\Tow_\A$ of $(\Ch_\A^+)^{\N^\op}$ where, for any $I\in \Dia$, $\Tow_\A(I)$ is spanned by those $X\colon \N^\op\to \Ch^+(\A^I)$ such that $X(n)\in \Ch^{\geq -n}(\A^I)$, for any $n\in\N$. Notice that, without specific hypotheses on $\A$, we do now have a model category on $\Tow(\A^I)$ (like in Subsection \ref{complete_subsection_tower_tel_complexes}). On the other hand, we can still consider the following adjunction of prederivators
\[
\tow\colon \Ch_\A\rightleftarrows \Tow_\A\colon \lim_{\N^\op}.
\]
If we start with a bounded above complex $X\in \Ch^-(\A^I)$, then the associated tower $\tow(X)$ actually belongs to $(\Ch^b(\A^I))^{\N^\op}$. Hence, letting $\Tow^-_\A(I):= (\Ch^b_\A(I))^{\N^\op}\cap \Tow_\A(I)$, we can restrict the above adjunction:
\[
\tow^-\colon \Ch^-_\A\rightleftarrows \Tow^-_\A\colon \lim_{\N^\op}.
\]
Dually, we obtain an adjunction
\[
\colim_{\N}\colon \Tel^+_\A\rightleftarrows \Ch^+_\A\colon \tel^+.
\]

\begin{thm}\label{heart_is_ch+}
Let $\t^-_B=(\FDp^{\leq0},\FDp^{\geq0})$ and $\t^+_B=(\FDm^{\leq0},\FDm^{\geq0})$ be the half-bounded Beilinson $t$-structures on $\FDp$ and $\FDm$, respectively. Consider the following  morphisms of prederivators:
\[
\xymatrix@R=0pt{
\Tot^-_\t\colon\Ch^-_\A\ar[rr]^(.4){\tow^-}&& (\Ch^b_\A)^{\N^\op}\subseteq \FDp\subseteq \D^{\N^\op\times \Z}\ar[rr]^(.7){\hocolim_{\Z}}&& \D^{\N^\op}.
\\
\Tot^+_\t\colon\Ch^+_\A\ar[rr]^(.4){\tel^+}&& (\Ch^b_\A)^{\N}\subseteq \FDm\subseteq \D^{\N\times \Z}\ar[rr]^(.7){\hocolim_{\Z}}&& \D^{\N}.
}
\]
where the inclusions $(\Ch^b_\A)^{\N^\op}\subseteq \FDp$ and $(\Ch^b_\A)^{\N}\subseteq \FDm$ are possible because $(\Ch^b_\A)^{\N^\op}$ and $(\Ch^b_\A)^{\N}$ are equivalent to the heart of $\t^-_B$ and $\t^+_B$, respectively, by Theorem \ref{heart_is_Chb}.
Then the following statements hold:
\begin{enumerate}
\item $\Tot^-_\t$ takes values in $\widehat \D\subseteq \D^{\N^\op}$, while $\Tot^+_\t$ takes values in $\widecheck \D\subseteq \D^{\N}$;
\item given a quasi-isomorphism $\phi$ in $\Ch^-({\A^{I}})$ (resp., $\Ch^+({\A^{I}})$) for some $I\in\Dia$, then $\Tot_\t^-\phi$ (resp., $\Tot_\t^+\phi$) is an isomorphism.
\end{enumerate} 
\end{thm}
\begin{proof}
We prove our statements just in the right-bounded case since the left-bounded case is completely dual.

(1) Let $I\in \Dia$ and $X\in \Ch^-(\A^I)$. Given $n\in\N$,
\[
(\hocolim_{\Z}\tow^-X)_n\cong \hocolim_{\Z}(\tow^-X)_n\cong \hocolim_{\Z}(\tau^{\geq -n}X),
\]
by~\cite[Prop.\,2.6]{Moritz}. Hence, the underlying diagram of $\Tot_\t^-X$ is given by
\[
\Tot_\t^b\tau^{\geq0}X\leftarrow \Tot_\t^b\tau^{\geq-1}X\leftarrow \cdots\leftarrow \Tot_\t^b\tau^{\geq- n}X\leftarrow\cdots,
\]
so that $\Tot_\t^-X\in\widehat \D(I)$ by the already verified properties of $\Tot_\t^b$. 

(2) Let $I\in\Dia$ and $\phi\colon X\to Y$ be a quasi-isomorphism in $\Ch^-(\A^I)$. 
We have to verify that $(\hocolim_{\Z}\tow^-\phi)_n$ is an isomorphism for any $n\in\N$. By~\cite[Prop.\,2.6]{Moritz}, 
\[
(\hocolim_{\Z}\tow^-\phi)_n\cong \hocolim_{\Z}(\tow^-\phi)_n\cong\hocolim_{\Z}\tau^{\geq -n}\phi .
\]
To conclude recall that we know by Proposition \ref{real_is_tot_prop} that $\hocolim_{\Z}\tau^{\geq -n}\phi=\Tot_\t^b(\tau^{\geq -n}\phi)$, which is an iso since $\tau^{\geq -n}\phi$ is a quasi-isomorphism.
\end{proof}

Suppose now that the Abelian category $\A$ has enough injectives and it is (Ab.$4^*$)-$k$ for some $k\in\N$. Then, we have seen in Subsection \ref{complete_subsection_tower_tel_complexes} that is makes sense to consider the prederivator $\sfD_\A$ where, for any $I\in \Dia$, 
\[
\sfD_\A(I)=\sfD(\A^{I})=\Ch(\A^{I})[\W_I^{-1}]
\]
where $\W_I$ is the class of quasi-isomorphisms in $\Ch(\A^{I})$. Consider now the following sub-prederivator $\sfD_\A^-\subseteq \sfD_\A$
\[
\sfD_\A^-\colon \Dia^{\op}\to \CAT
\]
such that $\sfD_\A^-(I)$ is spanned by those complexes $X$ such that $H^n(X)=0$ for all $n>>0$. One can prove that, for any $I\in \Dia$, 
\[
\sfD^-_\A(I)=\sfD^-(\A^{I})=\Ch^-(\A^{I})[(\W^-_I)^{-1}]
\]
where $\W^-_I$ is the class of quasi-isomorphisms in $\Ch^-(\A^{I})$ (for this see, for example,~\cite[\href{http://stacks.math.columbia.edu/tag/05RW}{Tag 05RW}]{stacks-project}). Hence, a consequence of Theorem \ref{heart_is_ch+} is that the corestriction $\Tot_\t^-\colon \Ch^-_\A\to \widehat\D$ factors uniquely through the quotient $\Ch^-_\A\to \sfD^-_\A$, giving rise to a morphism of prederivators 
\[
\real_\t^-\colon \sfD^-_\A\to \widehat \D.
\]
Applying a dual argument to the case when $\A$ has enough projectives and it is (Ab.$4$)-$k$ for some $k\in\N$, we obtain the following morphism of prederivators 
\[
\real_\t^+\colon \sfD^+_\A\to \widehat \D.
\]

\medskip
To conclude this subsection let us explain how the half-bounded realization functors can be considered as extensions of the bounded realization functor. Indeed, in the situation of Theorem \ref{heart_is_ch+}, consider the following restrictions to bounded complexes
\[
\xymatrix@R=0pt{
\Tot^-_\t\colon\Ch^b_\A\ar[rr]^(.4){\tow^-}&& (\Ch^b_\A)^{\N^\op}\subseteq \FDp\subseteq \D^{\N^\op\times \Z}\ar[rr]^(.7){\hocolim_{\Z}}&& \widehat\D.
\\
\Tot^+_\t\colon\Ch^b_\A\ar[rr]^(.4){\tel^+}&& (\Ch^b_\A)^{\N}\subseteq \FDm\subseteq \D^{\N\times \Z}\ar[rr]^(.7){\hocolim_{\Z}}&& \widecheck\D.
}
\]
It is not difficult to show that these restrictions of $\Tot^-_\t$ and $\Tot^+_\t$ take values in $\widehat\D^b_\t\cong \D^b_\t$ and $\widecheck\D^b_\t\cong \D^b_\t$, respectively (see Lemma \ref{bounded_in_completion}). Identifying $\D^b_\t$ with $\widehat\D^b_\t$ (resp., $\widecheck\D^b_\t$), one can show that the restriction of $\Tot_\t^-$ (resp., $\Tot^+_\t$) to bounded complexes is conjugated to $\Tot_\t^b$.

\subsection{Conditions on $\real_\t^-$ and $\real_\t^+$ to be fully faithful}
We have seen in the the previous subsection that, whenever the heart $\A$ has enough injectives and it is (Ab.$4^*$)-$k$ for some $k\in\N$, then we can construct a suitable morphism of prederivators $\real_\t^-\colon \sfD^-_\A\to \widehat \D$ so, composing it with $\holim_{\N^\op}\colon  \widehat \D\to \D$ we get a morphism of prederivators $\sfD^-_\A\to \D$. In the following proposition we give some general conditions to make this morphism fully faithful. 

\begin{prop}\label{prop_fully_minus}
Suppose that $\t$ satisfies the following conditions:
\begin{enumerate}
\item $\t$ is (co)effa\c cable;
\item $\t$ is left non-degenerate;
\item $\t$ is $k$-cosmashing for some $k\in\N$;
\item $\t$ is ($0$-)smashing;
\item the heart $\A$ of $\t$ has enough injectives.
\end{enumerate} 
Then, there is an exact and fully faithful morphism of prederivators
\[
\real_\t^-\colon \sfD^-_\A\to \D
\] 
that commutes with products and coproducts of uniformly right-bounded families of objects. Furthermore, the essential image of $\real_\t^-$ is $\D^-:=\bigcup_{n\in\N}\D^{\leq n}$.
\end{prop}
\begin{proof}
%
%\NB\NB CHANGE THIS PROOF \NB\NB\NB\NB
%
%Furthermore, notice that the following functors preserve coproducts (in the sense that, given a family of objects that admit a coproduct in the source category, their images admit a coproduct in the target category, and such coproduct is naturally isomorphic to the image of the coproduct in the source category):
%\begin{enumerate}
%\item the morphism $\tow^-\colon \Ch^-_\A\to \Tow^-_\A\subseteq (\Ch^b_\A)^{\N^\op}$, since coproducts are exact in $\A$;
%\item the inclusion $(\Ch^b_\A)^{\N^\op}\to \FDp$ commutes with coproducts. Again, this fact is a consequence of the fact that $\t$ is smashing;
%\item the morphism $\hocolim_\Z\colon \FDp\to \widehat \D$ commutes with coproducts. In fact, it is enough to show that, for any $n\in\N$, $\hocolim_\Z n^*\colon \FDp\to \FD\to \D$ commutes with coproducts by the same argument of the proof of Corollary \ref{real_b_commutes_prod_coprod};
%\item the quotient $\Ch^-_\A\to \sfD^-_\A$ commutes with coproducts, since coproducts are exact in $\A$.
%\end{enumerate}
%By the above discussion, it is easy to see that $\real_\t^-\colon  \sfD^-_\A\to  \widehat \D$ commutes with coproducts. Furthermore, by Proposition \ref{effacable} and our hypothesis that $\t$ is (co)effa\c cable, we deduce that the restriction of $\real_\t^-$ to $ \sfD^b_\A\subseteq \sfD^-_\A$ is fully faithful. By an infinite d\'evissage, we get that $\real_\t^-$ is fully faithful.
%\end{proof}
%
%
%
%
%Consider the functor 
%\[
%\real_\t^-\colon \sfD^-(\A)\to\D(\bbone) 
%\]
%and remember that, 
With the properties we have already verified, it is not difficult to show that, essentially by construction, for any $n\in\N$, there is a commutative diagram as follows:
\[
\xymatrix{
\sfD^-(\A)\ar[r]^{\real_\t^-}\ar[d]_{\tau^{\geq-n}}&\D(\bbone)\ar[d]^{\tau^{\geq-n}}\\
\sfD^b(\A)\ar[r]|\cong^{\real^b_\t}&\D^b(\bbone)\\
}
\]
where the truncation on the left is the truncation with respect to $\t_\A$, while the truncation on the right is taken with respect to $\t$.
Suppose now that we have a uniformly right-bounded family $\{X_i\}_I\subseteq \sfD^-(\A)$. We are going to prove that
\[
\real_\t^-\prod_IX_i\cong \prod_I\real_\t^-X_i\quad \text{and}\quad \real_\t^-\coprod_IX_i\cong \coprod_I\real_\t^-X_i.
\]
We start with products. Since, by (the dual of) Proposition \ref{complete_proposition}, we know that the hypotheses (2) and (3) on $\t$ imply that  $\tow:\D\rightleftarrows \widehat\D :\holim_{\N^\op}$ is an equivalence, it is enough to show that $\tau^{\geq-n}\real_\t^-\prod_IX_i\cong \tau^{\geq-n}\prod_I\real_\t^-X_i$, for all $n$:
\begin{align*}
\tau^{\geq-n}\real_\t^-\prod_IX_i&\cong\real_\t^b\tau^{\geq-n}\prod_IX_i&\qquad&\text{constr.\ of $\real_\t^-$;}\\
&\cong\real_\t^b\tau^{\geq-n}\prod_I\tau^{\geq-n-k}X_i&&\text{Lem.\ref{describe_truncated_product};}\\
&\cong\tau^{\geq-n}\real_\t^b\prod_I\tau^{\geq-n-k}X_i&&\text{$\real_\t^b$ is $t$-exact;}\\
&\cong\tau^{\geq-n}\prod_I\real_\t^b\tau^{\geq-n-k}X_i&&\text{$\real_\t^b$ comm. with $\prod$;}\\
&\cong\tau^{\geq-n}\prod_I\tau^{\geq-n-k}\real_\t^-X_i&&\text{constr.\ of $\real_\t^-$;}\\
&\cong\tau^{\geq-n}\prod_I\real_\t^-X_i&&\text{Lem.\ref{describe_truncated_product}}.
\end{align*}
Similarly, for coproducts, we have that
\begin{align*}
\tau^{\geq-n}\real_\t^-\coprod_IX_i&\cong\real_\t^b\tau^{\geq-n}\coprod_IX_i&\quad&\text{constr.\ of $\real_\t^-$;}\\
&\cong\real_\t^b\tau^{\geq-n}\coprod_I\tau^{\geq-n}X_i&&\text{Lem.\ref{describe_truncated_product};}\\
&\cong\tau^{\geq-n}\coprod_I\real_\t^b\tau^{\geq-n}X_i&&\text{$\real_\t^b$ comm. with $\coprod$;}\\
&\cong\tau^{\geq-n}\coprod_I\real_\t^-X_i&&\text{constr.\ of $\real_\t^-$ and Lem.\ref{describe_truncated_product}.}
\end{align*}
%Notice that the fact that $\t$ is $k$-cosmashing implies that $\A$ is (Ab.$4^*$)-$k$, in particular, $\A$ is a $\Dia$-bicomplete (Ab.$4^*$)-$k$ Abelian category with enough injectives. By Proposition \ref{prop_tow_tel_approx}, this implies that the functor $\tow\colon \Ch(\A)\to \Tow(\A)$ induces a fully faithful functor $\sfD(\A)\to \Ho(\Tow(\A))$, that restricts to a fully faithful functor $\sfD^-(\A)\to \Ho(\Tow^-(\A))$. 
%
We have already mentioned that $\real_\t^-$ is an exact morphism triangulated pre-derivators and that $\real_\t^-\restriction_{\Der^b_\A}\cong \real_\t^b$ is an equivalence $\Der_\A^b\to \D^b$. Furthermore, by the above discussion, $\real_\t^-$ commutes with products and coproducts of uniformly right-bounded families of objects. We can now prove that $\real_\t^-$ is fully faithful: given $X$ and $Y\in \Der^-(\A)$, there is a triangle
\[
Y\to \prod_\N Y_n\to\prod_\N Y_n\to \Sigma Y,
\] 
where $\{Y_n\}_{\N}\subseteq \Der^b(\A)$ and this family is uniformly right-bounded (to find this sequence use the equivalence $\tow:\Der(\A)\leftrightarrows \widehat\Der(\A):\holim_{\N^{\op}}$). Then, there is a commutative diagram with exact columns
\[
\xymatrix@R=15pt{
\vdots\ar[d]&\vdots\ar[d]\\
\prod_\N\Der(\A)(\Sigma X, Y_n)\ar[r]\ar[d]&\prod_\N\D(\bbone)(\Sigma \real_\t^-X, \real_\t^bY_n)\ar[d]\\
\prod_\N\Der(\A)(\Sigma X, Y_n)\ar[r]\ar[d]&\prod_\N\D(\bbone)(\Sigma \real_\t^-X, \real_\t^bY_n)\ar[d]\\
\Der(\A)(X,Y)\ar[r]\ar[d]&\D(\bbone)(\real_\t^-X,\real_\t^-Y)\ar[d]\\
\prod_\N\Der(\A)(X, Y_n)\ar[r]\ar[d]&\prod_\N\D(\bbone)(\real_\t^-X, \real_\t^bY_n)\ar[d]\\
\prod_\N\Der(\A)(X, Y_n)\ar[r]\ar[d]&\prod_\N\D(\bbone)(\real_\t^-X, \real_\t^bY_n)\ar[d]\\
\vdots&\vdots}
\]
So we are reduced to verify that $\Der(\A)(\Sigma X, Y_n)\to\D(\bbone)(\Sigma \real_\t^-X, \real_\t^bY_n)$ and $\Der(\A)(X, Y_n)\to\D(\bbone)( \real_\t^-X, \real_\t^bY_n)$ are isomorphisms. But in fact, this is true because, if $Y_n\in\Der^{\geq-n}(\A)$, then we have the following natural isomorphisms
\begin{itemize}
\item $\Der(\A)(\Sigma X, Y_n)\cong \Der(\A)(\Sigma \tau^{\geq-n+1}X, Y_n)$;
\item $\Der(\A)( X, Y_n)\cong \Der(\A)( \tau^{\geq-n}X, Y_n)$;
\item $\D(\bbone)(\Sigma \real_\t^-X, \real_\t^bY_n)\cong \D(\bbone)(\Sigma \real_\t^b\tau^{\geq-n+1}X, \real_\t^bY_n)$;
\item $\D(\bbone)(\real_\t^-X, \real_\t^bY_n)\cong \D(\bbone)(\real_\t^b\tau^{\geq-n}X, \real_\t^bY_n)$.
\end{itemize} 
Hence, the desired isomorphisms follow by the exactness and full faithfulness of $\real_\t^b$.
\end{proof}

Of course, a dual statement gives conditions on $\real_\t^+\colon \sfD^+_\A\to \D$ to be fully faithful.

\newpage
\section{The unbounded realization functor}\label{unbounded_section}

In this section we introduce the notion of two-sided $\t$-complete derivator and we use this notion  to extend the half-bounded realization functors $\real_\t^-\colon \D_\A^-\to \widehat\D$ and $\real_\t^+\colon \D_\A^+\to \widecheck\D$, to the unbounded realization functor
\[
\real_\t\colon \sfD_\A\to \D^\diamond,
\]
where $\D^\diamond$ is the two-sided $\t$-completion of $\D$. When $\t$ is smashing, $k$-cosmashing for some $k\in\N$, non-degenerate (left and right) and (co)effa\c cable, then we can show that  $\real_\t$ commutes with products and coproducts, and it is fully faithful. 

\medskip\noindent
\textbf{Setting for Section \ref{unbounded_section}}.
Fix throughout this section a category of diagrams $\Dia$ such that $\N\in\Dia$, a strong and stable derivator $\D\colon \Dia^\op\to \CAT$, a $t$-structure $\t=(\D^{\leq 0},\D^{\geq0})$ on $\D$, and let $\A = \D^\heartsuit(\bbone)$ be its heart.

\subsection{Bicomplete $t$-structures}
Let us start with the following definition that, somehow, puts together the notion of left and right completion of Definition \ref{def_co_completion}:

\begin{defn}\label{def_bi_completion}
We define the {\bf $\t$-bicompletion} $\D^\diamond$ of $\D$ to be the sub-prederivator 
\[
\D^\diamond\colon \Dia^{\op}\to \CAT
\]
of $\D^{ \N^{\op}\times \N}$ such that, for any $I\in \Cat$, $\D^\diamond(I)\subseteq \D^{ \N^{\op}\times \N}(I)$ is the full subcategory of those objects $\mathscr X\in\D^\diamond(I)$ such that, for all $n,\,m\in\N$,
\begin{enumerate}
\item[\rm (1)] $\mathscr X_{n,m}\in \D^{\geq -n}(I)$;
\item[\rm (2)] $\mathscr X_{n,m}\in \D^{\leq m}(I)$;
\item[\rm (3)] the  map $\mathscr X_{n,m}\to \mathscr X_{n-1,m}$ induces an iso $(\mathscr X_{n,m})^{\geq -n+1}\to \mathscr X_{n-1,m}$;
\item[\rm (4)] the  map $\mathscr X_{n,m}\to \mathscr X_{n,m+1}$ induces an iso $\mathscr X_{n,m}\to (\mathscr X_{n,m+1})^{\leq m}$.
\end{enumerate}
\end{defn}
Suppose now that $\t=(\D^{\leq0},\D^{\geq0})$ is a countably $k$-smashing and countably $k$-cosmashing $t$-structure. Hence, both the left $\t$-completion $\widehat \D$ and the right $\t$-completion $\widecheck \D$ are stable derivators, each endowed with a natural $t$-structure, which is again countably $k$-smashing and countably $k$-cosmashing. Hence, it actually makes sense to consider the stable derivator $\widecheck{\widehat\D}$, the right $\widehat\t$-completion of the left $\t$-completion, and $\widehat{\widecheck\D}$, the left $\widecheck \t$-completion of the right $\t$-completion. It is not difficult to check that $\widehat{\widecheck\D}\cong\D^\diamond\cong \widecheck{\widehat\D}$.

\begin{prop}\label{bicomplete_der}
If $\t=(\D^{\leq0},\D^{\geq0})$ is a countably $k$-smashing and countably $k$-cosmashing $t$-structure for some $k\in\N$, the following diagram commutes up to a natural transformation:
\[
\xymatrix{
&&\widehat\D\ar[drr]^{\tel_{\widehat\t}}\\
\D\ar[urr]^{\tow_\t}\ar[drr]_{\tel_\t}&&&&\D^{\diamond}\\
&&\widecheck \D\ar[urr]_{\tow_{\widecheck \t}}
}
\]
Furthermore, if $\t$ is non-degenerate, then the composition $\tow_{\widecheck \t}\tel_\t\cong\tel_{\widehat \t}\tow_\t$ induces an equivalence of derivators $\D\cong \D^\diamond$. 
\end{prop}
\begin{proof}
The proof boils down to the usual fact that left and right truncations with respect to a given $t$-structure commute with each other. For the equivalence $\D\cong \D^\diamond$ use Proposition \ref{complete_proposition} and its dual.
\end{proof}

\subsection{The (unbounded) filtered prederivator $\FDu$}

Let us start with the following definition:

\begin{defn}\label{half_filtered_der_cat}
Given $i\in \N$, let $(i,\N^\op,\Z)\colon \bbone\times \N^\op\times \Z\to \N\times \N^{\op}\times \Z$ be the inclusion such that $(0,n,m)\mapsto (i,n,m)$. Given $ \mathscr X\in \D^{\N\times \N^\op\times\Z}(I)$ we let $\mathscr X_{(i,\N^\op,\Z)}:=(i,\N^\op,\Z)^*\mathscr X\in \D^{\N^{\op}\times\Z}(I)$. We define the {\bf unbounded filtered prederivator} 
\[
\FDu\colon \Dia^{\op}\to \CAT,
\] 
as a full sub-prederivator of $\D^{\N\times \N^\op\times\Z}$ such that, given $I\in\Dia$, $\FDu(I)$ is the full subcategory of $\D^{\N\times\N^\op\times\Z}(I)$ spanned by those $\mathscr X$ such that, for any $i\in \N$, the object $\mathscr X_{(i,\N^\op,\Z)}$ belongs to $\FDp(I)$. 
\end{defn}

As for $\FD$, $\FDp$, and $\FDm$, in general, $\FDu$ is not a derivator but, for any $I\in \Dia$, $\FDu(I)$ is a full triangulated subcategory of $\D^{\N\times\N^\op\times\Z}(I)$. In particular, it makes sense to consider triangles and to speak about $t$-structures in $\FDu(I)$ for some $I\in \Dia$. 

\begin{cor}\label{lift_NNZ}
Given $I\in\Dia$, $\dia_{\N}\colon \D^{\N\times\N^\op\times \Z}(I)\to (\D^{\N^{\op}\times\Z}(I))^{\N}$ induces a functor
\[
\dia_{\N}\colon \FDu(I)\to (\FDp(I))^{\N}
\]
which is full and essentially surjective. Furthermore, given $\mathscr X,\,\mathscr Y \in \FDu(I)$, the diagram functor induces an isomorphism
\[
\FDu(I)(\mathscr X,\mathscr Y)\cong \FDp(I)^{\N}(\dia_{\N}\mathscr X,\dia_{\N}\mathscr Y)
\]
provided $\FDp(I)(\Sigma \mathscr X_{(i,\N^\op,\Z)},\mathscr Y_{(i+1,\N^\op,\Z)})=0$ for all $i\in\N$.
\end{cor}
\begin{proof}
Applying Lemma \ref{lem_epivalence_N_op} to the shifted stable derivator $\D^{\N^{\op}\times\Z}$, one can see that the functor $\dia_{\N}\colon \D^{\N\times\N^\op\times \Z}(I)=(\D^{\N^\op\times\Z})^{\N}(I)\to \D^{\N^\op\times\Z}(I)^{\N}$ is full and essentially surjective. Furthermore, the essentially image of the restriction of $\dia_{\N}$ to $\FDu(I)$ is, by construction, exactly $(\FDp(I))^{\N}$. This proves the first part of the statement.

To verify the second part, let $\mathscr X,\,\mathscr Y \in \FDu(I)\subseteq  \D^{\N\times{\N^\op}\times \Z}(I)$, then, again by Lemma \ref{lem_epivalence_N_op}, the diagram functor induces an isomorphism
\begin{align*}
\FDu(I)(\mathscr X,\mathscr Y)&=\D^{\N\times\N^\op\times \Z}(I)(\mathscr X,\mathscr Y)\\
&\cong \D^{\N^{\op}\times\Z}(I)^{\N}(\dia_{\N}\mathscr X,\dia_{\N}\mathscr Y)\\
&= (\FDp(I))^{\N}(\dia_{\N}\mathscr X,\dia_{\N}\mathscr Y)
\end{align*}
if $\FDp(I)(\Sigma \mathscr X_{(i,\N^\op,\Z)},\mathscr Y_{(i+1,\N^\op,\Z)})=\D^{\N^{\op}\times\Z}(I)(\Sigma \mathscr X_{(i,\N^{\op},\Z)},\mathscr Y_{(i+1,\N^{\op},\Z)})=0$ for all $i\in\N$. 
\end{proof}

\begin{prop}\label{heart_is_ChbNN}
Consider the following two classes of objects in $\FDu(\bbone)$:
\begin{align*}
\FDu^{\leq0}(\bbone)&=\{\mathscr X\in \FDu(\bbone):\mathscr X_{(i,\N^{\op},\Z)}\in \FDp^{\leq 0}(\bbone),\, \forall i\in \N\},\\
\FDu^{\geq0}(\bbone)&=\{\mathscr X\in \FDu(\bbone):\mathscr X_{(i,\N^{\op},\Z)}\in \FDp^{\geq 0}(\bbone),\, \forall i\in\N\}.
\end{align*}
Then $\t_{B}:=(\FDu^{\leq0}(\bbone),\FDu^{\geq0}(\bbone))$ is a $t$-structure on $\FDu(\bbone)$ whose heart is equivalent to $\Ch^b(\A)^{\N\times\N^\op}$.
\end{prop}
\begin{proof}
The proof is completely analogous to that of Proposition \ref{heart_is_Chb}, just using Corollary \ref{lift_NNZ} instead of Corollary \ref{lift_NZ}.
\end{proof}

\begin{defn}
The $t$-structure $\t_B:=(\FDu^{\leq0},\FDu^{\geq0})$ on $\FDu$ described in the above proposition is said to be the {\bf unbounded Beilinson $t$-structure} induced by $\t$. 
\end{defn}

\subsection{The unbounded realization functor}

\begin{thm}\label{heart_is_ch}
Let $\t_B=(\FDu^{\leq0},\FDu^{\geq0})$ be the Beilinson $t$-structure in $\FDu$ induced by $\t$. Define a morphism of prederivators
\[
\xymatrix{
\Tot_\t\colon\Ch_\A\ar[rr]^(.35){\widehat\tel\,\tow}&& (\Ch^b_\A)^{\N\times\N^{\op}}\subseteq \FDu\subseteq \D^{\N\times \N^\op\times \Z}\ar[rr]^(.65){\hocolim_{\Z}}&& \D^{\N\times\N^{\op}}.
}
\]
Then the following statements hold true:
\begin{enumerate}
\item $\Tot_\t$ takes values in $ \D^\diamond$;
\item given a quasi-iso $\phi$ in $\Ch(\A^{I})$ for some $I\in\Cat$,  $\Tot_\t\phi$ is an iso.
\end{enumerate} 
\end{thm}
\begin{proof}
(1) Let $I\in \Cat$ and $X\in \Ch(\A^I)$. By~\cite[Prop.\,2.6]{Moritz}, for any $(n,m)\in\N\times \N^\op$,
\begin{align*}
(\hocolim_{\Z}\widehat\tel\,\tow X)_{(n,m)}&\cong \hocolim_{\Z}(\widehat\tel\,\tow X)_{(n,m)} \\
&\cong\hocolim_{\Z}(\tau^{\geq-m}\tau^{\leq n}X),
\end{align*}
That is, the underlying diagram of $\Tot_\t X$ is given by
\[
\xymatrix@C=15pt@R=20pt{
\Tot_\t^b\tau^{\geq0}\tau^{\leq0}X\ar[r]& \Tot_\t^b\tau^{\geq0}\tau^{\leq1}X\ar[r]& \Tot_\t^b\tau^{\geq0}\tau^{\leq 2}X\ar[r]&\cdots\\
\Tot_\t^b\tau^{\geq-1}\tau^{\leq0}X\ar[u]\ar[r]& \Tot_\t^b\tau^{\geq-1}\tau^{\leq1}X\ar[r]\ar[u]& \Tot_\t^b\tau^{\geq-1}\tau^{\leq 2}X\ar[r]\ar[u]&\cdots\\
\Tot_\t^b\tau^{\geq-2}\tau^{\leq0}X\ar[u]\ar[r]& \Tot_\t^b\tau^{\geq-2}\tau^{\leq1}X\ar[r]\ar[u]& \Tot_\t^b\tau^{\geq-2}\tau^{\leq 2}X\ar[r]\ar[u]&\cdots\\
\vdots\ar[u]& \vdots\ar[u]& \vdots\ar[u]&
}
\]
so that $\Tot_\t X\in \D^{\diamond}(I)$ by the already verified properties of $\Tot_\t^b$. 

(2) Let $I\in\Cat$ and $\phi\colon X\to Y$ be a quasi-isomorphism in $\Ch(\A^I)$. 
We have to verify that $(\hocolim_{\Z}\widehat\tel\,\tow \phi)_{(n,m)}$ is an isomorphism for any $(n,m)\in\N\times \N^\op$. By~\cite[Proposition 2.6]{Moritz}, 
\[
(\hocolim_{\Z}\widehat\tel\,\tow \phi)_{(n,m)}\cong \hocolim_{\Z}(\widehat\tel\,\tow \phi)_{(n,m)}\cong\hocolim_{\Z}\tau^{\geq-m}\tau^{\leq n}\phi .
\]
To conclude recall that we know by Proposition \ref{real_is_tot_prop} that 
\[
\hocolim_{\Z}\tau^{\geq-m}\tau^{\leq n}\phi=\Tot_\t^b(\tau^{\geq-m}\tau^{\leq n}\phi),
\]
which is an iso since $\tau^{\geq-m}\tau^{\leq n}\phi$ is a quasi-isomorphism.
\end{proof}

Suppose now that the Abelian category $\A$ has enough injectives and it is (Ab.$4^*$)-$k$ for some $k\in\N$ or, alternatively, that the Abelian category $\A$ has enough projectives and it is (Ab.$4$)-$k$. In both cases, we have seen in Subsection \ref{complete_subsection_tower_tel_complexes} that it makes sense to consider the prederivator $\sfD_\A$ where, for any $I\in \Dia$, 
\[
\sfD_\A(I)=\sfD(\A^{I})=\Ch(\A^{I})[\W_I^{-1}]
\]
where $\W_I$ is the class of quasi-isomorphisms in $\Ch(\A^{I})$.
Hence, a consequence of Theorem \ref{heart_is_ch} is that the corestriction $\Tot_\t\colon \Ch_\A\to \D^{\diamond}$ factors uniquely through the quotient $\Ch_\A\to \sfD_\A$, giving rise to a morphism of prederivators 
\[
\real_\t\colon \sfD_\A\to \D^{\diamond}.
\]

\subsection{Conditions on $\real_\t$ to be fully faithful}
We have seen in the the previous subsection that, under suitable hypotheses on $\A$, we can construct a morphism of prederivators $\real_\t\colon \sfD_\A\to  \D^{\diamond}$ so, composing  with $\holim_{\N^\op}\hocolim_\N\colon  \D^{\diamond}\to \widehat\D\to \D$ we get a morphism of prederivators $\sfD_\A\to \D$. In the following proposition we give some general conditions to make this morphism fully faithful:

\begin{thm}\label{unbounded_ff_real}
Suppose that $\t$ satisfies the following conditions:
\begin{enumerate}
\item $\t$ is (co)effa\c cable;
\item $\t$ is non-degenerate (right and left);
\item $\t$ is $k$-cosmashing for some $k\in\N$;
\item $\t$ is ($0$-)smashing;
\item the heart $\A$ of $\t$ has enough injectives.
\end{enumerate} 
Then, $\real_\t\colon \sfD_\A\to \D$ commutes with products and coproducts, and it is fully faithful. Dually, $\real_\t\colon \sfD_\A\to \D$ commutes with products and coproducts, and it is fully faithful provided the following hypotheses hold:
\begin{enumerate}
\item[(1')] $\t$ is (co)effa\c cable;
\item[(2')] $\t$ is non-degenerate (right and left);
\item[(3')] $\t$ is $k$-smashing for some $k\in\N$;
\item[(4')] $\t$ is ($0$-)cosmashing;
\item[(5')] the heart $\A$ of $\t$ has enough projectives.
\end{enumerate} 
%In particular, $\sfD_\A$ is itself a derivator. Furthermore, the above equivalence restricts to equivalences
%\[
%\real^+_\t\colon \sfD_\A^+\to \D^+
%\]
%and
%\[
%\real^b_\t\colon \sfD_\A^b\to \D^b.
%\]
\end{thm}
\begin{proof}
We give an argument just for the first half of the statement since the proof of the rest is completely dual. Consider the functor  $\real_\t\colon \sfD(\A)\to\D(\bbone)$ and remember that, essentially by construction, for all $n\in\N$, there is a commutative diagram as follows:
\[
\xymatrix{
\sfD(\A)\ar[r]^{\real_\t}\ar[d]_{\tau^{\leq n}}&\D(\bbone)\ar[d]^{\tau^{\leq n}}\\
\sfD^-(\A)\ar[r]|\cong^{\real^-_\t}&\D^-(\bbone)\\
}
\]
where the truncation on the left is taken with respect to $\t_\A$, while the truncation on the right is done with respect to $\t$.
Suppose that we have a family $\{X_i\}_I\subseteq \sfD(\A)$. We are going to prove that
\[
\real_\t\prod_IX_i\cong \prod_I\real_\t X_i\quad \text{and}\quad \real_\t \coprod_IX_i\cong \coprod_I\real_\t X_i.
\]
We start with products. Since we already know that $\tel:\D\leftrightarrows \widecheck\D :\hocolim_{\N}$ is an equivalence, it is enough to show that $\tau^{\leq n}\real_\t\prod_IX_i\cong \tau^{\leq n}\prod_I\real_\t X_i$, for all $n$:
\begin{align*}
\tau^{\leq n}\real_\t\prod_IX_i&\cong\real_\t^-\tau^{\leq n}\prod_IX_i&\qquad&\text{constr.\ of $\real_\t^-$;}\\
&\cong\real_\t^-\tau^{\leq n}\prod_I\tau^{\leq n}X_i&&\text{Lem.\ref{describe_truncated_product};}\\
&\cong\tau^{\leq n}\prod_I\real_\t^-\tau^{\leq n}X_i&&\text{previous section;}\\
&\cong\tau^{\leq n}\prod_I\tau^{\leq n}\real_\t X_i&&\text{constr.\ of $\real_\t$;}\\
&\cong\tau^{\leq n}\prod_I\real_\t X_i&&\text{Lem.\ref{describe_truncated_product}.}
\end{align*}
Similarly, for coproducts, we have that
\begin{align*}
\tau^{\leq n}\real_\t\coprod_IX_i&\cong\real_\t^-\tau^{\leq n}\coprod_IX_i&\quad&\text{constr.\ of $\real_\t^-$;}\\
&\cong\real_\t^-\tau^{\leq n}\coprod_I\tau^{\leq n+k}X_i&&\text{Lem.\ref{describe_truncated_product};}\\
&\cong\tau^{\leq n}\real_\t^-\coprod_I\tau^{\leq n+k}X_i&&\text{$\real_\t^-$ is $t$-exact;}\\
&\cong\tau^{\leq n}\coprod_I\real_\t^-\tau^{\leq n+k}X_i&&\text{previous section;}\\
&\cong\tau^{\leq n}\coprod_I\real_\t X_i&&\text{constr.\ of $\real_\t^-$ and Lem.\ref{describe_truncated_product}.}
\end{align*}
One concludes that, under our hypotheses, $\real_\t$ is fully faithful using a technique similar to (the dual of) the argument in the proof of Theorem \ref{prop_fully_minus}.
\end{proof}

\newpage
\section{Applications to co/tilting equivalences}

In this final section we list some applications of the general constructions developed in the rest of the paper to co/tilting equivalence. In this way we recover some important results from~\cite{rickard1991derived,stovicek2014derived,NSZ,Jorge_e_Chrisostomos}. Let us start recalling some definitions from~\cite{NSZ,Jorge_e_Chrisostomos}:
\begin{defn}
A set $\S$ in a triangulated category $\mathcal D$ is said to be
\begin{itemize}
\item {\bf silting} if $(\S^{\perp_{>0}},\S^{\perp_{<0}})$ is a $t$-structure in $\mathcal D$ and $\S\subseteq  \S^{\perp_{>0}}$;
\item {\bf cosilting} if $({}^{\perp_{<0}}\S,{}^{\perp_{>0}}\S)$ is a $t$-structure in $\mathcal D$ and $\S\subseteq {}^{\perp_{>0}}\S$;
\item {\bf tilting} if it is silting and $\Add(\S)\subseteq \S^{\perp_{\neq0}}$;
\item {\bf cotilting} if it is cosilting and $\Prod(\S)\subseteq {}^{\perp_{\neq0}}\S$;
\item {\bf classical tilting} if it is tilting and any $S\in \S$ is compact.
%\item {\bf bounded tilting} if it is a tilting object and there exist $n\in\N$ and a classical tilting object $T$ such that 
%$T^{\perp_{<-n}}\subseteq X^{\perp_{<0}}\subseteq T^{\perp_{<n}}$.
%\item {\bf bounded cotilting} if it is a cotilting object and there exist $n\in\N$ and a classical tilting object $T$ such that 
%$T^{\perp_{<-n}}\subseteq {}^{\perp_{>0}}X\subseteq T^{\perp_{<n}}$.
\end{itemize}
\end{defn}

We collect in the following lemma some basic properties of co/tilting objects from~\cite{NSZ,Jorge_e_Chrisostomos}:
\begin{lem}\label{lem_for_ff_co_tilting_from_lit}
Suppose that $\mathcal D$ has products and coproducts, and let $\S\subseteq \mathcal D$ be a set. Then,
\begin{enumerate}
\item if $\S$ is tilting,  the associated $t$-structure is (co)effa\c{c}able, cosmashing, non-degenerate and $\S$ is a set of generators of $\mathcal D$. Furthermore,  products are exact in the heart and $\S$ is a set of projective generators of the heart;
\item if $\S$ is cotilting,  the associated $t$-structure is (co)effa\c{c}able, smashing, non-degenerate and $\S$ is a set of cogenerators of $\mathcal D$. Furthermore,  coproducts are exact in the heart and $\S$ is a set of injective cogenerators of the heart;
\item if $\S$ is a classical tilting set, then the associated $t$-structure is smashing and its heart is equivalent to a category of modules over a ring with enough idempotents. If, furthermore, $\S=\{S\}$ consists of a single object, then the heart is equivalent to $\Mod(\End_{\mathcal D}(S))$. 
\end{enumerate}
\end{lem}
\begin{proof}
Parts (1) and (2) follow by~\cite[Prop.\,4.3, Lem.\,4.5 and Prop.\,5.1]{Jorge_e_Chrisostomos}. Part (3) follows by~\cite[Coro.\,4.3]{NSZ}.
%
%
%It follows by~\cite[Cor.\,4.7]{Jorge_e_Chrisostomos} that, for a classical tilting $X$, the heart of the associated $t$-structure is equivalent to the category $\Mod(\End_{\mathcal D}(X))$. Consider now a family $\{D_i\}_{i}$ in $X^{\perp_{<0}}$, then $\mathcal D(\Sigma^nX,\coprod_i D_i)=\coprod_i\mathcal D(\Sigma^nX, D_i)=0$ for all $n\in\N_{>0}$, by compactness, hence $\coprod_i D_i\in X^{\perp_{<0}}$. 
\end{proof}

We will apply the following consequence of Brown representability to prove the essential surjectivity of $\real_\t$ in several cases:

\begin{lem}\label{brown_implies_ess_su}
Let $\mathcal D$ be a triangulated category with coproducts and with a set of compact generators $\S$. Then $\mathcal D$ satisfies the principle of infinite d\'evissage with respect to $\S$.
% Dually, let $\mathcal D'$ be a full subcategory of $\mathcal D$ such that 
%\begin{enumerate}
%\item $\mathcal D'$ has coproducts (not necessarily the same as in $\mathcal D$);
%\item the inclusion $\mathcal D'\to \mathcal D$ preserves products;
%\item $\S$ is contained in $\mathcal D'$.
%\end{enumerate}
%Then $\mathcal D'=\mathcal D$.
\end{lem}
\begin{proof}
Consider the inclusion $L\colon \mathrm{Loc}(\S)\to \mathcal D$ and notice that this has a right adjoint $R\colon \mathcal D\to \mathrm{Loc}(\S)$ by Brown representability. Let $X\in \mathcal D$, then there is a triangle 
\[
LRX\to X\to C\to \Sigma LRX
\]
with $LRX\in \mathrm{Loc}(\S)$ and $C\in \mathrm{Loc}(\S)^{\perp}$ (see,~\cite[Prop.\,5.3.1]{zbMATH05831327}). Since $\mathrm{Loc}(\S)^{\perp}\subseteq \left(\bigcup_{i\in \Z}\Sigma^i\S\right)^{\perp}=0$, that is, $\S$ is a set of generators, then $X\cong LRX\in \mathrm{Loc}(\S)$. 
\end{proof}

As a first application of our machinery of realization functors we can prove the following expected extension to the equivalence induced by a classical tilting object  to the context of stable derivators: 

\begin{thm}\label{compact_case}
The following are equivalent for a strong and stable derivator $\D\colon \Cat^{\op}\to \CAT$:
\begin{enumerate}
\item there is a classical tilting set $\S\subseteq \D(\bbone)$ (resp., a classical tilting object $T\in \D(\bbone)$);
\item there is a small preadditive category $\C$ and an equivalence of derivators $F\colon\sfD_{\Mod(\C)}\to \D$ (resp., a ring $R$ and an equivalence of derivators $F\colon\sfD_{\Mod(R)}\to \D$).
\end{enumerate} 
\end{thm}
\begin{proof}
If we assume (2), then the representable $\C$-modules form a classical tilting set in $\sfD_{\Mod(\C)}(\bbone)$, so $\S:=\{F(\C(-,c)):c\in\C\}$ is a classical tilting set in $\D(\bbone)$. 

On the other hand, if $\S$ is a classical tilting set in $\D(\bbone)$, then let $\t_\S=(\D_\S^{\leq 0},\D_\S^{\geq0})$ be the associated $t$-structure and consider 
\[
F:=\real_\S\colon \sfD_{\Mod(\C_\S)}\to \D,
\] 
where $\C_\S$ is the small preadditive category given by Lemma \ref{lem_for_ff_co_tilting_from_lit} (3). By Theorem \ref{unbounded_ff_real} and Lemma \ref{lem_for_ff_co_tilting_from_lit}, $F$ is fully faithful. Furthermore the image of $F$ is a localizing subcategory containing the set of compact generators  $\S$, hence, by Lemma \ref{brown_implies_ess_su}, $F$ is also essentially surjective.
\end{proof}

Before stating our main applications, let us recall the following definition:

\begin{defn}\cite{fiorot2016classification}
A pair of $t$-structures $(\t_1=(\mathcal D^{\leq0}_1,\mathcal D^{\geq0}_1),\t_2=(\mathcal D^{\leq0}_2,\mathcal D^{\geq0}_2))$ is said to have {\bf shift} $k\in\Z$ and
{\bf gap} $n\in\N$ if $k$ is the maximal number such that $\mathcal D^{\leq k}_2\subseteq \mathcal D^{\leq0}_1$ and $n$ is the minimal number such that $\mathcal D^{\leq -n}_1\subseteq \mathcal D^{\leq k}_2$. In this case we say that the pair $(\t_1,\t_2)$ is of {\bf type} $(n,k)$.

We say that $\t_1$ and $\t_2$ have {\bf finite distance} if there exist $k\in\Z$ and $n\in\N$ for which the pair $(\t_1,\t_2)$ is of {type} $(n,k)$.
\end{defn}

\begin{lem}\label{finite_distance}
Let $(\t_1,\t_2)$ be a pair of $t$-structures  of {type} $(n,k)$ on a triangulated category $\mathcal D$. If $\t_2$ is $h$-smashing (resp., $h$-cosmashing) for a given $h\in\N$, then $\t_1$ is $(n+h)$-smashing (resp., $(n+h)$-cosmashing).
\end{lem}
\begin{proof}
Let $\{D_i\}_i$ be a family of objects in $\mathcal D_1^{\geq0}\subseteq \mathcal D_2^{\geq k}$; if $\t_2$ is $h$-smashing, $\coprod_i D_i\in \mathcal D_2^{\geq k-h}\subseteq \mathcal D_1^{\geq -n-h}$. Similarly, consider a family $\{E_i\}_i$ of objects in $\mathcal D_1^{\leq0}\subseteq \mathcal D_2^{\leq k+n}$; if $\t_2$ is $h$-cosmashing, $\prod_iD_i\in\mathcal D_2^{\leq k+n+h}\subseteq \mathcal D_1^{\leq n+h}$.
\end{proof}

Notice that, by Lemmas \ref{brown_implies_ess_su} and \ref{finite_distance}, when a tilting $t$-structure is at finite distance from a classical tilting $t$-structure, then it is cosmashing and $k$-smashing for some $k\in\N$. Similarly, when a cotilting $t$-structure is at finite distance from a classical tilting $t$-structure, then it is smashing and $k$-cosmashing for some $k\in\N$. This easy observation allows us to apply our machinery of realization functors to such $t$-structures:

\begin{thm}\label{general_tilt_thm}
Let $\D\colon \Cat^{\op}\to \CAT$ be a strong, stable derivator, and suppose that we have a tilting set $\T$ and a classical tilting set $\S$ in $\D(\bbone)$. Let $\t_\T=(\D_\T^{\leq0},\D_\T^{\geq0})$ and $\t_\S=(\D_\S^{\leq0},\D_\S^{\geq0})$ be the induced $t$-structures and suppose that they have finite distance. Letting $\A_\T:=\D_\T^{\leq0}(\bbone)\cap\D_\T^{\geq0}(\bbone)$, there is an equivalence of prederivators
\[
\real_{\t_\T}\colon \sfD_{\A_\T}\to \D
\]
that restricts to equivalences $\real_{\t_\T}^-\colon \sfD_{\A_\T}^-\to \D^{-}_{\t_\T}$, $\real_{\t_\T}^b\colon \sfD_{\A_\T}^b\to \D^{b}_{\t_\T}$, and $y(\A_\T)\to \D^{\heartsuit}$. In particular, $\sfD_{\A_\T}$ is a (strong and stable) derivator.
\end{thm}
\begin{proof}
By Lemma \ref{lem_for_ff_co_tilting_from_lit}, $(\D_\T^{\leq0}(\bbone),\D_\T^{\geq0}(\bbone))$ is (co)effa\c{c}able, cosmashing, non-degenerate and it is $n$-smashing, for some $n\in\N$, by Lemma \ref{finite_distance}. On the other hand,  $\A_\T$ has exact products, it has a projective generator $X$ and it is (Ab.$4$)-$n$. Let us verify that $\t_I:=(\D_\T^{\leq0}(I),\D_\T^{\geq0}(I))$ and $\A_\T^I=\D_\T^{\leq0}(I)\cap\D_\T^{\geq0}(I)$ have the same properties for all $I\in\Dia$. In fact it is easy to verify that $\t_I$ is cosmashing, non-degenerate, $n$-smashing, and that $\A_\T^I$ has exact products, it has a set of projective generators $\{i_!T:i\in I,\, T\in\T\}$ and it is (Ab.$4$)-$n$. It remains to show that $\t_I$ is effa\c{c}able. Let $Y_1,\, Y_2\in \A_\T^I$, $n>0$ and  $\phi\colon Y_1\to \Sigma^nY_2$ in $\D(I)$. Since $\{i_!T:i\in I,\, T\in\T\}$  is a family of generators, there is an epimorphism $\coprod_{i\in I,\, T\in \T}i_!T^{(A_{i,T})}\to Y_1$. Then, 
\[
\D(I)\left(\coprod_{i\in I,\, T\in \T}i_!T^{(A_{i,T})},\Sigma^nY_2\right)=\prod_{i\in I,\, T\in \T}\prod_{A_{i,T}}\D(I)(T,\Sigma^ni^*Y_2)=0
\]
since $Y_2\in \T^{\perp_{>0}}$. Hence, by Theorem \ref{unbounded_ff_real}, 
\[
\real_{\t_\T}\colon \sfD_{\A_\T}\to \D
\]
is exact, fully faithful and it commutes with products and coproducts. Therefore, the essential image of  $\real_{\t_\T}$ is a localizing category, furthermore $\S\subseteq \S^{\perp_{\neq0}}\subseteq  \D^b_{\t_\T}= \real_{\t_\T}(\sfD_{\A_\T}^b)$, thus by Lemma \ref{brown_implies_ess_su}, $\real_{\t_\T}(\sfD_{\A_\T}(I))=\D(I)$ for all $I$, so $\real_{\t_\T}$ is also essentially surjective.
\end{proof}

As a corollary we obtain one of the main results of~\cite{NSZ}:

\begin{cor}\label{coro_NSZ_tilt_alg}\cite[Thm.\,7.4]{NSZ}
Let $K$ be a commutative ring, let $\mathcal D$ be a compactly generated algebraic triangulated $K$-category, let $\T$ be a bounded tilting set of $\mathcal D$ and let $\A_\T$ be the heart of the associated $t$-structure. The inclusion $\A_\T\to \mathcal D$ extends to a triangulated equivalence $\Psi\colon \sfD(\A_\T)\to \mathcal D$.
\end{cor}
\begin{proof}
Compactly generated algebraic triangulated $K$-categories are triangle equivalent to the derived categories of small dg-categories, and such derived categories are the homotopy categories of a suitable (combinatorial) model category. We refer to~\cite[Sections 2.2 and 2.3]{stovivcek2016compactly} for a more detailed discussion. In particular, one can show that there is a strong stable derivator $\D\colon \Cat^{\op}\to \CAT$ such that $\D(\bbone)\cong \mathcal D$. Furthermore, using that $\mathcal D$ is equivalent to the derived category of a small dg-category, one can consider the set $\S$ of representable modules as a classical tilting set in $\mathcal D$. By the results in~\cite[Section 5]{NSZ}, the fact that $\S$ is bounded implies that the $t$-structures $\t_\T$ and $\t_\S$ are at finite distance, so Theorem \ref{general_tilt_thm} applies to give an equivalence of derivators $\real_{\t_\T}\colon \sfD_{\A_\T}\to \D$ that, evaluated at $\bbone$, gives an equivalence $ \sfD({\A_\T})\to \mathcal D$.
\end{proof}

Let us now pass to the dual setting of cotilting $t$-structures:

\begin{thm}\label{general_co_tilt_thm}
Let $\D\colon \Cat^{\op}\to \CAT$ be a strong, stable derivator, and suppose that we have a cotilting set $\mathbf C$ and a classical tilting set $\T$ in $\D(\bbone)$. Let $\t_{\mathbf C}=(\D_{\mathbf C}^{\leq0},\D_{\mathbf C}^{\geq0})$ and $\t_\T=(\D_\T^{\leq0},\D_\T^{\geq0})$ be the induced $t$-structures and suppose that they have finite distance. Letting $\A_{\mathbf C}:=\D_{\mathbf C}^{\leq0}(\bbone)\cap\D_{\mathbf C}^{\geq0}(\bbone)$, there is an equivalence of prederivators
\[
\real_{\t_{\mathbf C}}\colon \sfD_{\A_{\mathbf C}}\to \D
\]
that restricts to equivalences $\real_{\t_{\mathbf C}}^+\colon \sfD_{\A_{\mathbf C}}^+\to \D^{+}_\t$, $\real_{\t_{\mathbf C}}^b\colon \sfD_{\A_{\mathbf C}}^b\to \D^{b}_{\t_{\mathbf C}}$, and $y(\A_{\mathbf C})\to \D^{\heartsuit}$. In particular, $\sfD_{\A_{\mathbf C}}$ is a (strong and stable) derivator.
\end{thm}
\begin{proof}
The proof that $\real_{\t_{\mathbf C}}\colon \sfD_{\A_{\mathbf C}}\to \D$ is exact, fully faithful and that it commutes with products and coproducts is completely dual to that given for Theorem \ref{general_tilt_thm}. It remains to verify that $\real_{\t_{\mathbf C}}$ is essentially surjective. By Theorem \ref{compact_case}, there is a small preadditive category $\C_\T$ and an equivalence of derivators
\[
\real_{\t_\T}\colon \sfD_{\Mod(\C_\T)}\to \D, 
\]
so we can identify these two derivators and consider $\real_{\t_{\mathbf C}}$ as a morphism $\sfD_{\A_{\mathbf C}}\to \sfD_{\Mod(\C_\T)}$. Notice that the image of this morphism is a colocalizing sub-prederivator containing $\sfD^b_{\Mod(\C_\T)}$ so, by Lemma \ref{devissage_of_der}, $\real_{\t_{\mathbf C}}$ is  essentially surjective.
\end{proof}

The setting of the above theorem may seem a bit unnatural since we are asking to a cotilting $t$-structure to be ``close" to a tilting $t$-structure. To show that this is not so artificial, let us recall the following definition from~\cite{hugel2001infinitely}.
 
\begin{defn}\label{big_cotilting}
Given a ring $R$, a right $R$-module $C$ is said to be {\bf big cotilting} if there exists $n\in\N$ such that the following three conditions are satisfied:
\begin{enumerate}
\item $C$ has injective dimension bounded by $n$;
\item $\Ext^j_R(C^I, C) = 0$ for every $j > 0$ and every set $I$;
\item there exists an exact sequence
$
0\to C_r \to \ldots\to C_1\to C_0\to Q \to 0
$
in $\Mod(R)$ such that $Q$ is an injective cogenerator of $\Mod(R)$, $r\in\N$ and $C_i \in \Prod(C)$ for all $i$.
\end{enumerate}
\end{defn}

By~\cite[Thm.\,4.5]{stovicek2014derived}, a big cotilting $R$-module $C$, when considered as an object of $\sfD(\Mod(R))$, is a cotilting object. Furthermore, letting $\t_C$ be the associated $t$-structure in $\sfD(\Mod(R))$, applying~\cite[Prop.\,4.17 and 4.14]{Jorge_e_Chrisostomos}, one gets that $\t_C$ is at finite distance from the natural $t$-structure $\t_R$ in $\sfD(\Mod(R))$. To conclude notice that $\t_R$ is induced by the classical tilting object $R$, so we are exactly in the setting of Theorem \ref{general_co_tilt_thm}. In this way we obtain a new proof for the following important result:

\begin{cor}\label{coro_stovicek2014derived}\cite[Thm.\,5.21]{stovicek2014derived}
Let $R$ be a ring, $C \in \Mod(R)$ a big cotilting module and $\mathcal G$ the corresponding tilted Abelian category.
Then the prederivators $\sfD_{\Mod(R)}$ and $\sfD_{\mathcal G}$ are equivalent and
both are strong stable derivators.
\end{cor}

As another consequence of our general theory of realization functors, one can obtain a ``Derived Morita Theory" for Abelian categories, completing~\cite[Thm.\,A]{Jorge_e_Chrisostomos}. Let us remark that the implication ``(3)$\Rightarrow$(1)", in the special case when $X\in \A\subseteq \sfD(\A)$, can be proved alternatively with the methods of~\cite{fiorot2017derived}.

\begin{thm}\label{morita_derived_derivators}
Let $\A$ be a Grothendieck category (resp., an (Ab.$4^*$)-$h$ Grothendieck category for some $h\in\N$), denote by $\t_\A=(\sfD_\A^{\leq0},\sfD_\A^{\geq0})$ the canonical $t$-structure on $\sfD_\A$, and let $\B$ be an Abelian category. The following are equivalent:
\begin{enumerate}
\item $\B$ has a projective generator (resp., an injective cogenerator) and there is an exact equivalence $\sfD(\B)\to \sfD(\A)$ that restricts to bounded derived categories;
\item $\B$ has a projective generator  (resp., an injective cogenerator) and there is an exact equivalence of prederivators $\sfD_\B\to \sfD_\A$ that restricts to an equivalence $\sfD^b_\B\to \sfD^b_\A$;
\item there is a tilting (resp., cotilting) object $X$  in $\sfD(\A)$,  whose heart $X^{\perp_{\neq 0}}$ (resp., $ {}^{\perp_{\neq0}}X$) is equivalent to $\B$ and such that the associated $t$-structure $\t_X=(\D_X^{\leq0},\D_X^{\geq0})$ has finite distance from $\t_\A$.
\end{enumerate} 
\end{thm}
\begin{proof}
The implication ``(1)$\Rightarrow$(3)'' is proved in~\cite[Thm.\,A]{Jorge_e_Chrisostomos}, while the implication ``(2)$\Rightarrow$(1)'' is trivial. Let us give an argument for the implication ``(3)$\Rightarrow$(2)'' in the tilting case, the cotilting case is dual. Indeed, by Lemma \ref{lem_for_ff_co_tilting_from_lit}, $\t_X:=(X^{\perp_{>0}},X^{\perp_{<0}})$ is (co)effa\c{c}able, cosmashing and non-degenerate. Since coproducts are exact in any Grothendieck category, $\t_\A$ is smashing and so, by Lemma \ref{finite_distance}, $\t_X$ is $n$-smashing for some $n\in\N$. Similarly, $\B$ has exact products, it has a projective generator $X$ and it is (Ab.$4$)-$n$. Exactly as in the proof of Theorem \ref{general_co_tilt_thm}, one can show that the lifting of $\t_X$ to any $\sfD(\A^I)$ has the same properties. Hence, the dual of Theorem \ref{unbounded_ff_real} applies to give an exact fully faithful morphism of prederivators
\[
\real_{\t_X}\colon \sfD_\B\to \sfD_\A
\]
that commutes with products and coproducts. Therefore, for any $I\in\Cat$, $\real_{\t_X}(\sfD(\B^I))$ is a localizing subcategory of $\sfD(\A^I)$ that contains $\A$ so, by Lemma \ref{devissage_of_der}, $\real_{\t_X}(\sfD(\B^I))=\sfD(\A^I)$, so $\real_{\t_X}$ is also essentially surjective.
\end{proof}

\bibliographystyle{alpha}
\bibliography{refs}

\medskip
\noindent\rule{5cm}{0.4pt}

\medskip
Simone Virili -- \texttt{s.virili@um.es} or \texttt{virili.simone@gmail.com}\\
{Departamento de Matem\'{a}ticas,
Universidad de Murcia,  Aptdo. 4021,
30100 Espinardo, Murcia,
SPAIN}

\end{document}